\documentclass[11 pt , letterpaper , oneside, openright ]{book}
\usepackage{amsmath, amsthm, amssymb,amscd}
\usepackage{cite}
\usepackage{float}
\usepackage{geometry}
\usepackage{graphicx}
\usepackage{setspace}
\usepackage{float}
\usepackage[all,cmtip]{xy}
\usepackage[english]{babel}
\usepackage{appendix}
\usepackage{MnSymbol}
\usepackage{hyperref}
\hypersetup{
    colorlinks,
    citecolor=black,
    filecolor=black,
    linkcolor=black,
    urlcolor=black
}
\usepackage{calligra}
\usepackage{pdfpages}
\usepackage{mathtools}
\newcommand{\slfrac}[2]{\left.#1\middle/#2\right.}
\newcommand{\proj}{\mathbb{P}}

\newcommand{\seq}{\subseteq}
\newcommand{\C}{\mathbb{C}}
\newcommand{\R}{\mathbb{R}}

\newcommand{\rank}{\text{rank }}

\newcommand{\nocontentsline}[3]{}
\newcommand{\tocless}[2]{\bgroup\let\addcontentsline=\nocontentsline#1{#2}\egroup}

\newtheorem{thm}{Theorem}[section]
\newtheorem*{thm-nl}{Theorem}
\newtheorem*{prop-nl}{Proposition}
\newtheorem{lem}[thm]{Lemma}

\newtheorem{cor}[thm]{Corollary}
\newtheorem*{cor-nl}{Corollary}
\newtheorem{conjecture}[thm]{Conjecture}
\newtheorem*{conjecture-nl}{Conjecture}
\newtheorem*{quest-nl}{Question}

\newtheorem*{quests-nl}{Questions}
\newtheorem{question}[thm]{Question}
\newtheorem{prop}[thm]{Proposition}

\theoremstyle{remark}
\newtheorem*{rem}{Remark}

\newtheorem{remark}[thm]{Remark}

\theoremstyle{definition}
\newtheorem{mydef}[thm]{Definition}
\newtheorem{eg}{Example}
\title{{Stable Maps and Singular Curves on K3 Surfaces}}
\author{Michael Kemeny}

\begin{document}
\newpage
\begin{titlepage}
\begin{center}
{\Huge 
\textsc{Stable maps and \\ singular curves \\ on K3 surfaces}
\vfill

\textsc{Dissertation}\\ }
\vspace{0.5cm}
\large 
zur\\ 
Erlangung des Doktorgrades (Dr.\ rer.\ nat.)\\
der\\
Mathematisch-Naturwissenschaftlichen Fakult\"at\\
der\\
Rheinischen Friedrich-Wilhelms-Universit\"at Bonn\\ 
\vfill
vorgelegt von\\
Michael Kemeny\\
aus\\
Sydney, Australien
\vfill

Bonn 2015
\vfill
\vfill
\end{center}
\newpage
\thispagestyle{empty}

\noindent
Angefertigt mit Genehmigung der Mathematisch-Naturwissenschaftlichen \\
Fakult\"at der Rheinischen Friedrich-Wilhelms-Universit\"at Bonn
\vfill

\noindent
1.\ Gutachter: Prof.\ Dr.\ Daniel Huybrechts \\
2.\ Gutachter: Prof.\ Dr.\ Gavril Farkas\\\
Tag der Promotion: 11 June 2015  \\\
Erscheinungsjahr:  2015

\end{titlepage}
\chapter*{Summary}
In this thesis we study singular curves on K3 surfaces. Let $\mathcal{B}_g$ denote the stack of polarised K3 surfaces of genus $g$ and set $p(g,k)=k^2(g-1)+1$. There is a stack
$ \mathcal{T}^n_{g,k} \to \mathcal{B}_g$ with fibre over the polarised surface $(X,L)$ parametrising all unramified morphisms $f: C \to X$, birational onto their image, with $C$ an integral smooth curve of genus $ p(g,k)-n$ and $f_*C \sim kL$. In particular, by associating the singular curve $f(C)$ to a point $[f: C \to X]$ of $ \mathcal{T}^n_{g,k}$, one can think of $ \mathcal{T}^n_{g,k}$ as parametrising all singular curves on K3 surfaces such that the normalisation map is unramified (or equivalently such that the curve has ``immersed" singularities). 

The stack  $ \mathcal{T}^n_{g,k}$ comes with a natural moduli map $$\eta \;  : \;\mathcal{T}^n_{g,k} \to \mathcal{M}_{p(g,k)-n}$$ to the Deligne--Mumford stack of curves, defined by forgetting the map to the K3 surface, i.e.\ by sending $[f: C \to X]$ to $[C]$. It is natural to ask what one can say about the image of $\eta$. 

We first show that $\eta$ is generically finite (to its image) on at least one component of $\mathcal{T}^n_{g,k} $, in all but finitely many cases. We also consider related questions about the Brill--Noether theory of singular curves on K3 surfaces as well as the surjectivity of twisted Gaussian maps on normalisations of singular curves. Lastly, we apply the deformation theory of $\mathcal{T}^n_{g,k}$ to a seemingly unrelated problem, namely the Bloch--Beilinson conjectures on the Chow group of points of K3 surfaces with a symplectic involution. \\

The results of this thesis have appeared in \cite{huy-kem} and the preprint \cite{kemeny-singular}.
\newpage
\thispagestyle{plain}
\mbox{}

\tableofcontents
\clearpage
\thispagestyle{plain}
\par\vspace*{.35\textheight}{\centering \Large{Dedicated to my parents, Kathy and Gabor\par}}
\newpage
\thispagestyle{plain}
\mbox{}
\chapter{Introduction}
The systematic use of K3 surfaces has led to something of a revolution in the study of the geometry of the general curve. One prominent example of this is that the Brill--Noether theory of a general curve is the same as that of a general smooth curve lying on a K3 surface of Picard number one, \cite{lazarsfeld-bnp}. A second, equally important example, is that the syzygies of general projective curves are in many cases the same as the syzygies of certain hyperplane sections of K3 surfaces, \cite{voisin-even}, \cite{voisin-odd}, \cite{aprodu-farkas-green}, \cite{kemeny-farkas}. In other words, in order to understand some property of a general smooth curve, it can be instructive to first study the case where the curve lies on a K3 surface. Under this assumption the problem can become much simpler, often allowing one to reduce the original question to an exercise in lattice theory.

In recent years, there has been interest in the study of \emph{singular} curves on K3 surfaces. This dates back to physicists such as Yau, Zaslow, Witten and others who initiated the enumerative study of rational curves on K3 surfaces due to the role they play in high energy physics. A stunning early result was the proof of the Yau--Zaslow formula which counts rational curves in a linear equivalence class on a general K3 surface, \cite{bryan-leung}. A key technique in this study was the use of genus zero stable maps, as defined by Gromov and Kontsevich.

Singular curves of high genus on K3 surfaces also have interesting applications to the geometry of K3 surfaces. For example, in \cite{dedieu} a relation between the Severi variety of singular curves and self-rational maps of K3 surfaces was found by Dedieu, see also \cite{chen-self}. More recently, Ciliberto and Knutsen have used such curves in order to construct rational curves on hyperk\"{a}hler manifolds, \cite{ciliberto-knutsen-gonal}. 

These works use a Cartesian approach to studying nodal curves on surfaces; in other words, one considers the curves as defined by equations on the surface and allows these equations to deform in such a way that the singularities are preserved, see \cite{flam}. This approach works well for integral, nodal curves on K3 surfaces but has the draw-back that it is can be somewhat difficult to control for non-nodal curves, which turn up naturally as deformations of curves on special K3 surfaces. Indeed, the main methods at our disposal to directly construct singular curves on K3 surfaces are either via the Torelli theorem or by constructing curves on a Kummer surface. Such curves are often non-reduced or may have several components which intersect non-transversally. 

In order to make up for this, one commonly allows the K3 surfaces to degenerate to unions of rational scrolls, using powerful techniques from\cite{ciliberto-lopez-miranda} and \cite{chen-rational}.  Chains of rational curves on unions of rational scrolls are constructed with the desired properties and then deformed sideways to a nodal curve on a K3 surface. The configurations of rational curves which appear on the central fibre can be intimidating yet beautiful, with huge amounts of structure and symmetry. For an example of the kind of chain configurations which occur, we refer the reader to \cite[Pg.\ 14]{dedieu-ciliberto-pluri}, which illustrates a degenerate nodal curve on the union of four planes. 

One of the main goals of this thesis is to try and develop a different approach to studying some problems related to high genus singular curves on K3 surfaces. 
The realisation that such an approach should be possible was gained by a close reading of \cite{chen}. In this paper, Chen reproves the main result of \cite{chen-singularities} without degenerating to unions of rational surfaces (see also \cite{chen-rational}). One instead allows for non-reduced curves by explicitly performing semi-stable reductions. We express this idea via the general theory of stable maps. A typical example of the resulting degenerations can be seen in Figure 6, \cite[Pg.\ 92]{chen}.

We start by designing a general-purpose ``machine" for producing families of singular curves of arbitrary geometric genus on K3 surfaces. These families can be given many special properties, allowing us to use them to attack several interesting problems. Our machine has two parts. The first part is the Torelli theorem, which combined with results of Nikulin gives the following statement: let $\Lambda$ be any even lattice of signature $(1, \rho)$ for $\rho \leq 9$. Then there exists a $19-\rho$ dimensional algebraic family of K3 surfaces $X_t$ with $\text{Pic}(X_t) \simeq \Lambda$ for all $t$ (note that these surfaces have Picard rank $\rho+1$, in contrast to the usual convention). Now assume $\Lambda$ has a basis 
$$\{e_1, \ldots, e_{\rho+1} \} .$$ By taking linear combinations $\alpha_1 e_1 + \ldots+\alpha_{\rho+1} e_{\rho+1}$ for $\alpha_i \in \mathbb{Z}$, we may produce divisors on $X_t$ which, can be given all kinds of arithmetic properties by choosing $\Lambda$ wisely. 

Such divisors produced by Torelli will not, in the interesting cases, be irreducible. They are usually not reduced. In order to make use of them, we must perturb them slightly so that they become integral. This is done by the second part of our machine, which is the deformation theory of stable maps. We view the curve $\alpha_1 e_1 + \ldots+\alpha_{\rho+1} e_{\rho+1}$ as the image of a stable map $f: D \to X_t$, and then deform the map sideways to a map $f': D' \to Y$, where $Y$ is a different K3 surface. In many cases, we may set-up the lattice theory of $Y$ so that it forces the curve $f'_*(D')$ to become integral. For instance, this must happen if we deform in a primitive class to a K3 surface $Y$ of Picard rank one, and even in higher rank cases the integrality may still be achievable. These integral curves are deformed in such a way as to retain the chosen arithmetic properties of the central divisor $f_* D$. 

\section{Deformations of stable maps to K3 surfaces}
A polarised K3 surface of genus $g$ is a pair $(X,L)$ of a K3 surface $X$ over $\C$ and an ample line bundle $L$ on $X$ with $(L)^2=2g-2$. For $g$ at least three there is a Deligne--Mumford stack $\mathcal{B}_g$ parametrising all K3 surfaces. Let $D$ be a connected nodal curve and consider a morphism $f: D \to X$. We call
$f$ a \emph{stable map} if it has finite automorphism group.

This thesis starts with an exposition of the deformation theory of stable maps to (varying) polarised K3 surfaces. Much of this is by now rather standard, but our presentation differs from the usual one in at least two respects. Namely, we freely use two pieces of heavy machinery: the analytic deformation theory of morphisms as developed by Flenner, \cite{flenner-ueber}, and the theory of relative cycles as developed by Suslin--Voevodsky, \cite{suslin-relative} and Koll\'{a}r, \cite{kollar}.

Set $p(g,k)=k^2(g-1)+1$.  Following \cite{behrend-manin}, one may construct a Deligne--Mumford stack
$$ \mathcal{W}^n_{g,k} \to \mathcal{B}_g,$$
proper over $\mathcal{B}_g$, with fibre over $[(X,L)]$ parametrising all stable maps $f: D \to X$ with $f_*D \sim kL$ and such that
$D$ has arithmetic genus $p(g,k)-n$. 

The stack $\mathcal{W}^n_{g,k}$  comes equipped with a natural morphism of stacks
\begin{align*}
\eta \; : \; \mathcal{W}^n_{g,k} &\to \overline{\mathcal{M}}_{p(g,k)-n} \\
[f: D \to X] & \mapsto [\hat{D}]
\end{align*}
where $\overline{\mathcal{M}}_{p(g,k)-n}$ is the moduli space of stable curves, and the stable curve $[\hat{D}]$ is
obtained by contracting all unstable components of $D$.

The stack $\mathcal{W}^n_{g,k}$ may well have several irreducible components of different dimensions. All that we have in general is the following bound, see Theorem \ref{good-bound}:  for any irreducible component  $I$ of $ \mathcal{W}^n_{g,k}$, $$\dim I \geq 19+p(g,k)-n.$$ A few words on this bound, which is essentially due to Bryan--Leung, \cite{bryan-leung}, are in order. In order to obtain the estimate, transcendental analytic methods play a key role. Indeed one has to allow the K3 surface to deform to a non-projective complex surface and apply the deformation theory of analytic morphisms, \cite{flenner-ueber}. A different (and more powerful) approach to this bound which is appropriate for constructing virtual fundamental cycles can be given using Bloch's theory of the semi-regularity map,  see\cite{bloch-semi}, \cite{maulik-pand-nl}, \cite{kool-thomas}.

The situation above improves if we only consider morphisms which are \emph{unramified}. If $D$ is smooth and $f$ is unramified, then each component $I$ through $[f]$ has dimension \emph{exactly} $19+p(g,k)-n$. Furthermore $I$ dominates $\mathcal{B}_g$; thus the unramified morphism $f: D \to X$ can be deformed horizontally to the general $(X', L') \in \mathcal{B}_g$. An important observation is that this analysis can be generalised to cover many cases where $D$ is reducible, see Proposition \ref{ishi-little}. 

Let $ \mathcal{T}^n_{g,k}$ denote the open substack of $ \mathcal{W}^n_{g,k}$ parametrising unramified morphisms $f: D \to X$ with $D$ smooth which are birational onto their image. By \cite{chen-rational}, $ \mathcal{T}^n_{g,k}$ is nonempty. Thus $ \mathcal{T}^n_{g,k}$ is of pure dimension $19+p(g,k)-n$ and dominates $\mathcal{B}_g$. If $C \seq X$ is an integral nodal curve on a K3 surface of genus $g$ with $C \in |kL|$ and exactly $n$ nodes, then the normalisation of $C$ gives an unramified stable map $f: D \to X$. This gives a one-to-one correspondence between such nodal curves and the open substack $$ \mathcal{V}^n_{g,k} \seq  \mathcal{T}^n_{g,k}$$ parametrising stable maps $f: D \to X$ with $f(D)$ nodal. By a result of Dedieu--Sernesi using an argument of Arbarello--Cornalba and Harris, any component $I \seq \mathcal{T}^n_{g,k}$ for $g \geq 1$ containing a map $[f: D \to X] \in I$ with $D$ non-trigonal must meet $ \mathcal{V}^n_{g,k}$ in a dense open subset, \cite[Thm.\ 2.8]{ded-sern}, \cite[P.\ 107ff]{harris-morrison},  \cite{harrissev}. 

The following deep conjecture has been around for some years but currently seems out of reach or at least very difficult, see \cite[Conj.\ 1.2]{chen-rational}, \cite{dedieu}. See also \cite{cilided}, \cite{dedieu-ciliberto-pluri}, \cite{kemeny-thesis} for some work on special cases.
\begin{conjecture}
The moduli space $\mathcal{V}^n_{g,k}$ is irreducible.
\end{conjecture}

Resticting $\eta$ to $ \mathcal{V}^n_{g,k}$ produces a morphism of stacks
$$ \eta \; : \;  \mathcal{V}^n_{g,k} \to  \mathcal{M}_{p(g,k)-n}$$
which was studied by Flamini, Knutsen, Pacienza, Sernesi in the case $k=1$, \cite{flam}. Specifically, they prove that $\eta$ is dominant
on each component of $\mathcal{V}^n_{g,1}$ in the range $2 \leq g \leq 11 $, $0 \leq n \leq g-2$  and $k=1$, and ask the following question (at least for $k=1$), :
\begin{question} \label{fkps-quest}
Is $\eta$ dominant on each component in the range $2 \leq p(g,k)-n \leq 11$? Is $\eta$ generically finite (to its image) for $p(g,k)-n \geq 12$? 
\end{question}

To ease the notation write $\mathcal{T}^n_{g}:=\mathcal{T}^n_{g,1} $ in the primitive case $k=1$.
The case $n=0$ of Question \ref{fkps-quest} is classical. Indeed, Mori--Mukai showed that if $n=0$, $k=1$ then the morphism 
$$\eta \; : \; \mathcal{T}^0_{g}  \to \mathcal{M}_{g}$$ is generically finite for $g \geq 13$ or $g=11$, but \emph{not} for $g=12$, \cite{mori-mukai}. In the non-primitive case $k \geq 2$, a very different approach using the deformation theory of cones shows that $\eta$ is generically finite for $g \geq 7$ and $n=0$, \cite{cili-classification}. Our first result is an extension of the results on generic finiteness to the singular case $n >0$. In the case $k=1$, we show:
\begin{thm} \label{main-prim-asd}
Assume $g \geq 11$, $n \geq 0$, and let $0 \leq r(g) \leq 5$ be the unique integer such that 
$$g-11 = \left \lfloor \frac{g-11}{6} \right \rfloor 6 +r(g).$$
Then there is a component $I \seq \mathcal{T}^n_{g}$ such that $${\eta}_{|_I}: I \to \mathcal{M}_{g-n}$$ is generically finite for
$g-n \geq 15$. Furthermore, if $r(g) \neq 5$, the lower bound can be improved to $g-n \geq 13$, and if $r(g)=0$ it can be improved to $g-n \geq 12$.
\end{thm}
In the case $k \geq 2$ we show:
\begin{thm} \label{main-nonprim-asd}
Assume $k \geq 2$, $g \geq 8$. 
Then there is a component $I \seq \mathcal{T}^n_{g,k}$ such that $${\eta}_{|_I}: I \to \mathcal{M}_{p(g,k)-n}$$ is generically finite for
$p(g,k)-n \geq 18$. Furthermore, in most cases the lower bound can be improved, see Chapter \ref{finny}.
\end{thm}
Setting $n=0$, we recover the (optimal) statement in the smooth, primitive, case and all cases other than $g=7$ in the nonprimitive case. In particular, this gives a new proof of the generic finiteness theorem for $n=0$, $k \geq 2$, $g \geq 8$ which more resembles the original approach of \cite{mori-mukai}. By construction the component $I$ contains maps $[f: D \to X]$ with $D$ non-trigonal. Thus we have:
\begin{cor} \label{univ-sev-corollary-sdfw}
The restriction $${\eta}_{|_{\mathcal{V}^{n}_{g,k}}}: \mathcal{V}^{n}_{g,k} \to \mathcal{M}_{p(g,k)-n}$$ is generically finite on one component, for the same bounds on $p(g,k)-n$ as in Theorem \ref{main-prim-asd} and \ref{main-nonprim-asd}.
\end{cor}

\section{Chow groups and Nikulin involutions}
In Chapter \ref{nik-inv}, we apply the stable map machinery to a conjecture on the
action of Nikulin involutions on the Chow group of points on a K3 surface. Let $X$ be a complex projective K3 surface with a symplectic involution $i :X\simeq X$, and let $CH^2(X)$ denote the Chow group of points.  The Bloch--Beilinson conjectures predict the following conjecture, see \cite{huybrechts-chow}.
\begin{conjecture}\label{conj:BB}
 The involution $i^*$ acts as the identity on $CH^2(X)$.
\end{conjecture}
Let $f: E \to X$ be a morphism from an smooth, proper curve of genus one to $X$, with one-dimensional image. The pushforward of $f$ induces a group homomorphism
$$ CH^1(E) \simeq \text{Pic}(E) \to CH^2(X).$$
This allows one to relate Conjecture \ref{conj:BB} to the study of singular elliptic curves on $X$ which are invariant under $i$. In particular, by using that $CH^2(X)$ is torsion-free due to a theorem of Roitman, one can show the following, see Section \ref{BloBei}.
\begin{prop}
Let $i: X \to X$ be a symplectic automorphism of finite order on a projective K3 surface over $\C$. Assume there exists a dominating family of integral
curves $C_t \seq X$ of geometric genus one, such that for generic $t$ the following two conditions are satisfied:
\begin{enumerate}
\item $C_t$ avoids the fixed points of $f$.
\item $C_t$ is invariant under $f$.
\end{enumerate}
Then $i^*=\text{id}$ on $CH^2(X)$.
\end{prop}
This reduces the problem to the question of constructing dominating families of singular elliptic curves which are kept invariant under the involution.

The moduli space of K3 surfaces together with a Nikulin involution has been constructed by van Geemen and Sarti, \cite{sar-gee}, building upon the foundational work of Nikulin,  \cite{nikulin-finite}. The moduli spaces are constructed as spaces of $\Lambda$-polarised K3 surfaces, as in \cite{dolgachev}. To be more precise, consider the lattice $$\Phi_g := \mathbb{Z}L \oplus E_8(-2),$$
where $L^2=2g-2$, $g \geq 3$. If $g$ is odd, then there is a unique lattice $ \Upsilon_g$ which is an over lattice of $\Phi_g$ with $ \Upsilon_g / \Phi_g \simeq \mathbb{Z} / 2\mathbb{Z}$ and such that $E_8(-2)$ is a primitive sublattice of $\Upsilon_g$. 
There is an involution $j : \Phi_g \to \Phi_g$ which acts at $1$ on $\mathbb{Z}L$ and $-1$ on $E_8(-2)$. This extends to an involution $j: \Upsilon_g \to \Upsilon_g$ for odd $g$. For the proof of the following, see Thm.\ \ref{sarti-geemans-thm}
\begin{thm} [Sarti--Geemen]
Let $\Lambda$ be either of the lattices $\Phi_g$ or $\Upsilon_g$. Let $X$ be a K3 surface with a primitive embedding $\Lambda \hookrightarrow \text{Pic}(X)$ such that $L$ is big and nef. Then $X$ admits a Nikulin involution $f: X \to X$ such that $f^*_{|_\Lambda}=j$ and $f^*_{|_{\Lambda^{\perp}}}=\text{id}$, for $\Lambda^{\perp} \seq H^2(X,\mathbb{Z})$. Conversely, if $X$ admits a Nikulin involution $f$ then there is a primitive embedding $\Lambda \hookrightarrow \text{Pic}(X)$ such that $L$ is big and nef, where $\Lambda$ is either $\Phi_g$ or $\Upsilon_g$, for some $g$. Further, $f^*_{|_\Lambda}=j$ and $f^*_{|_{\Lambda^{\perp}}}=\text{id}$.
\end{thm}

Using this, we are able to prove Conjecture \ref{conj:BB} in one-third of all cases.
\begin{cor}
Conjecture  \ref{conj:BB} holds for an arbitrary $\Upsilon_{2e+1}$-polarised K3 surface $(X,f)$.
\end{cor}

These results have been subsequently extended to cover \emph{all} K3 surfaces admitting a Nikulin involution, see \cite{voisin-bloch}. More generally, Conjecture \ref{conj:BB} is now known to hold for all K3 surfaces admitting symplectic automorphisms of finite order. See \cite[\S 5]{huy-kem} for this result on one component of the moduli space and order prime and \cite{huy-derived} for the full statement.

\section{An obstruction for a marked curve to admit a nodal model on a K3 surface}
It is a natural question to study the image of $$ \eta \; : \;  \mathcal{V}^n_{g,k} \to  \mathcal{M}_{p(g,k)-n}.$$  In the case of smooth curves $n=0$, there is a well-known conjectural characterization of the image $\eta$, due to Wahl \cite{wahl-square}. He makes the following remarkable conjecture, which would give a complete characterization of those smooth curves that lie on a K3 surface:
\begin{conjecture}[Wahl] \label{wahlconj}
Assume $C$ is a smooth curve of genus $g \geq 8$ which is Brill--Noether general. Then there exists a K3 surface $X \seq \proj^g$ such that $C$ is a hyperplane section of $X$ if and only if the Wahl map $W_C$ is nonsurjective. 
\end{conjecture}
Here the Wahl map refers to the map $\bigwedge^2 H^0(C,K_C) \to H^0(C,K_C^3)$ given by 
$s \wedge t \mapsto tds-sdt$. One side of this conjecture is well-known; indeed if $C \seq X$ is a smooth curve in a K3 surface then $W_C$ is nonsurjective, \cite{wahl-jac}. Furthermore, if $C$ is general and $\text{Pic}(X) \simeq \mathbb{Z}C$, then $C$ is Brill--Noether general, \cite{lazarsfeld-bnp}. In \cite[Question 5.5]{flam}, it was asked if there exists such a Wahl-type obstruction for a smooth curve to have a nodal model lying on a K3 surface.

Let $\widetilde{\mathcal{M}}_{p(g,k)-n,2n} := [\mathcal{M}_{p(g,k)-n,2n}/ S_{2n}]$ denote the stack of curves with an unordered marking (or divisor). One may slightly alter the above question and ask if there exists an obstruction for a \emph{marked} curve to have a nodal model lying on a K3 surface in such a way that the marking is the divisor over the nodes (when we forget about the ordering).
For  any positive integers $h,l$ and $[(C,T)] \in \widetilde{\mathcal{M}}_{h,2l}$, one may consider the Gaussian
$$W_{C,T}: \bigwedge^2 H^0(C,K_{C}(-T)) \to H^0(C,K_{C}^3(-2T)) $$
which we will call the \emph{marked Wahl map}, since it depends on both the curve and the marking. In Section \ref{sec:markedwahl} we show:
\begin{thm}
Fix any integer $l \in \mathbb{Z}$. Then there exist infinitely many integers $h(l)$, such that the general marked curve $[(C,T)] \in \widetilde{\mathcal{M}}_{h(l),2l}$ has surjective marked Wahl map.
\end{thm}
On the other hand we show:
\begin{thm}
Assume $g-n \geq 13$ for $k=1$ or $g \geq 8$ for $k >1$, and let $ n \leq \frac{p(g,k)-2}{5}$. Then there is an irreducible component $I^0 \seq \mathcal{V}^n_{g,k}$ such that for a general $[(f: C \to X,L)] \in I^0$ the marked Wahl map $W_{C,T}$ is nonsurjective, where $T \seq C$ is the divisor over the nodes of $f(C)$.
\end{thm}

\section{Brill--Noether theory for nodal curves on K3 surfaces}
In the last section we study the Brill--Noether theory of nodal curves on K3 surfaces. There are two related questions: for $[(f: C \to X,L)] \in \mathcal{V}_g^n$ general, one may firstly ask if the smooth curve $C$ is Brill--Noether general and secondly if the nodal curve $f(C)$ is Brill--Noether general. For the first question we show in Section \ref{BNP-nodal}:
\begin{prop} 
Assume $g-n \geq 8$. Then there exists a component $\mathcal{J} \seq \mathcal{V}^n_g$ such that for $[(f:C \to X,L)] \in \mathcal{J}$ general, $C$ is Brill--Noether--Petri general.
\end{prop}
The above result should not be expected to hold for all $[(f:C \to X,L)] \in \mathcal{J}$ (or even for all $[(f:C \to X,L)] \in \mathcal{J}$ with the general polarised K3 surface $(X,L)$ kept fixed), see \cite[Thm.\ 0.1]{ciliberto-knutsen-gonal}.

For the second question we again have a positive answer. For an integral nodal curve $D$, we denote by $\bar{J}^d(D)$ the compactified Jacobian of degree $d$, rank one, torsion-free sheaves on $D$. In Section \ref{rat} we show:
\begin{thm}
Let $X$ be a K3 surface with $\text{Pic}(X) \simeq \mathbb{Z} L$ and $(L \cdot L)=2g-2$. Suppose $D \in |L|$ is a rational, nodal curve. Then
$$\overline{W}^r_d(D) := \{ \text{$A \in \bar{J}^d(D)$ with $h^0(A) \geq r+1$} \}$$ is either empty or is equidimensional of the expected dimension $\rho(g,r,d)$.
\end{thm} 
As one may smoothen the nodes of a rational nodal curve $D$ on a K3 surface to produce a curve with an arbitrary number of nodes, the above result immediately gives the following corollary:
\begin{cor}
For any $n \geq 0$, there is a component $\mathcal{J} \seq \mathcal{V}^n_g$ such that if $[(f:C \to X,L)] \in \mathcal{J}$ is general and $D=f(C)$ then
$$\overline{W}^r_d(D) := \{ \text{$A \in \bar{J}^d(D)$ with $h^0(A) \geq r+1$} \}$$ is either empty or is equidimensional of the expected dimension $\rho(g,r,d)$. 
\end{cor}
In particular, if $\rho(g,r,d) <0$, $\overline{W}^r_d(D) = \emptyset$ for $D$ as in the above corollary; indeed this is well-known and follows from the arguments of \cite[\S 3.2]{gomez}, \cite[Cor.\ 1.4]{lazarsfeld-bnp}. On the other hand, if $\rho(g,r,d)  \geq 0$, then $\overline{W}^r_d(D) \neq \emptyset$ by deforming $D$ to a smooth curve on $X$ and semicontinuity.

\subsection*{Acknowledgements}
First of all, I am very much in debt to my thesis advisor Professor Daniel Huybrechts. He taught me the fundamental theory of K3 surfaces, explained the importance of rational and elliptic curves, and introduced me to stable maps as a way of parametrising singular curves. Just as importantly, he provided me with infinite encouragement at every step of the way. I would further like to offer special thanks to A.\ Knutsen and E.\ Sernesi for very important conversations and suggestions, and G.\ Farkas for agreeing to be the second reader of this thesis. Furthermore, discussions with my colleague S.\ Schreieder about the literature on singular curves on K3 surfaces proved to be most valuable. Thanks also to C.\ Ciliberto, X.\ Chen, F.\ Flamini, G.\ Galati, F. Gounelas and U.\ Greiner for discussions and explanations relating to this work. I am most grateful to an anonymous referee for compiling a very long list of misprints and errors which led to innumerable improvements.

Without the constant support of Ishita Agarwal it is doubtful that I would have completed this thesis.

This work was funded by a PhD scholarship from the Bonn International Graduate School in Mathematics and by SFB/TR 45.

\chapter{Stable maps}
In this chapter we construct the moduli space $ \mathcal{W}^n_{g,k}$ of stable maps to projective K3 surfaces and study its deformation theory. We proceed in three
steps. In Section \ref{ConstructMaps}, we restrict ourselves to considering families of stable maps over an algebraic base. 
In this setting one has an algebraic stack parametrising all such families. The standard deformation theory of this object is however insufficient for
most purposes, essentially because deformations of algebraic K3 surfaces need not remain algebraic.

 In order to have a workable deformation theory, in Section  \ref{complexDef} we no longer assume \emph{a priori} that all K3 surfaces remain algebraic in their deformations. We also allow the base of the families of stable maps to be arbitrary complex
spaces. There is then a \emph{local} construction of a moduli space parametrising such objects, which exists as a complex space, together with a well-defined analytic deformation
theory producing an optimal dimension count. A simple argument shows that, \emph{a posteriori}, all K3 surfaces do indeed remain algebraic in the deformations, so the resulting dimension count applies to the original algebraic moduli space $ \mathcal{W}^n_{g,k}$ as well. 

In Section \ref{unramSect}, we apply the above analysis to the case of an \emph{unramified} stable map $f: C \to X$. The computations done in this section form the backbone of the rest of this thesis.

\section{Construction of the moduli space of stable maps} \label{ConstructMaps}
We start by fixing some notation.
\begin{mydef}
By \emph{connected curve} we mean a projective algebraic variety of dimension one over $\C$ which is connected. A connected curve is
not assumed to be either reduced or irreducible. By  \emph{nodal curve} we mean a connected curve which is reduced and with all singularities (if any) nodal.
\end{mydef}
Let $X$ be a projective variety over $\C$. An integral connected curve $C \seq X$ can be viewed in one of two ways. The \emph{Cartesian perspective} views $C$ as being defined via polynomial equations on $X$; in other words we identify $C$ in a one-to-one fashion with the ideal sheaf $I_C$ defining $C$ via the exact sequence
$$0 \to I_C \to \mathcal{O}_X \to \mathcal{O}_C \to 0 .$$ The study of deformations of $I_C$ leads to the construction of the \emph{Hilbert scheme} of $X$, see for instance \cite[Ch.\ 1]{kollar}. Suppose now $C$ is an integral nodal curve. Deformation of $I_C$ in the Hilbert scheme of $X$ will often smooth out the nodes of $C$. If one wishes to preserve the singularities of $C$ one is led to theory of the \emph{Severi variety}, see \cite{tannen2} and \cite{flam}. This works best when $X$ is a smooth surface. 

We will take a different approach. Rather than viewing the integral nodal curve $C \seq X$ as defined by equations on $X$ and studying those deformations of the equations which preserve the nodes, we instead consider $C$ as the image $f_*(\tilde{C})$ of a morphism $$f \; : \; \tilde{C} \to X$$ where $\tilde{C}$ is a smooth curve. Indeed, we take $f$ to be the composition of the normalisation morphism $\tilde{C} \to C$ with the closed immersion $C \hookrightarrow X$. This viewpoint is known as the \emph{parametric perspective} and is due to Horikawa, \cite{horikawa}. The resulting deformation theory of the morphism $f$ turns out to be somewhat more workable than the theory of the Severi variety even when $X$ is a surface. Indeed, with this perspective one often has a much better description of limiting degenerations of $f_*(\tilde{C})$ (which will often be neither reduced nor irreducible, let alone integral nodal).

In order to describe all possible deformations of the morphism $f$ we make the following definition, \cite{fulpar}.
\begin{mydef}[Gromov, Kontsevich]
Let $f : D \to X$ be a morphism from a connected nodal curve $D$ to a projective variety $X$, and assume $f(D)$ is not a point. The morphism $f$ is a \emph{stable map} if the following condition is satisfied:
suppose $E \seq D$ is an irreducible component with $E \simeq \proj^1$ and $f(E)$ a point. Then $E$ meets the rest of the curve in at least three points, i.e.\ $$\# \{E \cap \overline{D \setminus E} \} \geq 3.$$
\end{mydef}
The above condition may be characterised in several alternate ways, \cite[Prop.\ 3.9]{behrend-manin}, \cite{fulpar}, \cite{li-tian}. 
\begin{prop}
Let $f : D \to X$ be a morphism from a connected nodal curve $D$ to a smooth projective variety $X$, and assume $f(D)$ is not a point. Then the following conditions are equivalent:
\begin{enumerate}
\item The morphism $f$ is stable.
\item The automorphism group of $f$ is finite.
\item The line bundle $\omega_D \otimes f^* \mathcal{O}_X(3)$ is ample.
\item The map $\text{Hom}_D (\Omega_D, \mathcal{O}_D) \to H^0(D, f^*T_X)$ induced by the natural map $f^* \Omega_X \to \Omega_D$
is injective.
\item The infinitesimal automorphism group of $f$ is trivial.
\end{enumerate}
\end{prop}

The definition of stable maps generalises easily to the relative setting:
\begin{mydef}
Let $S$ be a Noetherian scheme over $k= \C$ and $\mathcal{X}$ an $S$-scheme. A family of stable maps over an $S$-scheme $U$ is a pair $(\mathcal{C},f)$ where
\begin{itemize}
\item $\mathcal{C} \to U$ is a flat, proper family of nodal curves
\item $f: \mathcal{C} \to \mathcal{X} \times_S U$ is a $U$-morphism
\item for each $k$-point of $U$, $f_k: \mathcal{C}_k \to \mathcal{X}_k$ is a stable map.
\end{itemize}
\end{mydef}
Let $\mathcal{X} \to S$ be a projective morphism of $\C$-schemes over a Noetherian scheme $S$ with a fixed $S$-ample line bundle $\mathcal{O}_{\mathcal{X}}(1)$. There
is a functor from the category $\textbf{Sch}_{S}$ of $S$-schemes to the category $\textbf{Set}$ of sets
\begin{align*}
\mathcal{W} (\mathcal{X},d,p) \; : \; \textbf{Sch}_S \to \textbf{Set}
\end{align*}
such that $\mathcal{W} (\mathcal{X},d,p)(U)$ is the set of isomorphism classes of stable curves $(\mathcal{C},f)$ over $U$ such that for all $k$-points of $U$, $\mathcal{C}_k$ has constant arithmetic genus $p$ and $\deg f_k^*(\mathcal{O}_{\mathcal{X}}(1))=d$.

\begin{thm} \label{const-W}
Let $S$ be a Noetherian scheme over $\C$ and let $\mathcal{X} \to S$ be a projective morphism of schemes with a fixed $S$-ample divisor $\mathcal{O}_{\mathcal{X}}(1)$.
The functor $\mathcal{W} (\mathcal{X},d,p)$ is a Deligne--Mumford stack, proper over $S$, which admits a projective coarse moduli space $\mathbb{W} (\mathcal{X},d,p)$.
\end{thm}
\begin{proof}
The construction is given in \cite{behrend-manin} in the case $S=\text{Spec}(k)$, see also \cite{fulpar}.  For the relative case, see \cite[Thm.\ 50]{arakol}, \cite[\S 1.2]{vist-ab}, \cite{oort-ab}.
\end{proof}
\begin{remark}
Over finite characteristic, $\mathcal{W} (\mathcal{X},d,p)$ is still a proper, algebraic Artin stack, but it may fail to be Deligne--Mumford, see the hypotheses of \cite[Prop.\ 4.1]{behrend-manin}.
\end{remark}
We will outline the construction of the algebraic stack $\mathcal{W} (\mathcal{X},d,p)$, taken from \cite{arakol}. Let $(\mathcal{C},f)$ be a family of stable maps over
$U$, and let $\pi: \mathcal{C} \to U$ be the projection. By the third characterisation of stable maps, $M:=\omega_{\mathcal{C} / U} \otimes f^* \mathcal{O}_{\mathcal{X}}(3)$ is $U$-relatively ample. In fact, $M^2$ is $U$-very ample and $R^1\pi_*M^2=0$. Thus there is some constant $N$ (independent of $U$) such that $M^2$ induces a closed immersion $j: \mathcal{C} \hookrightarrow \proj^N_{U}$. The diagonal $(f, j)$ then gives a closed immersion
$$\mathcal{C} \hookrightarrow \mathcal{X} \times_S \proj^N_S \times_S U .$$ The proper algebraic stack $\mathcal{W} (\mathcal{X},d,p)$ is then a quotient of an open subset of $\text{Hilb}_S (\mathcal{X} \times_S \proj^N_S)$ by $\text{PGL}(N+1)$. The projectivity of $\mathbb{W} (\mathcal{X},d,p)$ is explained in \cite[\S 4.3]{fulpar} (for a reference which explicitly works with the relative case see \cite[\S 2.4, 2.5]{oort-ab}).

Let $\overline{\mathcal{M}}_p$ denote the Deligne--Mumford stack of stable curves of arithmetic genus $p$. For any nodal curve $D$, contracting all unstable components yields a unique stable curve $D'$. That this works well in families is the content of the next proposition, which is essentially \cite[\S 1.3]{fulpar} (see also \cite[\S 2.6]{oort-ab}).
\begin{prop} \label{const-morph}
There is a morphism of algebraic stacks 
\begin{align*}
\eta \; : \;  \mathcal{W} (\mathcal{X},d,p) \to  \overline{\mathcal{M}}_p 
\end{align*}
which on $k$-points acts as follows: if $f: D \to \mathcal{X}_k \in  \mathcal{W} (\mathcal{X},d,p)(\text{Spec}(k))$, then $$\eta([f])=[D'] \in \overline{\mathcal{M}}_p(\text{Spec}(k))$$ where
$D'$ denotes the contraction of all unstable components of $D'$.
\end{prop} 
\begin{proof}
Consider a family of stable curves $(\mathcal{C},f) \in \mathcal{W} (\mathcal{X},d,p)(U)$. Contracting unstable components of the fibres of $\mathcal{C} \to U$ yields a uniquely-defined family of stable maps $\mathcal{C}' \in \overline{\mathcal{M}}_p(U)$, by the proof of \cite[Prop.\ 2.1]{knudsen}. This therefore yields a morphism of stacks 
$\eta \; : \;  \mathcal{W} (\mathcal{X},d,p) \to  \overline{\mathcal{M}}_p$.
\end{proof}

We now recall some facts from the theory of families of algebraic cycles over a base $S$ from \cite{suslin-relative}, \cite[I.3]{kollar}. See also \cite[Ch.\ 1]{fulton}, \cite[Appendix I.A]{levine-mixed}, \cite[\S 8]{cisinski-deglise}. \footnote{The main definition of \cite{suslin-relative} apparently requires a straightforward modification in the case where the base $S$ is non-reduced at generic points, see \cite[Ch.\ 2]{shane-kelly}. Much of the theory anyhow requires $S$ to be \emph{normal} and we will always make this assumption, so this issue is not relevant to us (families of cycles over non-reduced bases are left undefined in \cite[Def.\ I.3.10]{kollar}).} We will use this theory to pass from a family of stable maps $f_k : C_k \to X_k$ to a family of (non-reduced) curves $f_{k,*} (C_k) \seq X_k$, at least in the simple case where $X_k$ is a smooth surface.

First recall the definition of algebraic cycles and pushforward from \cite[\S 1.4]{fulton}, \cite[I.3.1]{kollar}.
\begin{mydef}
Let $X$ be a scheme of finite type over $\C$. A $d$-dimensional algebraic cycle on $X$ is a formal linear combination $\sum_{i=1}^k a_i [V_i]$, $a_i \in \mathbb{Z}$, where $V_i$ is an integral closed subscheme of dimension $d$ in $X$. The abelian group of all $d$-dimensional cycles is denoted $Z_d(X)$. If $W \seq X$ is a closed subscheme of pure dimension $d$ and $W_1, \ldots, W_k$ are the irreducible components of $W_{red}$, we define the fundamental cycle $[W] \in Z_d(X)$ by
$$[W]= \sum_{i=1}^k l(\mathcal{O}_{W, \eta_i}) [W_i] $$
where $l(\mathcal{O}_{W, \eta_i})$ denotes the length of $\mathcal{O}_{W, \eta_i}$, where $\eta_i$ is the generic point of $W_i$ (the fundamental cycle is written as $\text{cycl}_X(W)$ in \cite{suslin-relative}).

Let $f:X \to Y$ be a morphism of schemes of finite type over $\C$, and let $V_i$ be an integral closed subscheme of dimension $d$ in $X$.  Let $W_i$ be the closure of $f(V_i)$ with reduced structure. We set
$$f_*(V_i) = \begin{cases} \deg{(V_i / W_i)} [V_i] & \mbox{if} \dim(W_i)=d \\
0 & \mbox{otherwise}.
\end{cases} $$ This extends by linearity to a homomorphism of abelian groups
$$ f_* \; : \; Z_d(X) \to Z_d(Y)$$
called the \emph{pushforward}.
\end{mydef}

The following lemma is \cite[Prop.\ 3.4.8]{suslin-relative} (compare with \cite[I.3.23.2]{kollar}). 
\begin{lem} \label{kollar-flat}
Let $S$ be an integral normal scheme of finite type over $\C$ and let $p:Y \to S$ a smooth, finite type morphism of relative dimension $d$. Let $X/S \hookrightarrow Y/S$ be a closed immersion of $S$-schemes with $X$ integral and such that all fibres of the projection $q:X \to S$ have pure dimension $d-1$. Then $q:X \to S$ is flat, and $X$ is a relatively effective Cartier divisor in $Y$.
\end{lem}
\begin{proof}
From \cite[21.14.3]{egaiv-iv}, $X$ is a locally principal closed subscheme of $Y$. Thus each fibre $X_s \seq Y_s$ over $s \in S$ is locally principally closed, and since $X \to S$ has all fibres of pure dimension $d-1$ and $Y \to S$ is smooth, each fibre $X_s \seq Y_s$  must be an effective Cartier divisor. The claim now follows from \cite[\href{http://stacks.math.columbia.edu/tag/062Y}{Tag 062Y}]{stacks-project}.
\end{proof}
Note that when $S$ is itself smooth the above result is also a special case of the well-known result \cite[Prop.\ 6.1.5]{egaiv}.
As a consequence of the above lemma we have the following:
\begin{prop} \label{push-forwards}
Let $\mathcal{X} \to S$ be a smooth, projective family of surfaces, where $S$ is an integral normal scheme of finite type over $\C$. Suppose
$ \mathcal{C} \to S$ is a flat, proper family of nodal curves and $f: \mathcal{C} \to \mathcal{X}$ is a family of stable maps over $S$. Assume that for each closed point $s: \text{Spec}(k) \to S$, $f_s: \mathcal{C}_s \to \mathcal{X}_s$ has one-dimensional image. Then there is a relatively effective Cartier divisor $\mathcal{D} \to  \mathcal{X} $ over $S$ such that $\mathcal{D}_s \simeq f_*{\mathcal{C}_s}$ (as subschemes) for all closed points $s \in S$. 
\end{prop}
\begin{proof}
Let $\mathcal{C}_1, \ldots, \mathcal{C}_j$ be the irreducible (and reduced) components of $\mathcal{C}$, set $V_i:=f(\mathcal{C}_i)$ with reduced induced scheme structure and let $f_i: \mathcal{C}_i \to V_i$ be the restriction of $f$ to the components. Let $d_i=\deg f_i$ in case $\dim \mathcal{C}_i=\dim V_i$ and $d_i=0$ otherwise, for $1 \leq i \leq j$. By Lemma \ref{kollar-flat}, each $V_i$ satisfying $\dim \mathcal{C}_i=\dim V_i$ is a relatively effective Cartier divisor over $S$, so we have a relatively effective Cartier divisor $\mathcal{D}:= \sum_{i=1}^j d_i V_i$.

We claim that $\mathcal{D}_s \simeq f_*{\mathcal{C}_s}$ for any point $s: \text{Spec}(k) \to S$. Then $\mathcal{D}$ is our desired family of relatively effective Cartier divisors. To show \emph{rational equivalence} $\mathcal{D}_s \sim f_*{\mathcal{C}_s}$, one could take a desingularisation of $S$ and then apply standard intersection theory as in \cite[Ch.\ 10]{fulton}. That the two divisors are \emph{equal} in the strong sense requires the methods of \cite{suslin-relative} or \cite{kollar}. For instance, we may directly apply \cite[Thm.\ 3.6.1 ]{suslin-relative} (also see \cite[Prop.\ I.3.2]{kollar}).
\end{proof}
In the situation of the above lemma, we will write $f_{*} \mathcal{C}$ for the relatively-effective divisor $\mathcal{D}$. 
\begin{prop} \label{unramifiedisanopenstack}
Let $S$ be a scheme of finite type over $\C$, and let $\mathcal{X} \to S$ be a smooth, projective morphism of K3 surfaces. Then there is an open substack $\mathcal{T}(\mathcal{X},d,p) \seq \mathcal{W}(\mathcal{X},d,p)$ parametrising stable maps $[f: C \to X]$ with $C$ smooth, $f$ unramified and birational onto its image. There is further an open substack $\mathcal{V}(\mathcal{X},d,p) \seq \mathcal{T}(\mathcal{X},d,p) $ parametrising stable maps $[f: C \to X]$ which satisfy the additional assumption that $f(C)$ is nodal.
\end{prop}
\begin{proof}
Let $ \hat{\mathcal{W}}(\mathcal{X},d,p) \to  \mathcal{W}(\mathcal{X},d,p)$ be an \'etale atlas, with universal family of stable maps
$$\xymatrix{
\mathcal{C} \ar[r]^{f} \ar[d]
&\mathcal{X} \ar[ld] \\
 \hat{\mathcal{W}}(\mathcal{X},d,p).
}$$ We need to construct open subsets $U_1 \seq U_2 \seq \hat{\mathcal{W}}(\mathcal{X},d,p)$ such that $p \in U_2$ if and only if $[f_p: \mathcal{C}_p \to X_p]$ has $\mathcal{C}_p$ smooth, $f_p$ unramified and birational onto its image, and $p \in U_1$ if and only if we have in addition that $f_p (\mathcal{C}_p)$ is nodal (note that these properties are obviously preserved by automorphisms of $f_p$). Let $I_1, \ldots, I_k$ be the irreducible components of $\hat{\mathcal{W}}(\mathcal{X},d,p)$, with reduced structure, and let $\pi_i : N_i \to I_i$ be the normalisation morphism for each $i$. Let $\tilde{f} : \mathcal{C}_{N_i} \to \mathcal{X}_{N_i}$ be the pull-back of $f$. It suffices to construct open subsets $\tilde{U}_{1,i} \seq \tilde{U}_{2,i} \seq N_i$ parametrising points $s$ such that $\tilde{f}_s$ is a stable map with the properties corresponding to $U_1$ and $U_2$; indeed $\pi_i$ is proper so we can set $$U_j=\bigcup_{i=1}^k\hat{\mathcal{W}}(\mathcal{X},d,p) \setminus \pi(N_i \setminus \tilde{U}_{j,i})$$ for $j=1,2$. It is well-known that there is an open set $V_{1,i} \seq N_i$ parametrising points $p$ where $\mathcal{C}_{N_i,p}$ is smooth. If we let $V_{2,i}$ be the complement of the support of $\Omega_{\mathcal{C}_{N_i} / \mathcal{X}_{N_i}}$, then $V_{2,i}$ parametrises points $p$ where $ \tilde{f}_p$ is unramified. 

Let $W_i=V_{1,i} \cap V_{2,i}$. Applying Lemma  \ref{push-forwards}, we let $\mathcal{D} \to \mathcal{X}_{W_i}$ be the flat family of Cartier divisors corresponding to the push-forward of $\tilde{f}_{|_{W_i}}.$ We let $\tilde{U}_{2,i}$ be the open subset parametrising points $p \in W_i$ with the irreducible curve $\mathcal{D}_p$ reduced and then $\tilde{U}_{1,i} \seq \tilde{U}_{2,i}$ the open subset with $\mathcal{D}_p$ in addition nodal, \cite[Lemma 3.34]{harris-morrison}.
\end{proof}
The following proposition is standard, e.g.\ compare with \cite[\S 8.2]{vist-ab}.
\begin{prop}
Assume that $\mathcal{X} \to S$ is a smooth, projective family of K3 surfaces, where $S$ is separated of finite type over $\C$. Let $\mathcal{L} \in \text{Pic}(\mathcal{X})$ be an $S$-flat line bundle with $(\mathcal{L}_s \cdot \mathcal{O}_{\mathcal{X}_s}(1))=d$ for any $s \in S$. There is an open and closed substack $$\mathcal{W} (\mathcal{X}, \mathcal{L},p) \seq \mathcal{W} (\mathcal{X},d,p)$$ such that a closed point $f_k: \mathcal{C}_k \to \mathcal{X}_k$ in $\mathcal{W} (\mathcal{X},d,p)(\C)$ is contained in $\mathcal{W} (\mathcal{X}, \mathcal{L},p)(\C)$ if and only if we have the rational equivalence $$f_{k,*} \mathcal{C}_k \sim \mathcal{L}_k.$$
\end{prop}
\begin{proof}
Let $ \hat{\mathcal{W}}(\mathcal{X},d,p) \to  \mathcal{W}(\mathcal{X},d,p)$ be an \'etale atlas and $I_1, \ldots, I_k$ the (reduced) irreducible components of $ \hat{\mathcal{W}}(\mathcal{X},d,p)$. Let $\pi_i : R_i \to I_i$ be a resolution of singularities for each $i$ and let $\tilde{f} : \mathcal{C}_{R_i} \to \mathcal{X}_{R_i}$ be the pull-back of the universal stable map as in the previous proposition, with $j:\mathcal{X}_{R_i} \to R_i$ the projection. It is enough to show that for each $R_i$ we have either $f_{s,*} \mathcal{C}_s \sim \mathcal{L}_s$ holds for all closed points $s \in R_i$ or for no points at all. Indeed, suppose this holds and label the irreducible components such that $f_{s,*} \mathcal{C}_s \sim \mathcal{L}_s$ holds for all closed points (or equivalently one point) of precisely the components $I_1, \ldots, I_m$. Then 
$$\bigcup_{i=1}^m I_i=\bigcup_{j=m+1}^k (\hat{\mathcal{W}}(\mathcal{X},d,p) \setminus I_j)$$ is an open and closed subset of $ \hat{\mathcal{W}}(\mathcal{X},d,p)$ with the desired property.

Applying Lemma \ref{push-forwards}, let $\mathcal{D} \to \mathcal{X}_{R_i}$ be the flat family of Cartier divisors corresponding to the push-forward of $\tilde{f}_{|_{R_i}}$ and let $\mathcal{M} = \mathcal{O}_{\mathcal{X}_{R_i}}(\mathcal{D})$ be the corresponding line bundle. We work in the category of complex manifolds and follow \cite[\S I.3.1.2]{voisin-vol2}. Let $s \in R_i$ be a closed point, and let $U \seq R_i$ be an contractible open set about $s$, taken in the classical topology. Denote by $\mathcal{X}_U$ resp. $\mathcal{X}_s$ the complex manifolds $j^{-1}(U)$ resp.\ $j^{-1}(s)$. Replacing $U$ with a smaller open set if necessary, Ehresmann's lemma states that there exists a homeomorphism $\mathcal{X}_U \simeq \mathcal{X}_s \times U$ of fibrations over $U$. Since $U$ was chosen to be contractible, pullback via the inclusion $\mathcal{X}_s \to \mathcal{X}_U$ induces an isomorphism
$$H^2(\mathcal{X}_U, \mathbb{Z}) \simeq H^2(\mathcal{X}_s, \mathbb{Z})$$
for any $s \in U$.
This gives (by the functorialities of the first Chern class), that for any $s,t \in U$,
$\mathcal{L}_t \simeq \mathcal{M}_t$ if and only if $c_1(\mathcal{L}_t \otimes \mathcal{M}^*_t)=0$, which happens if and only if $\mathcal{L}_s \simeq \mathcal{M}_s$. As $R_i$ is connected, we see that either $f_{s,*} \mathcal{C}_s \sim \mathcal{L}_s$ holds for all closed points $s \in R_i$ or for no points at all, as required.
\end{proof}
Now consider the Deligne--Mumford stack $\mathcal{B}_g$ of primitively polarised K3 surfaces of genus $g \geq 3$. Analogously to Theorem \ref{const-W} and Proposition \ref{const-morph} we have:
\begin{prop}
Set $p(g,k):=k^2(g-1)+1$ for $g \geq 3$.
There is a Deligne--Mumford stack $\mathcal{W}_{g,k}^n$, proper over $\mathcal{B}_g$, with fibre over a polarised K3 surface $[(X,L)]$ parametrising all stable maps $f: D \to X$ with $f_{*}D \in |kL|$ and such that $D$ is a nodal curve of arithmetic genus $p(g,k)-n$. The stack $\mathcal{W}_{g,k}^n$ comes equipped with a morphism of stacks $$\eta \; : \;  \mathcal{W}_{g,k}^n \to \overline{\mathcal{M}}_{p(g,k)-n} .$$
\end{prop}
\begin{proof}
Let $\mathcal{X} \seq \proj^{N}_{\mathcal{B}_g}$ be the universal stack of polarised K3 surfaces, where $N=9(g-1)$. The claim follows from Theorem \ref{const-W} and the considerations in \cite[\S 8.3]{vist-ab} or \cite[\S 2.7]{oort-ab}.
\end{proof}
In the same exact same manner as Proposition \ref{unramifiedisanopenstack}, we have an open substack $\mathcal{T}^n_{g,k} \seq \mathcal{W}^n_{g,k}$ parametrising stable maps $[f: C \to X]$ with $C$ smooth, $f$ unramified and birational onto its image, and $\mathcal{V}^n_{g,k} \seq \mathcal{T}^n_{g,k}$ parametrising maps satisfying the additional condition that $f(C)$ is nodal.
\begin{remark}
Throughout this section we have always assumed that our family $\mathcal{X} \to S$ is projective. In the case $\mathcal{X} \to S$  is only proper, one can show that the analogous functor $\mathcal{W} (\mathcal{X},d,p)$ can be constructed as an algebraic Artin stack, locally of finite type, \cite[\S 8.4]{vist-ab}. This can be useful in some applications, for instance when one wants to consider stable maps into families of pseudo-polarised K3 surfaces, i.e.\ pairs $(X,L)$ with $L$ only big and nef, rather than ample. We have stuck to the projective case, since the references are more complete, but many of the statements in this chapter generalise in some form to the proper setting.
\end{remark}

\section{Morphisms of complex spaces} \label{complexDef}
General deformations of an algebraic K3 surface are not algebraic. As a result, the constructions in the preceding section do not capture the full story of deformations of stable maps into K3 surfaces and transcendental techniques are needed before one can work effectively with the stack $\mathcal{W}^n_{g,k}$. In this section we recall the deformation theory of morphisms of complex spaces, following Flenner's Habilitationsschrift \cite{flenner-ueber}. See also the survey article \cite{palamodov}. As a reference for the general theory of complex spaces we recommend \cite{fischer},  \cite[Ch.\ I]{greuel-introduction}. The standard reference for formal deformation theory is \cite{rim-formal}. For us, complex spaces are assumed to be Hausdorff, but \emph{not} reduced.

Let $\mathcal{A}n_{*}$ be the category of complex space germs $(S,*)$ at $*$. The objects of this category are pointed spaces $(S,*)$ and the morphisms are equivalence classes of morphisms defined in an open neighbourhood of the distinguished point. For any $(S,*) \in \mathcal{A}n_{*}$, a \emph{deformation} of a complex space $X$ is a proper, flat morphism $f: \mathcal{X} \to S$ with $f^{-1}(*) \simeq X$. Two deformations are equivalent if they coincide on an open set about the distinguished point $*$. 
\begin{mydef}
A deformation $\phi: \mathcal{X} \to S$ of a compact complex space $X$ is called \emph{complete}, if, for any other deformation $\psi: \mathcal{Y} \to B$ of $X$, there is a morphism of germs $h: B \to S$ such that the pullback deformation $h^* \phi: \mathcal{X}_B \to B $ is equivalent to $\psi$. The deformation $\phi$ is called \emph{versal}, if it satisfies the following stronger condition: for any deformation $\psi$ as above and for all closed embeddings $$i: (A,*) \hookrightarrow (B,*)$$ of germs together with a morphism 
$$\gamma: (A,*) \to (S,*)  $$
such that the pullback $\gamma^* \phi: \mathcal{X}_A \to A$ is isomorphic to $i^* \psi: \mathcal{Y}_A \to A$, then there exists a morphism $h: B \to S$, such that $h^* \phi : \mathcal{X}_B \to B$ is isomorphic to $\psi: \mathcal{Y} \to B$, and in addition $\gamma=h \circ i$. 
\end{mydef}
Note that there are no uniqueness assumptions on the morphism $h$. A versal deformation $\phi: \mathcal{X} \to S$ is called \emph{semiuniversal} if the differential $$Dh: T(B,*) \to T(S,*)$$ between the tangent spaces at $*$ is independent of the choice of $h$. A complete deformation $\phi: \mathcal{X} \to S$ with the property that the morphism $h$ as above is unique is called \emph{universal}; it is automatically versal. Be aware that many authors use the term versal to describe deformations which we are merely calling complete. The above definition is that used in \cite[II.1.3]{greuel-introduction}, \cite{flenner-ueber} and others and is required to ensure that versality behaves well with the corresponding notion for formal spaces and to ensure that versality is an open condition.
 
One of the fundamental results in the deformation theory of complex spaces is the following, \cite[II.8.3]{flenner-ueber}:
\begin{thm}[Douady, Forster--Knorr, Grauert, Palamodov]
Any compact complex space $X$ admits a semiuniversal deformation.
\end{thm}
Any compact complex space $X$ has a deformation-obstruction theory which is described by a complex $T^*(X)$, see \cite[Ch.\ 2.3]{palamodov}. The vector space $T^1(X)$ parametrises deformations of $X$, whereas $T^2(X)$ gives a group of obstructions. If $\mathcal{X} \to S$ is a semiuniversal deformation of $X$, then the germ $S$ can be described locally near the origin as the zeroes of a holomorphic map $\phi: V \to T^2(X)$, where $V \seq T^1(X)$ is an open neighbourhood about $0 \in T^1(X)$. In particular, $\dim (S,*) \geq \dim T^1(X)-\dim T^2(X)$. Moreover, the set of points in $s \in S$ such that  $\mathcal{X} \to S$ remains versal at $s$ is Zariski open, \cite[Thm.\ 7.1]{bingener}.

For any complex space $S$, let $\mathcal{A}n_S$ be the category of complex spaces over $S$. Let $ X \to S$ be a morphism of complex spaces. There is a functor
$$\text{Hilb}(X /S) \; : \;  \mathcal{A}n_S \to \{ \text{Sets} \}$$ defined by setting
$\text{Hilb}(X /S)(Z)$ equal to the set of subspaces $V \seq X \times_S Z$ which are proper and flat over $Z$. The following key result is due to Douady in the absolute setting and Pourcin in the relative case:
\begin{thm} [\cite{douady}, \cite{pourcin}]
Let $X \to S$ be a morphism of complex spaces. Then the functor $\text{Hilb}(X /S)$ is representable by a complex space over $S$.
\end{thm}
The complex space $\text{Hilb}(X /S)$ is called the \emph{Douady space}. The construction of $\text{Hilb}(X /S)$ has a very different flavour to the algebraic construction of the Hilb functor. 

Let $f: X \to Y$ be a morphism between compact complex spaces. A deformation of $f$ over the germ representative $(S,*) \in\mathcal{A}n_{*}$ is defined to be a diagram 
$$\xymatrix{
\mathcal{X} \ar[r]^{F} \ar[d]^g
&\mathcal{Y} \ar[ld]^h \\
 S
}$$
where $F,g,h$ are morphism with $g,h$ flat and proper, and such that the diagram reduces to $f$ at $* \in S$. Two deformations of $f$ are equivalent if they coincide in an open set about $* \in S$. Morphisms between deformations of $f$ are defined in the natural way, \cite[I.3.D]{flenner-ueber}. Define the functor 
$$\text{Def}(X,f,Y) \; : \;  \mathcal{A}n_* \to \{ \text{Sets} \}$$
by sending $S$ to all equivalence classes of deformations of $f$ over $S$. This functor admits a deformation-obstruction theory in the sense of 
\cite[I.4.D]{flenner-ueber}, \cite[\S 6]{buchweitz-flenner}. For any deformation $a \in \text{Def}(X,f,Y)(S)$, this deformation obstruction theory is specified by functors
$$ T^{i}_{f}(a) \: : \: \text{Coh}_{\mathcal{O}_{S,*}} \to \text{Coh}_{\mathcal{O}_{S,*}} $$ between the categories of coherent $\mathcal{O}_{S,*}$ modules, see also \cite[\S 7.4]{buchweitz-flenner}. We define $$T^{i}_{f}:= T^{i}_{f}(a_0)(\C);$$
where $a_0$ is the trivial deformation of $f$; this is a $\C$ vector space. Note that there is a canonical isomorphism $T^{i}_{f}(a)(\C)\simeq T^{i}_{f}$ for any $i$ and \emph{any} deformation $a \in \text{Def}(X,f,Y)(S)$, \cite[Satz I.3.22]{flenner-ueber}. 

We call an element $a \in \text{Def}(X,f,Y)(A)$ \emph{versal} if for any $b \in \text{Def}(X,f,Y)(B)$ and a pair of a closed immersion $i:  (C,*) \hookrightarrow (B,*)$ of germs and a morphism $\gamma: (C,*) \to (A,*)$ of germs with $i^*b \simeq \psi^* a$, then there exists a morphism $$h: (B,*) \to (A,*)$$ of germs such that $h^{*}(a) \simeq b$ and $h \circ i = \gamma$. A versal deformation $a$ is called \emph{semiuniversal} if, in the above situation, the differential 
$$dh: T(B,*) \to T(A,*)$$
is independent of the choice of $h$. 
The following is \cite[II.8.3]{flenner-ueber}:
\begin{thm}[Flenner] \label{semiuniv-maps}
If $f: X \to Y$ is a morphism between compact complex spaces, then $f$ admits a semiuniversal deformation.
\end{thm}
\begin{proof}
From \cite[Satz III.8.1]{flenner-ueber} or \cite[Satz 5.2]{flenner-kriterium} it suffices to show that there exists a versal deformation of $f$.
The graph of $f$ gives a closed immersion $\Gamma \seq X \times Y$. \text{Pic}k semiuniversal deformations $\mathcal{X} \to A_1$, $\mathcal{Y} \to A_2$ of $X$, $Y$. There is an open subset $U \to A_1 \times A_2$ of the relative Douady space $Hilb_{A_1 \times A_2}(\mathcal{X} \times \mathcal{Y})$ parametrising morphisms $f_t: X_t \to Y_t$, by identifying the morphisms with their graph. The pull-back of the universal family over $Hilb_{A_1 \times A_2}(\mathcal{X} \times \mathcal{Y})$ via 
$U \seq Hilb_{A_1 \times A_2}(\mathcal{X} \times \mathcal{Y})$ gives a versal family.
\end{proof}
The above semiuniversal deformation remains versal in a Zariski-open subset about $* \in S$, \cite[Kor.\ I.4.11, Satz III.8.1]{flenner-ueber}, also see\cite{flenner-kriterium}. 

Fix a semiuniversal deformation $a \in \text{Def}(X,f,Y)(S)$ of $f$. As in the case of deformations of compact complex spaces, one can give a local description of the base $S$ of the semiuniversal deformation $a$. The following dimension bound is a direct application of \cite[Cor.\ 6.11]{buchweitz-flenner}.
\begin{prop}
Let $f: X \to Y$ be a morphism between compact complex spaces, and let $$a \in \text{Def}(X,f,Y)(S)$$ be a semiuniversal deformation. Then
$$ \dim (S,*) \geq \dim T^1_{f}-\dim T^2_{f}.$$
\end{prop} 
\begin{rem}
In fact, for the above dimension count, $a$ need only be a \emph{formal} semiuniversal deformation, as defined in \cite[\S 6]{buchweitz-flenner}.
\end{rem}

We now consider the case $f: D \to X$ is a stable map from a nodal curve to a projective K3 surface, with $f(D)$ one-dimensional. In this case, we have $$T^{i}(D)= \text{Ext}^{i}_{D}(\Omega_{D}, \mathcal{O}_{D}), \; \; \; \; T^{i}(X)=H^{i}(X,T_{X}),$$
where the first equality is just the statement that the cotangent complex of a nodal curve $D$ is the sheaf $\Omega_{D}$.
There is a long exact sequence
\begin{equation} \label{rans-sequence}
0 \to T^0(D / X) \to T^0_{f} \to T^0(X) \to T^1(D / X) \to T^1_{f} \to T^1(X) \to \ldots
\end{equation}
where the $\C$ vector spaces $T^i(D / X)$ fit into the long exact sequence
\begin{equation} \label{relative-morph-les}
\begin{split}
0 &\to T^0(D / X) \to T^0(D) \to H^0(D, f^*T_{X}) \to T^1(D / X) \to T^1(D) \\
&\to H^1(D, f^*T_{X}) \to \ldots
\end{split}
\end{equation}
see for example \cite[\S 4.3, \S 4.5]{palamodov}, \cite[Satz I.3.4]{flenner-ueber}, \cite[5.13]{bingener} (for the second sequence) and \cite[Appendix C]{greuel-introduction}. \footnote{A similar sequence appears in the well-known paper \cite{ran-maps} which takes a very different approach using non-commutative algebra. Its appears, however, that Ran's use of the cotangent \emph{sheaf} needs to be replaced with the cotangent \emph{complex} for his results to hold in the claimed generality (this makes no difference in our situation).}

The following calculation is taken from \cite[Prop.\ 4.1]{kemeny-thesis}.
\begin{prop} \label{main-good-dimension-count}
Let $f: D \to X$ be a stable map from a nodal curve to a projective K3 surface, with $f(D)$ one-dimensional.  Then we have the following bounds
\begin{align}
 \dim T^1_{f}-\dim T^2_{f} & = 19+p(D) \\
  \dim T^1_{D/X}-\dim T^2_{D/X} & \geq p(D) -1
 \end{align}
 where $p(D)$ denotes the arithmetic genus of $D$.
\end{prop}
\begin{proof}
Since $f$ is stable, $T^0(D) \to H^0(D, f^* T_{X})$ is injective. Since we further have $T^0(X)=0$, this gives $T^0_{f}=0$.
To compute this, we follow \cite[p.\ 62ff.\ \!]{arakol}. A power of the ample bundle $\omega_{D} \otimes f^* \mathcal{O}_{D}(3)$ is very ample and hence it induces an embedding
$$ j \; : \; D \hookrightarrow \proj^N$$
for some integer $N$. The diagonal $\delta:=(f,j)$ then gives a second closed immersion
$$ \delta \; : \; D \hookrightarrow \proj^N \times X.$$
 As $D$ is nodal, $D$ is in particular a reduced, local complete intersection curve. Thus these two embeddings induce two normal bundles on $D$, namely $N_{D, \proj^N} $ and $N_{D, \proj^N \times X}$. We have a short exact sequence
 $$0 \to N^*_{D, \proj^N \times X} \to f^* \Omega_{X} \oplus (\Omega_{\proj^N})_{|_{D}} \to \Omega_{D} \to 0 .$$ Taking the long exact sequence of $\text{Ext}(\; , \mathcal{O}_{D})$ gives
 $$ \dim T^0_{D}-\dim T^1_{D} + \dim T^2_{D}= \chi(f^*T_{X})+\chi((T_{\proj^N})_{|_{D}})-\chi(N_{{D}, \proj^N \times {X}}).$$
 In fact, $T^2_{D}=H^1(\mathcal{E}xt^1(\Omega_{D}, \mathcal{O}_{D}))$ by the local-to-global spectral sequence for $Ext$, and since $\mathcal{E}xt^1(\Omega_{D}, \mathcal{O}_{D})$ is supported at the nodes of $D$ we have $T^2_{D}=0$, but we will not use this.
 
 From equation (\ref{relative-morph-les}) we get 
 \begin{align*}
 \dim T^0(D / X)-\dim T^1(D / X)+\dim T^2(D / X)
 =\chi((T_{\proj^N})_{|_{D}})-\chi (N_{D, \proj^N \times {X}})
 \end{align*}
 Combining this with the exact sequence (\ref{rans-sequence}), then gives
 $$\dim T^1_{f}-\dim T^2_{f} = \chi (N_{{D}, \proj^N \times {X}})-\chi((T_{\proj^N})_{|_{D}})+20.$$ It now suffices to determine $\deg N_{D, \proj^N \times {X}}$, as then $\chi (N_{D, \proj^N \times {X}})$ and $\chi((T_{\proj^N})_{|_{D}})$ can be readily computed using Riemann--Roch and the Euler sequence. There is a short exact sequence
 $$0 \to f^* T_{X} \to N_{D, \proj^N \times {X}} \to N_{D, \proj^N} \to 0 $$ and thus
 $\deg N_{D,\proj^N \times {X}}=\deg N_{D, \proj^N}$ since $X$ is a K3 surface. For an embedded local complete intersection curve
 $ D \seq \proj^N$, we have 
 $$\deg N_{D, \proj^N}=(N+1) \deg D+2p(D)-2 $$
 from \cite[p.\ 63]{arakol}, which allows one to finish the computations.
\end{proof}

Following \cite[\S 7.4]{buchweitz-flenner}, we now define a ``restricted" version of the $Def(X,f,Y)$ functor, which only allows chosen deformations of $Y$. Let $f: X \to Y$ be a morphism of compact complex spaces. Fix a representative of a complex space germ $(B,*) \in \mathcal{A}n_*$ and a flat deformation $h: \mathcal{Y} \to B$ of a compact complex space $Y$. Let  $\mathcal{A}n_{(B,*)}$ be the category of germs over $(B,*)$. Define
$$\text{Def}_{h}(X,f,Y) \; : \;  \mathcal{A}n_{(B,*)} \to \{ \text{Sets} \}$$ as follows. For any morphism of germs $j: (A,*) \to (B, *)$, let
$h_A: \mathcal{Y}_A \to A$ be the pull-back of $\mathcal{Y} \to B$ under $j$. Then $\text{Def}_{h}(X,f,Y)(A)$ is defined to be the subset of $Def(X,f,Y)(A)$ consisting of deformations of $f$ such that the induced deformation of $Y$ is $h_A: \mathcal{Y}_A \to A$. In other words, $\text{Def}_{h}(X,f,Y)(A)$ is the set of equivalence classes of diagrams of the form 
$$\xymatrix{
\mathcal{X} \ar[r]^{F} \ar[d]^g
&\mathcal{Y_A} \ar[ld]^{h_A} \\
 A
}$$
which reduce to $f$ at the distinguished point $* \in A$. Morphisms between objects are defined in the natural way, as are versal and semiuniversal deformations. The proof of Theorem \ref{semiuniv-maps} similarly shows that $\text{Def}_{h}(X,f,Y)$ has a semiuniversal deformation. 

The next lemma is rather obvious.
\begin{lem} \label{silly-lem-one}
Assume the compact complex space $Y$ admits a \emph{universal} deformation $h: \mathcal{Y} \to B$ over a germ $(B,*)$. Then any semiuniversal deformation $a \in \text{Def}(X,f,Y)(S)$ may be considered as a semiuniversal deformation for the deformation functor $\text{Def}_h(X,f,Y)$.
\end{lem}
The following lemma is straightforward.
\begin{lem} \label{silly-lem-two}
Let $j: (B,*) \to (A,*)$ be a morphism of complex germs, let $h: \mathcal{Y} \to A$ be a flat deformation of a compact complex space $Y$, and let 
$h_B: \mathcal{Y}_B \to B$ be the pullback of $h$ under $j$. Let $a \in \text{Def}_h(X,f,Y)(S)$ be a semiuniversal element, over the base $k: S \to A$. Let $pr_1: S \times_A B \to S$ denote the projection. Then $$b:=pr_1^*a \in \text{Def}_{h_B}(X,f,Y)(S\times_A B)$$ provides a semiuniversal deformation for $\text{Def}_{h_B}(X,f,Y)$.
\end{lem}
\begin{proof}
Versality is an easy application of the universal property of the fibre-product, \cite[Def.\ I.1.46]{greuel-introduction}. We omit the proof. We denote by $T_{\epsilon} \in \mathcal{A}n_*$ the double point at $*$; i.e.\ the germ associated to the analytic $\C$-algebra $\C \{t \}/ (t^2)$; this has a one-dimensional tangent space spanned by $\epsilon$. For any complex space germ $(W,*)$ there is a canonical identification between the tangent space $T(W,*)$ and the vector space of morphisms $r: T_{\epsilon} \to W$ given by $r \mapsto dr(\epsilon)$. Let $c \in \text{Def}_{h_B}(X,f,Y)(C)$ be a deformation over a base $l: C \to B$, $v: T_\epsilon \to C$ any tangent vector and let $q_i: C \to S\times_A B$ for $i=1,2$ be $B$-morphisms such that $q_i^*b=c$. We need to show that $q_1 \circ v=q_2 \circ v$ (or equivalently $dq_1 (v)=dq_2(v)$, identifying $v$ with a tangent vector).

We have $q_i^*pr_1^*a=q_i^*b=c$ for $i=1,2$. Since $c$ is a semiuniversal element for $\text{Def}_h(X,f,Y)(S)$, we have $pr_1 \circ q_1 \circ v=pr_2 \circ q_2 \circ v$. Set $\alpha:=pr_1 \circ q_r \circ v \; : \;T_{\epsilon} \to S$ and set $\beta:=l \circ v \; : \; T_{\epsilon} \to B$. Then both $q_1 \circ v$ and $q_2 \circ v$ fit into the commutative diagram:
$$\xymatrix{
T_{\epsilon} \ar@/_/[ddr]_{\beta} \ar@/^/[drr]^{\alpha} \ar[dr]^{q_i \circ v} \\
& S \times_A B \ar[d]^{pr_2} \ar[r]_{pr_1}
& S \ar[d]^k  \\
& B \ar[r]_j &A }.$$ The universal property of the fibre product then implies $q_1 \circ v=q_2 \circ v$.
\end{proof}

We now restrict to our setting. Suppose $h: \mathcal{X} \to B$ is a flat family of deformations of the K3 surface $X$. Let $f: D \to X$ be a stable map with $f(D)$ one-dimensional.
There is a deformation theory of $\text{Def}_h(D,f, \mathcal{X})$, relative to $h$, described by the vector spaces $T^i(D/X)$ from Equation (\ref{relative-morph-les}). This gives the following, see \cite[Cor.\ 6.11, \S 7.4]{buchweitz-flenner} (also compare with the discussion in \cite[\S 2]{huy-kem}).
\begin{prop} \label{relative-analy-dimct}
Let $h: \mathcal{X} \to B$ is a flat family of deformations of the K3 surface $X$ and let $a \in \text{Def}_h(D,f, \mathcal{X})(S)$ be a semiuniversal element.
 Then
$$ \dim (S,*) \geq \dim B+\dim T^1(D/X)-\dim T^2(D/X).$$
\end{prop} 
We have implicitly used there the fact that, in the notation of Buchweitz--Flenner: $$\dim T^i_{\widetilde{\mathcal{D}} / \widetilde{\mathcal{X}}}(\mathcal{O}_D)=\dim T^{i}_{D/X}(\mathcal{O}_D)$$ for any deformation $F: \widetilde{\mathcal{D}} \to \widetilde{\mathcal{X}}$ of $f$. This follows, for instance, from the long exact sequence \cite[Satz I.3.4]{flenner-ueber} and base change \cite[Satz I.3.22]{flenner-ueber}, \cite[(5.14)]{bingener}.

We now return to the Deligne--Mumford stack $\mathcal{W}_{g,k}^n$. By seeing this as an algebraic stack, it comes with the ``standard" deformation-obstruction theory \cite{behrend-fantechi}, \cite{behrend}. This leads to the bound 
\begin{equation} \label{bad-bound}
\dim \mathcal{W}_{g,k}^n \geq 18+p(g,k)-n
\end{equation}
 which is perhaps most easily seen from the description of $\mathcal{W}_{g,k}^n$ as a quotient of a relative Hilbert scheme, cf.\ \cite[\S 10]{arakol}. 

It is a well-known fact that Bound \ref{bad-bound} is not optimal. For instance, applying it to the case of rational curves in primitive classes, i.e.\ $k=1$, $p(g,k)=n$, it would suggest that there is an $18$-dimensional divisor in the space $\mathcal{B}_g$ of primitively polarised K3 surfaces admitting a rational curve in the linear system induced by the polarisation. This does not sit well with the result of Mori--Mukai stating that \emph{all} primitively polarised K3 surfaces admit (possibly non-reduced) rational curves in the class of the polarisation, \cite[Appendix]{mori-mukai}. The solution to this problem comes by considering \emph{all} deformation of the K3 surface $X$ and not merely the algebraic ones.

We firstly need a simple lemma about complex spaces.
\begin{lem} \label{local-parametrisation-theorem}
Let $S$ be an irreducible, reduced complex space, and $x \in S$ a point. There exists a holomorphic map 
$$z \; : \; \Delta \to S $$
where $\Delta \seq \C$ is a disc about the origin, such that $z(0)=x$ and such that there exists a nonempty, analytically open set $\Delta^0 \seq \Delta$ with $z(y)$ lying in the smooth locus of $S$ for all $y \in \Delta^0$.
\end{lem}
\begin{proof}
Let $m=\dim S$. If $m=0$ the result is trivial, whereas for $m = 1$ the result follows from the existence of normalisations of reduced complex spaces, \cite[Thm.\ I.1.95]{greuel-introduction}. Assume $m \geq 2$. By the Local Parametrisation Theorem (e.g.\ \cite[Cor.\ 4.6.7]{taylor-several} or \cite{gunning}), there exists an open neighbourhood $U \seq S$ about $x$, an open neighbourhood $V \seq \C^m$ about the origin and a finite morphism $\pi: U \to V$, such that $\pi(U)$ contains an open subset about the origin and $\pi(x)=0$. \text{Pic}k a point $v \in V$ near the origin and outside of the branch locus (which is a proper closed analytic set), and a line $L$ from the origin to $v$. By considering $L$ as the zeroes of a holomorphic function $\C^m \to \C^{m-1}$ and applying \cite[Prop.\ I.1.85]{greuel-introduction}, there exists an irreducible, one-dimensional analytic subset $\widetilde{\Delta} \seq U$ containing $x$ and such that $\pi(\widetilde{\Delta}) \seq L$. Since the restriction $\pi_{\widetilde{\Delta}}: \widetilde{\Delta} \to L$ has isolated fibres it is locally finite, \cite[Thm.\ I.1.66]{greuel-introduction}, and hence $\pi(\widetilde{\Delta})$ contains an open subset of $L$ about the origin. In particular, $\pi(\widetilde{\Delta})$ contains points outside the branch locus. Let $\widetilde{\Delta}_{red}$ denote the reduction, and then take $\Delta$ to be an open disc over $x$ in the normalisation of $\widetilde{\Delta}_{red}$. Then $\Delta$ has the desired property.
\end{proof}
We now need a proposition about the topology of nodal curves. The main idea we need is that of oriented real blow-ups. These occur naturally in the study of degenerations of surfaces, \cite[Pg.\ 33-35]{persson}. They have become a standard tool in logarithmic geometry and Teichm\"{u}ller theory, see e.g.\ \cite[Pg.\ 404]{kawamata-namikawa}, \cite{kato-log}, \cite[\S 2]{looijenga-cellular}, \cite[VI]{hubbard-compact}. We follow the presentation from \cite[\S 8.2]{abramovich-gillam}; also see \cite[X.9]{arbarello-II} and the unpublished manuscripts \cite[\S 2]{kawamata-fiber}, \cite{gillam}. 

Let $X$ be a topological space, $\pi: L \to X$ a complex line bundle and $s: X \to L$ a section. Choose local trivialisations $(\pi, \phi): L \to X \times \C$. We start with the set
$$ B_{L,s}(X):=\{ l \in L \; | \; |\phi(l)| \cdot (\phi s \pi)(l)= \phi l \cdot |(\phi s \pi)(l)| \}.$$ This space is invariant under the $\C^*$ action and glues to give a well-defined topological space over $X$. Further, it contains the zero-section and is invariant under the natural $\mathbb{R}_{>0}$ action. Let $B^*_{L,s}(X) \seq B_{L,s}(X)$ denote the complement of the zero section and define the \emph{simple real blow-up},\footnote{This is denoted $Blo_{L,s}(X)$ in \cite{abramovich-gillam}. We are choosing the notation to emphasise the difference with the full oriented real blow up.} denoted $SiBl_{L,s}(X)$, by
$$SiBl_{L,s}(X) := B^*_{L,s}(X) /  \mathbb{R}_{>0}.$$ This comes with a proper morphism of topological spaces $\pi: SiBl_{L,s}(X)  \to X$, which is a homeomorphism away from the zeroes of $s$ and an oriented $S^1$ bundle over $Z(s)$. We have the following functorialities. Let $p: Y \to X$ be a continuous morphism, then there is a homeomorphism
$$SiBl_{L,s}(X) \times_X Y \simeq SiBl_{p^*L,p^*s}(Y) .$$ Next let $L_1, \ldots, L_n$ be complex line bundles on $X$ with sections $s_i$, and $L=L_1 \otimes \ldots \otimes L_n$ with section $s=s_1 \otimes \ldots \otimes s_n$. Then the multiplication
$$L_1 \times_X L_2 \ldots \times_X L_n \to L $$
induces a map 
$$ SiBl_{L_1,s_1}(X) \times_X \ldots \times_X SiBl_{L_n,s_n}(X) \to SiBl_{L,s}(X) .$$
Now assume that $X$ is a complex manifold and that $D$ is a smooth divisor, defined by a section $s$ of a holomorphic line bundle $L$. In this case $SiBl_{L,s}(X)$
has the structure of a real analytic manifold with corners (for more on this notion see \cite[Pg.\ 363]{lee-intro}). If $U \seq \C^n$ is a local chart for $X$, and if $s$ is locally the function $s : U \to \C$ on $U$, then we can describe $SiBl_{L,s}(X)$ locally as the set
$$\{ (x, \tau) \in U \times S^1 \; | \; s(x)=|s(x)| \tau\},$$
where $S^1$ denotes the unit circle in the complex plane. Now assume $X$ is a complex manifold and $D=D_1 \cup \ldots \cup D_k$ is a union of smooth divisors with normal crossings. Let $L = \mathcal{O}(D)$, and $L_i= \mathcal{O}_{D_i}$, and let $s_i \in H^0(X,L_i)$ define $D_i$, with $s=s_1 \otimes \ldots \otimes s_k$. We define the \emph{oriented real blow up} of $X$ at $D$ to be 
$$ Bl_{L,s}(X):= SiBl_{L_1,s_1}(X) \times_X \ldots \times_X SiBl_{L_k,s_k}(X).$$
In fact, if $D$ is any normal crossing divisor, meaning it has the above form locally but not necessarily globally, then can one define the real oriented blow up 
by glueing together charts of the above form, but we will not need this. $ Bl_{L,D}(X)$ is an analytic manifold with corners, which comes with a proper map $p: Bl_{L,D}(X) \to X$ (use \cite[Prop.\ 10.1.5 d]{bourbaki-top} to see properness). It has the following local description: for a small open $U \seq X$, $p^{-1}(U)$ is given by
$$\{ (x, \tau_1, \ldots, \tau_k) \in U \times (S^1)^k \; | \; s_i(x)=|s_i(x)| \tau_i \; \text{for all $1 \leq i \leq k$}\}.$$ If $n= \dim (U)$, then these local patches are real-analytically isomorphic to
$$(S^1)^k \times [0, \epsilon)^k \times \Delta_{\epsilon}^{n-k},$$ where $\Delta_{\epsilon}^{n-k}$ is an $\epsilon$-disc in $\C^{n-k}$, see \cite[Pg.\ 150]{arbarello-II}.
Note that $ Bl_{L,s}(X)$ is related but different to $SiBl_{L,s}(X)$; indeed if all divisors pass through the origin then the first has central fibre $(S^1)^k$, whereas the second has $S^1$ as the central fibre.

The payoff for all these definitions is that we can now state the following remarkable result, due to E.\ Looijenga.
\begin{thm} [Looijenga] \label{looi-thm}
Let $\mathcal{C} \to T$ be a Kuranishi family of stable marked curves. Let $0 \in B$ be the origin, and let $p_1, \ldots, p_k$ be the nodes in the central fibre. Let $D_1, \ldots, D_k$ be the divisors in $T$ corresponding to deformations which preserve the node $p_k$. Let $D$ be the normal crossing divisor $D_1 \cup \ldots \cup D_k$ and $L=\mathcal{O}_T(D)$ with $s \in H^0(T,L)$ the section defining $D$. Then there is an analytic manifold with corners $Z$, an analytic fibration $Z \to Bl_{L,s}(T)$ with fibre a fixed orientable Riemann surface $\Sigma$, and a surjective, continuous map $\lambda: Z \to \mathcal{C}$ giving a commutative diagram
$$\xymatrix{
Z \ar[r]^{\lambda} \ar[d]
&\mathcal{C} \ar[d]^g \\
Bl_{L,s}(T) \ar[r]^h & T \;.
}$$ For each point $p \in Bl_{L,s}(T)$, the induced map on fibres
$$\lambda_p \; : \; \Sigma \simeq Z_p \to \mathcal{C}_p $$
is a diffeomorphism away from the nodes of $\mathcal{C}_p $.
\end{thm}
\begin{proof}
See \cite[\S 2]{looijenga-cellular}. For the analyticity statements, see \cite[Ch.\ X.9, XV.8]{arbarello-II}.
\end{proof}
In other words, one can perform simultaneous \emph{real} resolutions of families of nodal curves, after modifying the base, in situations where one could not possibly perform simultaneous resolutions of singularities in the complex category. This is illustrated in Figure \ref{pinching}, which illustrates a family of smooth quadrics degenerating to a nodal quadric. By Teichm\"{u}ller theory, all degenerations to nodal curves occur by pinching finitely many loops in this fashion.

Whilst families of cycles work well in the algebraic setting (see\cite[Ch.\ 10]{fulton}), they seem to be much harder to work with in the case of analytic families of singular complex spaces. We will circumvent this problem using Theorem \ref{looi-thm}. Let $f: D \to X$ be a morphism from a nodal curve to a K3 surface, and let $\tilde{D}$ be the normalisation and $\tilde{f}: \tilde{D} \to X$ the induced map. We define $$f_* D:= \tilde{f}_* \tilde{D} \in H_2(X, \mathbb{Z}).$$ 

\newpage
\begin{figure}[H] \label{pinching}
\begin{center}
\includegraphics[scale=0.4]{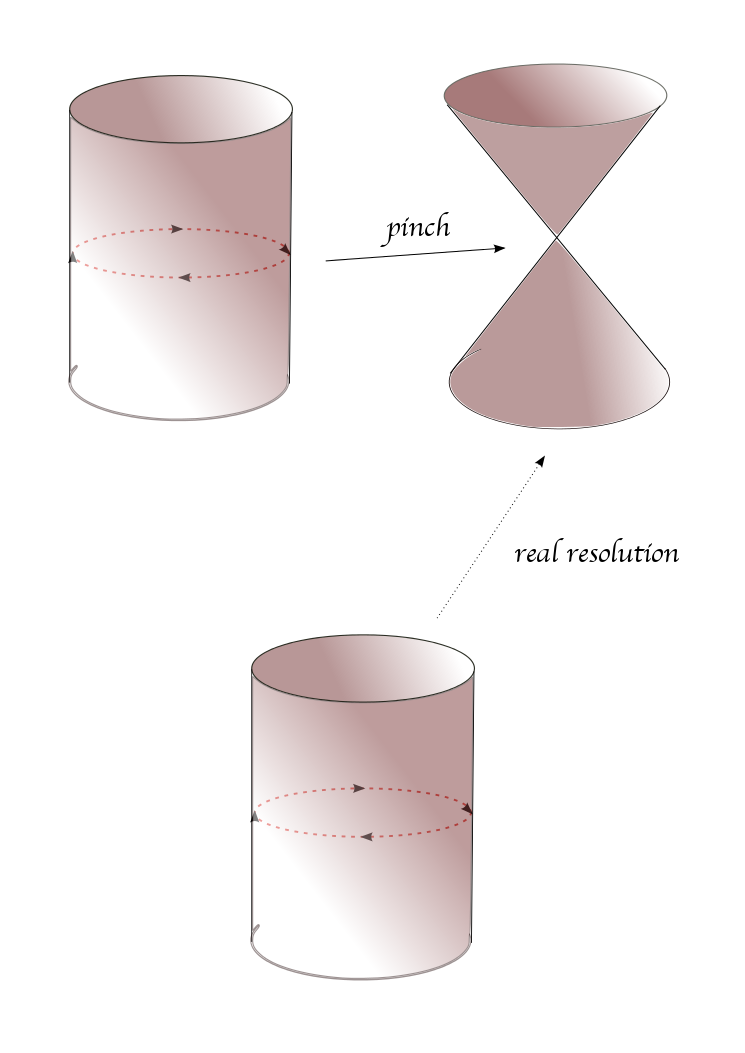}
\caption{A degeneration of a smooth quadric to a nodal curve via pinching a loop. The lower cylinder illustrates the real resolution locally. }
\end{center}
\end{figure}
\newpage

Before we proceed we will prove two easy lemmas from topology. The first one is standard. Let $X$ be an $n$-dimensional, compact topological manifold with boundary and connected components $X_i$, $1 \leq i \leq k$. Recall that the compact manifold with boundary $X$ is called orientable if the (possibly non-compact) open manifold $X-\partial X$ is orientable. If we assume that $X$ is orientable, then Lefschetz duality gives isomorphisms $H_n(X_i, \partial X_i; \mathbb{Z}) \simeq  H^0(X_i, \mathbb{Z}) \simeq \mathbb{Z}$, \cite[Thm.\ 3.43]{hatcher}. In particular, $$H_n(X, \partial X; \mathbb{Z}) \simeq \mathbb{Z}^k.$$ The point of the next lemma is to describe this isomorphism slightly differently.
\begin{lem}. 
Let $X$ be an $n$-dimensional, compact topological manifold with boundary and connected components $X_i$, $1 \leq i \leq k$. For each $i$, let $x_i \in X_i$ be a point in the interior
and let $U_i \seq X_i$ be a small, contractible, open neighbourhood of $x_i$. Assume $X$ is orientable. Then there are isomorphisms
$$H_n(X, \partial X; \mathbb{Z}) \simeq \bigoplus_{1 \leq i \leq k} H_n(X_i, X_i-x_i; \mathbb{Z}) \simeq \bigoplus_{1 \leq i \leq k} H_n(U_i, U_i-x_i; \mathbb{Z}), $$
induced by the obvious inclusions.
\end{lem}
\begin{proof}
We follow \cite[p.\ 252ff ]{hatcher}.  By excision $H_n(X_i, X_i-x_i; \mathbb{Z}) \simeq H_n(U_i, U_i-x_i; \mathbb{Z}) \simeq \mathbb{Z}$. It remains to show that inclusion induces an isomorphism $$H_n(X_i, \partial X_i ; \mathbb{Z}) \simeq H_n(X_i, X_i-x_i; \mathbb{Z}).$$ By \cite[Prop.\ 3.42]{hatcher}, there is an open neighbourhood $V_i \seq X_i$ containing $\partial X_i$, together with a homeomorphism $\psi_i: V_i \to \partial X_i \times [0,1)$ such that $\psi_i(\partial X_i )=\partial X_i \times \{0 \}$. We may choose $V_i$ so that $x_i \notin V_i$. By excision
$$H_n(X_i, \partial X_i ; \mathbb{Z}) \simeq H_n(X_i, \partial X_i \times [0,1) ; \mathbb{Z}) \simeq H_n(X_i-\partial X_i, \partial X_i \times (0,1); \mathbb{Z}).$$ Since the manifold $X_i-\partial X_i$ is orientable, applying \cite[Lemma 3.27(a)]{hatcher} to the compact set $A:=X_i-V_i \seq X_i-\partial X_i$ shows that the natural map
$$  H_n(X_i-\partial X_i, \partial X_i \times (0,1); \mathbb{Z}) \to H_n(X_i-\partial X_i, (X_i-\partial X_i)-x_i; \mathbb{Z}) \simeq \mathbb{Z} $$
induced by inclusion is surjective. Since Lefschetz duality gives $H_n(X_i, \partial X_i; \mathbb{Z}) \simeq  H^0(X_i; \mathbb{Z}) \simeq \mathbb{Z}$, this completes the proof. 
\end{proof}

Let $X$ be $n$-dimensional, compact topological manifold with boundary and connected components $X_i$, $1 \leq i \leq k$. Assume further that $X$ is orientable. An \emph{orientation class} is then an element
$\Theta \in H_n(X, \partial X; \mathbb{Z}) $ which can be written as a sum of generators of $H_n(X_i, \partial X_i; \mathbb{Z}) \simeq \mathbb{Z}$ for each $i$. Note that we have $2^k$ choices for the orientation class, corresponding to a choice of sign for each component.
\begin{lem} \label{orientation-topological}
Let $A$ be an orientable $n$-dimensional, compact topological manifold with boundary and finitely many connected components and let $B$ be an orientable, $n$-dimensional, compact manifold without boundary (possibly disconnected). Assume we have a surjective, continuous map $f: A \to B$ such that $f(\partial A) \seq B$ is a closed set and the restriction
$$f: A-\partial A \to B-f(\partial A) $$
is a homeomorphism. Assume that each connected component of $f(\partial A)$ is homeomorphic to a topological manifold of dimension strictly less than $n$. Finally assume that if
$$\delta \; : \; H_n(A, \partial A; \mathbb{Z}) \to H_{n-1}(\partial A, \mathbb{Z})  $$
is the boundary map then 
$$f_* \circ \delta (\Xi)=0 \in H_{n-1}(f(\partial A), \mathbb{Z}) $$
for some orientation class $\Xi \in H_n(A, \partial A; \mathbb{Z})$.
Then we can choose an orientation $\Theta \in H_n(B, \mathbb{Z})$ such that
$$ f_* \Xi = j \Theta \in H_n(B, f(\partial A); \mathbb{Z}),$$
where $j: H_n(B, \mathbb{Z}) \to H_n(B, f(\partial A); \mathbb{Z})$ is the natural inclusion.
\end{lem}
\begin{proof}
By the assumption  $f_* \circ \delta (\Xi)=0$, there is a
$\Theta \in H_n(B, \mathbb{Z})$ with $ f_* \Xi = j \Theta$. It remains to show that $\Theta$ is an orientation class, or equivalently that for each connected component
$B_i$ of $B$, there is an element $x_i \in B$ such that the image of $\Theta$
under the natural map $H_n(B, \mathbb{Z}) \to H_n(B, B-x_i; \mathbb{Z})$ is a generator. For any component $B_i$ of $B$, let $x_i \in B_i$, $x_i \notin f(\partial A)$
and pick $y_i \in f^{-1}(x_i)$. Now consider the natural commutative diagram
$$\xymatrix{
 \mathbb{Z} \{ \Xi \} \ar[r] \ar[d]_g
&H_n(A, A-y_i; \mathbb{Z}) \ar[d] \\
\mathbb{Z} \{ \Theta \} \ar[r] &H_n(B, B-x_i; \mathbb{Z}).
}$$ For small open subsets $U_i$ resp.\ $V_i$ about $y_i$ resp.\ $x_i$, the inclusions give isomorphisms
$H_n(U_i, U_i-y_i; \mathbb{Z}) \simeq H_n(A, A-y_i; \mathbb{Z}) $ and $H_n(V_i, V_i-x_i; \mathbb{Z}) \simeq H_n(B, B-x_i; \mathbb{Z}) $. 
Thus the assumption that $f$ is a homeomorphism away from $\partial A$ gives that the rightmost vertical map in the above commutative diagram is an isomorphism. Further, the topmost horizontal map takes $\Xi$ to a generator of $H_n(A, A-y_i; \mathbb{Z})$ by the previous lemma. The claim follows.
\end{proof}
We can now prove the previously mentioned result on the topology of stable maps.
\begin{prop} \label{analytic-top}
Let $\Delta \seq \C$ be a disc about the origin and let 
$$\xymatrix{
\mathcal{D} \ar[r]^{F} \ar[d]_g
&\mathcal{X} \ar[ld]^h \\
\Delta
}$$
be a flat family of morphisms, with $g: \mathcal{D}  \to \Delta$ a flat and proper family of nodal curves of genus $p$ and $h: \mathcal{X} \to \Delta$ a proper, smooth family of K3 surfaces. Assume $F_0: \mathcal{D}_0 \to \mathcal{X}_0$ is non-constant. Let $\phi_t: H^2(\mathcal{X}_t, \mathbb{Z}) \simeq H^2(\mathcal{X}_0, \mathbb{Z})$ be a holomorphic family of markings for $t \in \Delta$, with Poincare dual
$\hat{\phi}_t: H_2(\mathcal{X}_0, \mathbb{Z}) \to H_2(\mathcal{X}_t, \mathbb{Z})$. Then $\hat{\phi}_t(F_{*,0} \mathcal{D}_0)=F_{*,t} \mathcal{D}_t \in H_2(\mathcal{X}_t, \mathbb{Z})$ for all $t \in \Delta$.
\end{prop}
\begin{proof}
 It is enough to prove the statement locally on $\Delta$.
Using that a flat morphism of complex spaces is open, \cite[\S 3.19]{fischer}, and the implicit function theorem \cite[Cor.\ 1.1.12]{huybrechts-complex}, we assume we have markings $s_i: \Delta \to \mathcal{D}$, so that $\mathcal{D}$ becomes a family of $m$-pointed stable curves, for some $m$. Let $\pi: \mathcal{C} \to T$ be a Kuranishi family of stable, $m$-pointed genus $p$ curves. Shrinking $\Delta$ if necessary, the universal property of Kuranishi families gives a morphism $j: \Delta \to T$. 

Let $D_i$ be the divisors from Theorem \ref{looi-thm}, set $L_i=\mathcal{O}_T(D_i)$, $s_i \in H^0(L_i)$ the sections defining $D_i$. Then assume $j^*s_i=0$ for $1 \leq i \leq m$ (allowing $m=0$), and $j^*s_i=t^{l_i}$ for $m+1 \leq i \leq k$ (allowing $m=k$). Let $M \in \text{Pic}(\Delta)$ be the line bundle associated to the degree one divisor $0$. By the functorialities of the simple real blow-up and composing with the obvious diagonal maps we have, if $m \neq 0$, a map
$$SiBl_{M,0}(\Delta) \to Bl_{L_1, s_1}(T) \times_T \ldots \times_T Bl_{L_m, s_m}(T)$$ and, if $m \neq k$, a map
$$ Bl_{M,t}(\Delta) \to Bl_{L_{m+1}, s_{m+1}}(T) \times_T \ldots \times_T Bl_{L_k, s_k}(T).$$ 

Let $\tilde{B}$ denote $SiBl_{M,0}(\Delta)$ in case $m=k$, $ Bl_{M,t}(\Delta)$ in case $m=0$ and $SiBl_{M,0}(\Delta) \times_{\Delta} Bl_{M,t}(\Delta)$ if $m \neq 0$, $m \neq k$. Then we have a continuous map
$$\tilde{B} \to Bl_{L,s}(T)$$ and a surjective projection $\pi: \tilde{B} \to \Delta$. Shrinking $\Delta$ if necessary we can explicitly write down the map $\pi$. If $m=k$ it is the projection
$$ S^1 \times \Delta \to \Delta.$$
If $m=0$ it is
\begin{align*}
[0, \epsilon) \times S^1 &\to \Delta_{\epsilon} \\
(x, \tau) &\mapsto x\tau.
\end{align*}
Lastly, if $m \neq 0$, $m \neq k$, it is
\begin{align*}
[0, \epsilon) \times S^1 \times S^1 &\to \Delta_{\epsilon} \\
(x, \tau, \alpha) &\mapsto x\tau.
\end{align*}
From Theorem \ref{looi-thm}, we can form a commutative diagram of topological spaces
$$\xymatrix{
Z_{\widetilde{B}} \ar[r]^{\Gamma} \ar[d]_{h}
&\mathcal{D} \ar[d]^g \\
\widetilde{B} \ar[r]^{\pi} &\Delta
}$$
with all maps surjective. For any point $q \in \pi^{-1}(0)$ there is an open subset $U \seq \widetilde{B}$ about $q$ such that $h^{-1}(U)$ is homeomorphic to $\Sigma \times U$ (as fibre bundles). Let $\Theta_1 \in H_2(\Sigma, \mathbb{Z}) $
 be an orientation class on $\Sigma$. Then, for any $p \in U$ over $t \in \Delta$ sufficiently close to $0$, we have
$$(F_{t} \circ \Gamma_{p})_* \Theta_1= \hat{\phi}_t((F_0 \circ \Gamma_{q})_* \Theta_1) \in H_2(\mathcal{X}_t, \mathbb{Z}),$$
by homotopy invariance. It suffices to show that $$(F_t \circ \Gamma_{p})_*\Theta_1 =\pm F_{*,t} \mathcal{D}_t$$ for any $t \in \Delta$ (including $t=0$) and $p \in \pi^{-1}(t)$. Indeed, assuming the above equation, we can pick the orientation on $\Sigma$ to ensure $(F_0 \circ \Gamma_{q})_* \Theta_1=F_{*,0} \mathcal{D}_0$. We claim that we now have $(F_t \circ \Gamma_{p})_* \Theta_1= F_{*,t} \mathcal{D}_t$ for all $t$ close to $0$. Suppose there is a $t$ with $(F_t \circ \Gamma_{p})_*\Theta_1=- F_{*,t} \mathcal{D}_t$. Let $\omega$ be a K\"ahler form on $\mathcal{X}_0$ and let $\omega_t:=\phi_t^{-1}(\omega)$ be the corresponding K\"ahler form on $\mathcal{X}_t$.  Then
$((F_t \circ \Gamma_{p})_* \Theta_1 \cdot \omega_t)=((F_0 \circ \Gamma_{q})_* \Theta_1 \cdot \omega)=(F_{*,0} \mathcal{D}_0 \cdot \omega)>0$, where the pairing is given by the cap product (since $F_0$ is by assumption non-constant). For any holomorphic map $f: C \to \mathcal{X}_t$ from a smooth, projective curve $C$, one has $(f_*C \cdot \omega_t)=\int_C f^*\omega_t \geq 0$. Thus
$(- F_{*,t} \mathcal{D}_t \cdot \omega_t) \leq 0$ by definition of $F_{*,t} \mathcal{D}_t$, giving a contradiction.

Let $\alpha: \widetilde{D}_t \to \mathcal{D}_t$ be the normalisation and
 let $V \seq \mathcal{D}_t$ be the smooth locus. By definition, $F_{*,t} \mathcal{D}_t= \widetilde{F}_{*,t}(\widetilde{D}_t)$. 
 Let $R \seq \widetilde{D}_t$ be the reduced divisor over the nodes, defined by a section $s \in H^0(\widetilde{D}_t,\mathcal{O}(R))$. Define $$\widetilde{\Sigma}:=Bl_{\mathcal{O}(R),s}(\widetilde{D}_t ).$$ We have a commutative diagram
 $$\xymatrix{
\widetilde{\Sigma} \ar[r]^{ \; \; \; \; \; \nu} \ar[d]_{\mu}
&\widetilde{D}_t \ar[d]^{\widetilde{F}_t } \\
\Sigma \ar[r]^{F_t \circ \Gamma_p} &\mathcal{X}_t
}$$
 from Figure 10, \cite[Pg.\ 152]{arbarello-II}. As is explained in \cite{arbarello-II}, $\widetilde{\Sigma}$ is an oriented manifold with boundary and $\mu$ is obtained by glueing circles in the boundary which have \emph{opposite} orientations. 
 This means that if $\Xi \in H_2(\widetilde{\Sigma},\partial \widetilde{\Sigma}; \mathbb{Z})$ is the orientation class  on the (possibly disconnected) manifold with boundary $\widetilde{\Sigma}$, then $\mu_* \circ \delta (\Xi)=0 \in H_1(\mu(\partial \widetilde{\Sigma}), \mathbb{Z})$, where 
 $$\delta \; : \; H_2(\widetilde{\Sigma} , \partial \widetilde{\Sigma} ; \mathbb{Z}) \to H_1(\partial \widetilde{\Sigma} ; \mathbb{Z}) $$
 is the boundary map. Thus there is an orientation class $\Theta_1 \in H_2(\Sigma, \mathbb{Z}) $ on $\Sigma$ such that we have $i(\Theta_1)=\mu_*(\Xi) \in H_2(\Sigma, \mu(\partial \widetilde{\Sigma}), \mathbb{Z})$, where
 $$i \; : \; H_2(\Sigma, \mathbb{Z})  \to  H_2(\Sigma, \mu (\partial \widetilde{\Sigma}); \mathbb{Z})  $$
 is the natural inclusion, by Lemma \ref{orientation-topological}. Let $\Theta_2 \in H_2(\widetilde{D}_t, \mathbb{Z})$ denote the orientation class on the (not-necessarily connected) complex manifold $\widetilde{D}_t$.  Since
 $H_1(R, \mathbb{Z})=0$, we have
 $i'(\Theta_2)=\pm \nu_*(\Xi)$, where
 $$i': H_2(\widetilde{D}_t, \mathbb{Z}) \to H_2(\widetilde{D}_t, R ; \mathbb{Z})$$ is the natural inclusion (here the sign depends on the choice of orientation on $\Sigma$). Thus we have
 $$(F_t \circ \Gamma_p)_*(i(\Theta_1))=\pm \widetilde{F}_{t,*}(i'(\Theta_2)) \in H_2(\mathcal{X}_t, \tilde{F}_t(R); \mathbb{Z}).$$
 The natural inclusion
 $$j: H_2(\mathcal{X}_t, \mathbb{Z}) \to H_2(\mathcal{X}_t, \tilde{F}_t(R); \mathbb{Z})$$
 is an isomorphism since $\dim(\tilde{F}_t(R))=0$, so this completes the proof.
\end{proof}

\begin{remark}
Assume that the central fibre $\mathcal{X}_0$ is projective and $(F_{0,*} \mathcal{D}_0)^2>0$. Then the above result can also be shown in a more algebro-geometric way. Consider the coherent sheaf of $\mathcal{O}_X$ modules $F_* \mathcal{O}_{\mathcal{D}}$. It is easy to see that this is a flat module over the smooth curve $\Delta$. Indeed this is equivalent to showing that $F_* \mathcal{O}_{\mathcal{D}}$ is torsion free $h^*\mathcal{O}_{\Delta}$ module, which follows from the fact that $\mathcal{O}_{\mathcal{D}}$ is flat (hence torsion free) over $\Delta$. Following \cite[\S 3a]{bismut-gillet-soule}, one can construct the determinant of a coherent sheaf on a complex manifold, in a similar manner to the algebraic construction of \cite{knudsen-mumford-I}.  Applying $F_*$ to the short exact sequence
$$0 \to \mathcal{O}_{\mathcal{D}} \xrightarrow{t}  \mathcal{O}_{\mathcal{D}} \to \mathcal{O}_{\mathcal{D}_t} \to 0$$
 gives an exact sequence of coherent sheaves on $\mathcal{O}_{\mathcal{D}}$
$$0 \to (F_* \mathcal{O}_{\mathcal{D}})_t  \to F_*(\mathcal{O}_{\mathcal{D}_t}) \xrightarrow{\psi} R^1 F_* \mathcal{O}_{\mathcal{D}} \to \ldots$$
Now $Im \psi$ is a coherent sheaf supported on a zero dimensional subset of $\mathcal{D}_t$, so it has trivial determinant.
Thus there is an isomorphism  
$\det ((F_* \mathcal{O}_{\mathcal{D}})_t) \simeq \det F_{*,t} \mathcal{O}_{\mathcal{D}_t}$ of sheaves on $\mathcal{D}_t$. The determinant is constructed by glueing together determinants of local free resolutions, and this gives readily that if $\mathcal{F}$ is a coherent sheaf on $\mathcal{X}$, flat over $\Delta$, then
$$ \det (\mathcal{F}_t) \simeq (\det (\mathcal{F}))_t.$$ This now gives 
$$c_1((\det(F_* \mathcal{O}_{\mathcal{D}}))_t)=c_1(\det F_{*,t} \mathcal{O}_{\mathcal{D}_t}) .$$ Using Ehresmann's theorem, we now see that $\mathcal{X}$ admits a holomorphic line bundle $\mathcal{L}$ with $(\mathcal{L}_t)^2$ constant and equal to $(\det F_{*,0} \mathcal{O}_{\mathcal{D}_0})^2$. If we assume $\mathcal{X}_0$ is projective, then, as we show below, Grothendieck--Riemann--Roch shows that $c_1(\det F_{*,0} \mathcal{O}_{\mathcal{D}_0})$ is identified with $F_{0,*} \mathcal{D}_0$ under Poincar\'e duality. Furthermore, if $(F_{0,*} \mathcal{D}_0)^2>0$, then \emph{all} fibres $\mathcal{X}_t$ must be projective (see e.g.\ \cite[Rem.\ 8.1.3]{huy-lec-k3}), and Grothendieck--Riemann-Roch gives that $c_1(\det F_{*,t} \mathcal{O}_{\mathcal{D}_t})$ is identified with $F_{t,*} \mathcal{D}_t$, allowing one to conclude.

We now give the details of the GRR calculation. Let $g: C \to Y$ be a morphism from a nodal curve to a projective K3 surface with one-dimensional image. We need to show that $(c_1(g_*\mathcal{O}_C) \cdot \alpha)=(g_*C \cdot \alpha)$ for any $\alpha \in \text{Pic}(Y)$. Let $\mu\; : \; \widetilde{C} \to C$ be the normalisation of $C$ and $\tilde{g}: \widetilde{C} \to Y$ the induced morphism. Let $C_1, \ldots ,C_m$ be the connected components of $\widetilde{C}$ and $g_i: C_i \to Y$ the induced maps. Assume $g_i$ has one-dimensional image for $i \leq m'$ and zero-dimensional image for $m'<i \leq m$. We have $(g_*C \cdot \alpha)=\sum_{i=1}^{m'} \deg g^*_i(\alpha)$. There is an inclusion
$$\mathcal{O}_C \to \mu_* \mathcal{O}_{\widetilde{C}} $$ which is an isomorphism away from the nodes; after pushing forward to the surface this gives $c_1(g_*\mathcal{O}_C) \simeq c_1(\tilde{g}_*\mathcal{O}_{\tilde{C}})\simeq \sum_{i=1}^{m'} c_1(g_{i,*}\mathcal{O}_{C_i})$. It thus suffices to consider the case that $C$ is smooth and connected with one-dimensional image. In this situation, GRR applies and gives
\begin{align*}
c_1(g_*\mathcal{O}_C) \cdot \alpha&=[ch(g_* \mathcal{O}_C)td(Y)]_1 \cdot \alpha \\
&=[g_*(ch\mathcal{O}_C td(C))]_1 \cdot \alpha \\
&=g_*(ch\mathcal{O}_C td(C)) \cdot \alpha \\
&=ch\mathcal{O}_C td(C) \cdot g^* \alpha \\
&=\deg g^* \alpha,
\end{align*}
where the second last line follows from the projection formula.
\end{remark}
We can now prove the main result of this section. 
\begin{thm} \label{good-bound}
Every component of the Deligne--Mumford stack $\mathcal{W}_{g,k}^n$ has dimension at least $19+p(g,k)-n$.
\end{thm}
\begin{proof}
Let $\mathcal{A}$ be an \'etale atlas for $\mathcal{W}_{g,k}^n$. This is a scheme of finite type over $\C$, since $\mathcal{W}_{g,k}^n \to \mathcal{B}_g$ is proper, so it suffices to show that, for any closed point $* \in \mathcal{A}$ corresponding to a stable map $f: D \to X$ we have
$$\dim \hat{\mathcal{O}}_{\mathcal{A},*} \geq 19+p(D), $$
where $\hat{\mathcal{O}}_{\mathcal{A},*}$ denotes the completion. Let $\mathcal{B} \to \mathcal{B}_g$ be an \'etale atlas, this comes with a universal family
$ h_p \; : \; \mathcal{X} \to \mathcal{B}$ of polarised K3 surfaces together with a closed immersion 
$\mathcal{X} \hookrightarrow \proj^N_{ \mathcal{B}}$ for some integer $N$.
 Consider the universal family,
$$\xymatrix{
\mathcal{D} \ar[r]^{F} \ar[d]
&\mathcal{X}_{\mathcal{A}} \ar[ld] \\
 \mathcal{A}
}$$ corresponding to $\mathcal{W}_{g,k}^n(\mathcal{A})$. Restricting to the completion 
$\hat{\mathcal{A}}_*$ gives a \emph{formal} semiuniversal deformation $\hat{a}$ of $f$ for the functor $\text{Def}_{h_p}(D,f,X)$ restricted to formal complex spaces. Semiuniversal deformations (either formal or local) are unique up to non-canonical isomorphisms, cf.\ \cite[Lem\ II.1.12]{greuel-introduction}. Thus it suffices to show that there exists a semiuniversal element $b$ of the functor $\text{Def}_{h_p}(D,f,X)$ such that the base of $b$ has dimension at least $19+p(D)$. Let $h: \widetilde{\mathcal{X}} \to \text{Def}(X)$ be the universal deformation of the K3 surface $X$, see \cite[\S 6.2]{huy-lec-k3}. The complex space $\mathcal{B}$ is smooth of dimension $19$, \cite[Pg.\ 150-151]{sernesi-def}. Further, we have a closed immersion of germs about $*$
$$j \; : \; \mathcal{B} \hookrightarrow \text{Def}(X),$$
see \cite[Ch.\ 6.2]{huy-lec-k3}. Choose any semiuniversal deformation of $ c \in \text{Def}(D,f,X)(S)$. By Proposition \ref{main-good-dimension-count}, $$\dim (S,*) \geq 19+p(D).$$ If $c$ corresponds to the diagram,
$$c : \xymatrix{
\mathcal{D}_2 \ar[r]^{F_2} \ar[d]
&\mathcal{X}_{2} \ar[ld] \\
 \mathcal{S}
},$$
let $$\pi: \mathcal{S} \to \text{Def}(X)$$ be the morphism of germs given by the universal property of $\text{Def}(X)$. From Lemma \ref{silly-lem-one}, $\pi$ allows us to consider $c$ as a semiuniversal deformation of $\text{Def}_h(D,f,X)$. By Lemma \ref{silly-lem-two},  $b:=pr_1^*c \in \text{Def}_{h_p}(D,f,X)(S \times_{\text{Def}(X)} \mathcal{B})$ provides a semiuniversal deformation for the functor $\text{Def}_{h_p}(D,f,X)$. 

We need to show that $\dim S \times_{\text{Def}(X)} \mathcal{B}=\dim (S,*)$. It suffices to show that (locally about *), we have the inclusion of sets $$\pi(S) \subseteq \mathcal{B}.$$ By replacing the germ representative $S$ with its reduction $S_{red}$ and base-changing, it suffices to assume $S$ reduced. It further suffices to show
$$\pi(S^{sm}) \subseteq \mathcal{B}, $$ 
where $S^{sm}$ is the smooth locus, since $S^{sm}$ is open and dense and $\mathcal{B} \seq \text{Def}(X)$ is closed. Replacing $S$ with a small open subset about $*$, we may assume $S$ has only finitely many components, and then for each component $S_i$ it suffices to show
$$\pi(S_i^{sm}) \subseteq \mathcal{B}, $$
where $S_i^{sm}$ is the smooth locus. Thus we may assume $S$ is irreducible.

Let $z: \Delta \to S$ be as in Lemma \ref{local-parametrisation-theorem}, with $\Delta(0)=*$, and consider the pull-back family
$$ \xymatrix{
\mathcal{D}_{2,\Delta} \ar[r]^{F_{2,\Delta}} \ar[d]_q
&\mathcal{X}_{2,\Delta} \ar[ld] \\
 \Delta
}.$$ Proposition \ref{analytic-top} now shows that $\pi(z(\Delta)) \seq \mathcal{B}$, since the proposition shows that the 1-1 classes $F_{2, t*}(\mathcal{D}_{2,t})$ on $\mathcal{X}_{2,t}$ deform to $f_{*}(D)$ for $t \in \Delta$ near the origin. Since the set $z(\Delta)$ contains a smooth point of $S$, and the smooth locus of the irreducible complex space $S$ is connected, applying Proposition \ref{analytic-top} again shows $\pi(S^{sm}) \subseteq \mathcal{B}$ as required.
\end{proof}

Let $\mathcal{X} \to S$ be a smooth, projective family of K3 surfaces over a smooth and connected algebraic variety $S$ with a point $0 \in S$, and let $\mathcal{L}$ be an $S$-flat, ample line bundle 
on $\mathcal{X}$. Further, let $\mathcal{M}$ be any effective $S$-flat line bundle 
on $\mathcal{X}$, such that $(\mathcal{M}_t)^2 \geq 2$ for each $t \in S$.
The theorem above can be generalised to the Deligne--Mumford stacks $\mathcal{W} (\mathcal{X}, \mathcal{M},p)$.
Recall the construction of the \emph{twistor-family}, \cite[\S 3]{bryan-leung}. Let $\Lambda_{K3}$ denote the K3 lattice, and consider the period domain
$$ \mathcal{D} = \{ x \in \proj(\Lambda_{K3} \otimes \C) \; : \; (x \cdot \bar{x})>0, (x)^2 =0  \}.$$ Now, let $(X,L)$ be a marked, polarised K3 surface, and let $\Omega \in \mathcal{D}$ denote the corresponding period point, in other words the point corresponding to the one-dimensional vector space $h^{0,2}(X)$ under the marking $H^2(X, \mathbb{Z}) \hookrightarrow \Lambda_{K3}$. Consider the plane
$$ V_{\Omega, c_1(L)} \simeq \proj^2 \seq \proj(\Lambda_{K3} \otimes \C)$$
defined by the span of $\{ \Omega, \overline{\Omega}, c_1(L) \} \seq \Lambda_{K3} \otimes \C$. The intersection of $V_{\Omega, c_1(L)}$ with 
$\mathcal{D}$ is the smooth, projective, plane quadric
$$T_{\Omega, c_1(L)} = \{ (x)^2=0 \; | \; x \in   V_{\Omega, c_1(L)} \}$$
called the twistor family. The twistor family has the following key property:
\begin{prop} [\cite{bryan-leung}, Prop.\ 3.1] \label{twistor-prop}
Let $M $ be an effective divisor on $X$.
Let $Y$, $\phi : H^2(Y, \mathbb{Z}) \simeq \Lambda_{K3}$ be a marked K3 surface corresponding to a point $y \in T_{\Omega, c_1(L)}$.
Then $\phi^{-1}(c_1(M))$ lies in $H^{1,1}(Y, \mathbb{Z})$  if and only if $y \in \{ \Omega, \overline{\Omega}\}$, i.e.\ if and only if $Y$ is either the marked K3 surface $X$ or its complex conjugate. Further, if $(M)^2 \geq -2$, then $\phi^{-1}(c_1(M))$ is an effective $1-1$ class if and only if $Y$ is the marked K3 surface $X$.
\end{prop}
Treat $S$ as a complex manifold. Let $\text{Def}(\mathcal{X}_0)$
be the Kodaira--Spencer space of deformations of $\mathcal{X}_0$. By the local Torelli theorem we have a local isomorphism
$$\text{Def}(\mathcal{X}_0) \to \mathcal{D}.$$ After replacing $S$ with an (analytic) open set about $0$, let
$$\rho :  S \to \mathcal{D} \seq  \proj(\Lambda_{K3} \otimes \C), $$
be the period map, see \cite[Ch.\ 6.2]{huy-lec-k3}. Shrinking $S$ about $0$ if necessary, the image of $\rho$ lies in the hypersurface
$$ \mathcal{D} \cap \proj(l^{\perp}) \seq \proj(\Lambda_{K3} \otimes \C) ,$$
where $l=c_1(\mathcal{L}_0) \in  H^2(\mathcal{X}_0, \C)\simeq \Lambda_{K3} \otimes \C$. If $(l)^2=2g-2$, then we denote by $\mathcal{H}_g$ a small open set  $ \mathcal{D} \cap \proj(l^{\perp}) $ about $\rho(0)$; $\mathcal{H}_g$ is smooth and connected of dimension $19$. Now let $\widetilde{\mathcal{H}} \seq \mathcal{D}$ be the holomorphic $\proj^1$ bundle over $\mathcal{H}$ given by the union of twistor spaces $$\bigcup_{\Omega \in \mathcal{H}_g} T_{\Omega, l}.$$ At least locally about $0$, we construct a $\proj^1$ bundle $\pi: \widetilde{S} \to S$ and a morphism of $\proj^1$ bundles
$$\tilde{\rho} \; : \; \widetilde{S} \to \widetilde{\mathcal{H}} \seq \mathcal{D} $$
together with a section $s: S \to \widetilde{S}$, with $\tilde{\rho} \circ s = \rho$. We denote by $\widetilde{X} \to \widetilde{S}$ the pull back of the universal family of K3 surfaces under $\tilde{\rho}$.
\begin{thm} \label{relative-gdct}
Let $h: \mathcal{X} \to S$ be a smooth, projective family of K3 surfaces over a smooth and connected algebraic variety $S$. Let $\mathcal{M}$ be any effective $S$-flat line bundle 
on $\mathcal{X}$, such that $(\mathcal{M}_t)^2 \geq -2$ for each closed point $t \in S$. Then each component of $\mathcal{W} (\mathcal{X}, \mathcal{M},p)$ has dimension at least
$\dim(S)+p.$
\end{thm}
\begin{proof}
This is similar to Theorem \ref{good-bound}. Let $0 \in S$ be a point and  $f: D \to X$ be a closed point of $\mathcal{W} (\mathcal{X}, \mathcal{M},p)$ over $0$. Let $\tilde{h}: \widetilde{X} \to \widetilde{S}$ be the universal family of K3 surfaces over the twistor family $\widetilde{S}$ as above.
Let $a$ be a semiuniversal deformation of $\text{Def}_h(D,f,X)$; we have to show that the base of $a$ has dimension at least $\dim(S)+p$. From Proposition \ref{twistor-prop} and the argument of Theorem \ref{good-bound}, this reduces to showing that if $b$ is a semiuniversal deformation of $\text{Def}_{\tilde{h}}(D,f,X)$, then the base of $b$ has dimension at least $\dim(S)+p=\dim\widetilde{S}+p-1$. This is Proposition \ref{main-good-dimension-count} together with Proposition \ref{relative-analy-dimct}.
\end{proof}
\begin{remark}
For an alternative proof of this theorem see \cite[Thm.\ 2.4, Rem.\ 3.1]{kool-thomas}. This approach makes use of Bloch's semiregularity map, \cite{bloch-semi}, \cite{maulik-pand-nl}.
\end{remark}

\section{Deformation theory for an unramified map} \label{unramSect}
We now restrict ourselves to the case of unramified stable maps. This makes the deformation theory much simpler to handle.

Let $f: D \to S$ be a stable map from a connected nodal curve to a smooth projective surface. If we assume that $f$ is \emph{unramified}, then
the natural morphism $$f^* \Omega_S \to \Omega_D$$ of sheaves of differentials is surjective. We may thus define a normal bundle $N_f \in \text{Pic}(D)$ as the dual of the line bundle $N_f^*$ defined via the exact sequence
$$0 \to N_f^* \to f^* \Omega_S \to \Omega_D \to 0 .$$ As in \cite[\S 3.2]{sernesi-def}, one may define a deformation functor describing those deformation of the stable map $f$ which \emph{leave the target $S$ fixed}. If we let $\text{Def}(f,S)$ be the vector space of deformations of $f$ which leave $S$ fixed, we have
$$ \text{Def}(f,S)=h^0(D,N_f)$$
and a space of obstructions is given by 
$$ \text{Obst}(f,S)=h^1(D,N_f).$$
The following fact from \cite[\S 2.1]{ghs-rat} will prove to be useful:
\begin{prop} \label{normal-bdl}
Let $f: D \to X$ be an unramified morphism from a connected nodal curve to a projective K3 surface and let $B \seq D$ be a connected union of components. If we let $f_B := f_{|_B}$, then
$$ N_{f_B}(Y) \simeq (N_{f})_{|_B}$$
where $Y=B \cap \overline{D \setminus B}$.
\end{prop}
\begin{proof}
The second exact sequence for K\"ahler differentials gives a surjection $(\Omega_{D})_{|_B} \to \Omega_{B} \to 0$ which is an isomorphism outside $Y$. Thus we have a short exact sequence on $B$
\begin{align} \label{kahler-ses}
0 \to T \to (\Omega_{D})_{|_B} \to \Omega_{B} \to 0
\end{align}
where $T$ is a torsion sheaf, supported on $Y$. We claim that $T \simeq \mathcal{O}_Y$. Let $p \in Y$ be a closed point. It suffices to show
that the stalk $T_p$ is isomorphic to the sky-scraper sheaf $\mathcal{O}_p$. Since $\mathcal{O}_B \to \hat{\mathcal{O}}_B$ is flat, where $\hat{\mathcal{O}}_B$ denotes the completion, and $D$ is nodal, it suffices to consider the case where $D$ is the affine plane curve defined by $xy=0$,
$B \seq D$ is the curve $x=0$ and $p$ is the origin. But in this case Equation \ref{kahler-ses} boils down to the following exact sequence of $A=k[y]$ modules
$$0 \to \slfrac{A<dx>}{ydx} \to \slfrac{A<dx,dy>}{ydx} \to A<dy> \to 0 . $$
Since $\slfrac{A<dx>}{ydx}$ is isomorphic to the $A$-module $k$, this proves the claim.

Next, $f$ is an unramified, hence stable, map to a projective variety, and so $\omega_D \otimes f^*(\mathcal{O}_X(3))$ is ample. Thus we have a closed
immersion $D \hookrightarrow \proj^k$ for some $k$. Since $D$ is a reduced local complete intersection curve, we have a short exact sequence
$$ 0 \to N_{D / \proj^k}^* \to (\Omega_{\proj^k})_{|_D} \to \Omega_D \to 0,$$
where $N_{D / \proj^k}^*$ is locally free.
Thus $(\Omega_{D})_{|_B}$ admits a locally free resolution and so the line bundle $\det (\Omega_{D})_{|_B}$ is defined. Likewise, $\det \Omega_B$ is a well-defined line bundle. We have a short exact sequence
$$ 0 \to \mathcal{O}_B(-Y) \to \mathcal{O}_B \to T \to 0 ,$$
so $\det T \simeq \mathcal{O}_B(Y)$ and  Equation \ref{kahler-ses} gives
$$ \det (\Omega_{D})_{|_B} \simeq \det \Omega_B (Y).$$
The result now follows from taking determinants of the short exact sequences 
$$0 \to (N_f^*)_{|_B} \to (f^* \Omega_X)_{|_B} \to (\Omega_D)_{|_B} \to 0 $$
and 
$$0 \to N_{f_B}^* \to f_B^* \Omega_X \to \Omega_B \to 0 .$$
\end{proof}
\begin{prop} \label{ishi-little}
Let $[(f: D \to X, L)] \in \mathcal{W}^{n}_{g,k}$ represent an unramified stable map such that $h^0(N_f) \leq p(D)$, where $p(D)=p(g,k)-n$ denotes the arithmetic genus of $D$. Then for every irreducible component $J \seq\mathcal{W}^{n}_{g,k}$ containing $[(f: D \to X, L)]$ the projection $\pi: J \to \mathcal{B}_g$ is dominant.
\end{prop}
\begin{proof}
We have seen that $\dim J \geq p(D)+19$ in Theorem \ref{good-bound}. The fibre $\pi^{-1}([(X,L)])$ may be identified with the space of stable maps into the \emph{fixed} surface $X$, and thus each component of $\pi^{-1}([(X,L)])$ containing $[(f: D \to X, L)]$ has dimension at most $h^0(N_f) \leq p(D)$, \cite[\S 2.1]{ghs-rat}. Thus $\dim \pi(J)=\dim \mathcal{B}_g$, so $\pi: J \to \mathcal{B}_g$ is dominant (note that this also forces the equality $h^0(N_f) = p(D)$).
\end{proof}
By using Theorem \ref{relative-gdct}, one has the following, slightly different version of the above lemma. The proof is identical.
\begin{prop} \label{relative-dominant}
Let $S$ be a smooth and connected algebraic variety over $\C$. Let $\mathcal{X} \to S$ be a projective, flat family of K3 surfaces. Let $\mathcal{M}$ be an effective $S$-flat line bundle, with $(\mathcal{M}_t)^2 \geq -2$ for any $t \in S$. Let $[f: D \to \mathcal{X}_s] \in \mathcal{W}(\mathcal{X}, \mathcal{M}, p)$ represent an unramified stable map such that $h^0(N_f) \leq p$. Then the projection $\pi: \mathcal{W}(\mathcal{X}, \mathcal{M}, p) \to S$ is dominant near $[f: D \to \mathcal{X}_s]$. Further, $\dim_f \mathcal{W}(\mathcal{X}, \mathcal{L}, p)=p+\dim S$.
\end{prop}

The following is a generalisation of \cite[Prop.\ 2.3]{huy-kem}. \footnote{There is a minor mistake in the proof of \cite[Prop.\ 2.3]{huy-kem}; the claimed isomorphism $\Omega_{D_{|_{D_n}}} \simeq \mathcal{O}(-1)$ should be replaced with $det(\Omega_{D_{|_{D_n}}}) \simeq \mathcal{O}(-1)$, as the sheaf $\Omega_{D_{|_{D_n}}}$ is not torsion free.}
\begin{lem} \label{stablem}
Let $f: D \to X$ be an unramified morphism from a connected nodal curve to a K3 surface, and let $N_f$ denote the normal bundle of $f$. Assume that the irreducible components $Z_1, \ldots, Z_s$ of $D$ are smooth. Assume further that we may label the components such that $\bigcup_{i=1}^j Z_i$ is connected for all $j \leq s$. Then $h^0(N_f) \leq p(D)$, where $p(D)$ denotes the arithmetic genus of $D$.
\end{lem}
\begin{proof}
We will prove this by induction on $s$. If $D$ is irreducible, then by assumption $D$ is smooth, so we have a short exact sequence of vector bundles
$$0 \to T_D \to f^* T_X \to N_f \to 0 $$
and taking determinants gives $N_f \simeq \omega_D$. Thus $h^0(N_f)=h^0(\omega_D)=p(D)$. Now let $T:=\overline{D \setminus Z_s}=\bigcup_{i=1}^{s-1} Z_i$; this is connected by assumption. Let $\{p_1, \ldots, p_r\} = Z_s \cap T$. We have a short exact sequence
$$0 \to {N_f}_{|_T}(-p_1 -\ldots -p_r) \to N_f \to {N_f}_{|_{Z_s}} \to 0 .$$ If $B \seq D$ is a connected union of components, and $f_B:=f_{|_B}$, $Y:=B \cap (\overline{D \setminus B})$, then $N_{f_B}(Y)={N_f}_{|_B}$ from Proposition \ref{normal-bdl}. Thus
$$ h^0(N_f) \leq h^0(N_{f_T})+h^0(\omega_{Z_s}(p_1+\ldots+p_r)).$$ 
By induction, $h^0(N_{f_T}) \leq p(T)$, and further $h^1(\omega_{Z_s}(p_1+\ldots+p_r))=h^0(\mathcal{O}_{Z_s}(-p_1 \ldots -p_r))=0$, so Riemann--Roch gives $h^0(\omega_{Z_s}(p_1+\ldots+p_r))=p(Z_s)+r-1$. Thus the claim follows from
 $p(B)=p(T)+p(Z_s)+r-1$.
\end{proof}
\begin{remark} \label{slightgen-nodal}
It follows from the proof that the above result may be generalised as follows. Suppose $f: D \to X$ be an unramified morphism from a connected nodal curve to a K3 surface, and $D= \bigcup_{i=1}^s Z_i$ where $Z_1$ is connected, but not necessarily irreducible or smooth, and with $Z_2, \ldots, Z_s$ smooth (and with $s>1$). Assume $\bigcup_{i=1}^j Z_i$ is connected for all $j \leq s$, and $h^0(N_{f_1}) \leq p(Z_1)$, where $f_1:=f_{|_{Z_1}}$. Then $h^0(N_f) \leq p(D)$.
\end{remark}

\begin{mydef}
We define $\mathcal{T}^n_{g,k} \seq \mathcal{W}^{n}_{g,k}$ to be the open substack parametrising stable maps $f: D \to X$ such that $D$ is integral and smooth and $f$ is unramified and birational onto its image.
\end{mydef}
We now need the following result on simultaneous normalisation of families of singular curves, which in this generality is usually attributed to Raynaud, generalising the results of Teissier, \cite{teissier}. For modern treatments and further generalisations, see \cite{chiang-lipman},  \cite[Thm.\ 12]{kollar-simultaneous}.
\begin{prop} [Teissier, Raynaud] \label{sim-res}
Let $B$ be a normal, integral scheme of finite-type over $\C$, and let 
$$ f: \mathcal{C}_1 \to B$$ be a projective, flat family of relative dimension one with reduced fibres. Then there exists a simultaneous resolution, i.e.\ a flat family $ g: \mathcal{C}_2 \to B$ of relative dimension one with normal fibres, together with a finite map $h: \mathcal{C}_2 \to \mathcal{C}_1$ such that $f \circ h=g$ and the restriction morphism $$h_b:  \mathcal{C}_{2,b} \to \mathcal{C}_{1,b} $$ is the normalisation map for each fibre over $b \in B$.
\end{prop}
\begin{lem} \label{defo-nodal-lemma}
Let $f: D \to X$ be an unramified morphism from an integral, nodal curve $D$ to a K3 surface, with $[(f: D \to X,L)] \in \mathcal{W}^{n}_{g,k}$ and assume that $f$ is birational onto its image. Then $[(f: D \to X,L)]$ lies in the closure of $\mathcal{T}^n_{g,k}$.
\begin{proof}
Suppose there was a component of $J$ of $\mathcal{W}^{n}_{g,k}$ containing $[(f: D \to X,L)]$ such that if  $[(f': D' \to X',L')]$ is general, then $D'$ is nodal with at least $m >0$ nodes. Replace $J$ with the dense open subset parametrizing unramified maps, which are birational onto the image and such that the base $D'$ is integral with exactly $m$ nodes. Composing $f'$ with the normalization $\widetilde{D} \to D'$ gives an unramified stable map $h: \widetilde{D} \to X'$; thus we have a map $G: J(\C) \to \mathcal{T}^{n+m}_{g,k}(\C)$ between the \emph{sets} of closed points of the respective stacks. The fact that $f$ is birational onto its image implies that for a general $y \in Im(G)$, $G^{-1}(y)$ is a finite set. Indeed, if $y$ corresponds to the stable map $h: \widetilde{D} \to X'$, then any element of $G^{-1}(y)$ corresponds to an unramified stable map $h': B \to X'$, birational onto its image, where $B$ is integral and has exactly $m$ nodes and the normalization of $B$ is $ \widetilde{D}$. Further, 
if $\mu: \widetilde{D} \to B$ is the normalization morphism, then $h' \circ \mu=h$, by definition of $G$. As $h'$ is birational onto its image, the set 
$$ Z := \{ z \in \widetilde{D} \; | \; \text{There exists $y \neq z \in \widetilde{D}$ with $h(z)=h(y)$} \}$$
is finite. Since $h'$ is obtained from $h$ by glueing $m$ pairs of points in $Z$, there are only finitely many possibilities for $h'$.

We claim that, at least after a finite base change, $G$ is locally induced by a morphism of stacks. After replacing $J$ with an etale cover, we may assume it comes with a universal family. Let $J' \to J$ be the normalization of $J$. Pulling back the universal family on $J$ gives a family of stable maps over $J'$. In particular, we have a flat family $\mathcal{D} \to J'$ of nodal curves with exactly $m$ nodes specializing to $D'$.  Applying Proposition \ref{sim-res}, we may simultaneously resolve the $m$ nodes of the fibres of $\mathcal{D}$ to produce a family $\widetilde{\mathcal{D}} \to J'$ of \emph{smooth} curves, together with a morphism $\widetilde{\mathcal{D}} \to \mathcal{D}$ restricting to the normalization over each point in $J'$. By composing with the universal family of stable maps $\mathcal{D} \to \mathcal{X}$, where $\mathcal{X}$ is a family of K3 surfaces, we produce a family of stable maps $\widetilde{\mathcal{D}} \to \mathcal{X}$ over $J'$. By the universal property of $\mathcal{T}^{n+m}_{g,k}$, this produces a morphism of Deligne--Mumford stacks $J' \to \mathcal{T}^{n+m}_{g,k}$ which coincides with the composition $G'$ of $G$ with $J' \to J$ on the level of closed points. As this morphism is generically finite onto its image, the dimension of $\mathcal{T}^{n+m}_{g,k}$ is at least $\dim J'=\dim J \geq p(D)+19$. But $\mathcal{T}^{n+m}_{g,k}$ is smooth of dimension $p(D)-m+19$, so this is a contradiction.
\end{proof}
\end{lem}
\begin{prop} \label{prim-cor}
Let $[(f: D \to X, L)] \in \mathcal{W}^{n}_{g,k}$ represent an unramified stable map such that  $h^0(N_f) \leq p(D)$, where $p(D)=p(g,k)-n$. Assume furthermore that there is no decomposition $$D= \bigcup_{i =1}^t D_i$$ for $t >1$ with each $D_i$ a connected union of irreducible components of $D$ such that $D_i$ and $D_j$ meet transversally (if at all) for all $i \neq j$ and such that for all $1 \leq i \leq t$, $f_*(D_i) \in |m_iL|$ for a positive integer $m_i >0$ (this is automatic if $k=1$). Lastly, assume that there is some component $D_j$ such that $f_{|_{D_j}}$ is birational onto its image, and if $D_i \neq D_j$ is any component, $f(D_i)$ and $f(D_j)$ intersect properly. Then $[(f: D \to X, L)]$ lies in the closure of $\mathcal{T}^{n}_{g,k}$.
\end{prop}
\begin{proof}
 We need to show that we may deform $[(f: D \to X, L)]$ to an unramified stable map $[(f': D' \to X', L')]$ with $D'$ irreducible and smooth. We will firstly show that $[(f: D \to X, L)]$ deforms to a stable map with irreducible base. 
 
 From Proposition \ref{ishi-little}, $\pi: \mathcal{W}^{n}_{g,k} \to \mathcal{B}_g$ is dominant near $[(f: D \to X, L)]$. Thus we may deform $[(f: D \to X, L)]$ to an unramified stable map $[(f': D' \to X', L')]$, where $\text{Pic}(X') \simeq \mathbb{Z}L'$. For any irreducible component $Z \seq D'$, $f'_*(Z) \in |a_iL'|$ for some integer $a_i >0$. Given any one-parameter family $\gamma(t)=[(f_t: D_t \to X_t, L_t)]$ of stable maps with $\gamma(0)=[(f: D \to X, L)]$, then, by the Zariski connectedness theorem, after performing a finite base change about $0$ if necessary the irreducible components of $D_t$ for generic $t$ deform to a connected union of irreducible components of $D$ as $t \to 0$ (see \cite[\href{http://stacks.math.columbia.edu/tag/0551}{Tag 0551}]{stacks-project} and \cite[Cor.\ 8.3.6]{liu}). Thus the condition on $D$ ensures that $[(f: D \to X, L)] $ deforms to an unramified stable map of the form $[(f': D' \to X', L')]$ with $D'$ \emph{integral} and nodal. 
 
We will next show that $[(f: D \to X, L)] $ deforms to an unramified stable map of the form $[(f': D' \to X', L')]$ with $D'$ integral and nodal and such that $f'$ is birational onto its image. By Lemma \ref{defo-nodal-lemma}, this will complete the proof. Let $S$ be a smooth, irreducible, one-dimensional scheme with base point $0$, and suppose we have a diagram
 \[
\xymatrix{
\mathcal{D} \ar[r]^{\tilde{g}}  \ar[rd]_{\pi_1} & \mathcal{X} \ar[d]^{\pi_2} \\
&S
}
\] 
with $\tilde{g}$ proper, $\pi_1$, $\pi_2$ flat and with $\tilde{g}_s: \mathcal{D}_s \to \mathcal{X}_s$ an unramified stable map to a K3 surface for all $s$, with $\tilde{g}_0=f$ and such that $\mathcal{D}_s$ is integral for $s \neq 0$. By Proposition \ref{push-forwards} there is an $S$-flat line bundle $\mathcal{L}$ on $\mathcal{X}$, with $\mathcal{L}_0=kL$ and that the cycle $\tilde{g}_{*}(\mathcal{D}) \sim \mathcal{L}$ is a relatively effective (Cartier) divisor. So the cycle $\tilde{g}_{*}(\mathcal{D})$ may be considered as an $S$-relatively effective divisor $\bar{\mathcal{D}} \seq \mathcal{X}$. By the assumptions on $f=\tilde{g}_0$, the irreducible surface $\bar{\mathcal{D}}$ is reduced on a dense open subset meeting $f(D_j)$, which forces $\tilde{g}_s$ to be birational for $s$ near $0$ (as if $\deg(\tilde{g}_s)=d$, $\bar{\mathcal{D}}_s=d \tilde{g}_s(\mathcal{D}_s)$).
\end{proof}

\chapter{Constructing elliptic K3 surfaces} 
The goal of this chapter is to offer two different ways to construct elliptic K3 surfaces with many special properties.
These special surfaces can then be used in deformation arguments to gain insight on the properties of singular curves
on general K3 surfaces.
We start with the following definition from \cite{miranda}:
\begin{mydef}
A \emph{minimal elliptic surface over $\proj^1$} is an integral smooth surface $S$ together with a surjective morphism
$\pi: S \to \proj^1$ such that:
\begin{enumerate}
\item The general fibre of $\pi$ is an elliptic curve.
\item There are no $(-1)$ curves in the fibres of $\pi$.
If in addition $\pi$ has a section $s: \proj^1 \to S$, then we call $S$ a \emph{minimal elliptic surface over $\proj^1$ with section}. 
If $\pi$ defines a minimal elliptic surface over $\proj^1$ with section, and we further have that all fibres of $\pi$ are integral, then
$\pi$ is called a \emph{Weierstrass fibration over $\proj^1$}.
\end{enumerate}
We will be mostly interested in the case $S$ is a K3 surface. In this case the word ``minimal" is vacuous, so we will simply use the term \emph{elliptic K3 surface} to denote a minimal elliptic surface $S$ over $\proj^1$ such that $S$ is a K3 surface. 

In this chapter we discuss the two main techniques to construct elliptic K3 surfaces. In Section \ref{WeierEll} we describe how the Weierstrass equation can be used to very explicitly construct such surfaces. We will give several interesting examples of this construction. In Section \ref{TorelliEll} we will use the Torelli theorem for K3 surfaces to show that many elliptic K3 surfaces exist with prescribed
Picard groups. In contrast to the first approach, this only gives an implicit construction of elliptic K3 surfaces. In practice, however, one is nevertheless able to extract much geometric information out of the Torelli procedure. 
\end{mydef}
\section{Constructions via the Weierstrass equation} \label{WeierEll}
We start by reminding the reader of some elementary classical theory on elliptic curves.
Any integral curve over $\C$ of arithmetic genus one can be realised as a plane cubic. After a coodinate change, this plane cubic can be
put in the form 
$$ y^2z=x^3+axz^2+bz^3$$
for $a,b \in \C$. This is the \emph{Weierstrass equation}. The corresponding elliptic curve is singular if and only if
$ \Delta :=-16 (4a^3+27b^2) $ vanishes.

If we replace the constants $a,b \in \C$ with polynomials $a(t), b(t) \in \C [t]$, then we obtain a two-dimensional variety $S_0 \seq \proj^2 \times \mathbb{A}^1$, together with a projection $\pi: S_0 \to  \mathbb{A}^1$. The point $[x,y,z]=(0,1,0)$ defines a section to $\pi$. The fibre of $\pi$ over $t=t_0$ is a plane cubic curve, and it is singular if and only if $\Delta(t_0)=-16(4 a^3(t_0)+27b^3(t_0))=0$. 

Following \cite{miranda-moduli}, one can globalise the above equation to construct elliptic surfaces with a Weierstrass fibration. Let $L:= \mathcal{O}_{\proj^1}(N)$ for $N >0$, let $(A,B) \in H^0(L^4) \times H^0(L^6)$ be pairs of sections and assume that
$$ \Delta(A,B):= -16(4A^3+27B^2) \in H^0(L^{12})$$ is not the zero section. Consider the closed subscheme $S_{A,B}$ of $\proj:=\proj(L^2 \oplus L^3 \oplus \mathcal{O}_{\proj^1})$ defined by $y^2z=x^3+Axz^2+Bz^3$, where $(x,y,z)$ is a global coordinate system of $\proj$ relative to $(L^2,L^3,\mathcal{O}_{\proj^1})$. If we assume that $S$ is smooth, then the third projection $\pi: S \to \proj^1$ defines a Weierstrass fibration. Any Weierstrass fibration of a smooth elliptic surface comes from the above construction. Finally, $S$ is a K3 surface if and only if $N=2$.

Let $(A,B) \in H^0(\mathcal{O}_{\proj^1}(8)) \times H^0(\mathcal{O}_{\proj^1}(12))$ such that $\Delta(A,B)$ is not the zero section. We now describe explicitly the local equations for $S_{A,B}$. Let $U_0, U_1$ be two copies of $\mathbb{A}^1$ and set $W_0=U_0 \times \proj^2$, $W_1=U_1 \times \proj^2$. Then $\proj(\mathcal{O}_{\proj^1}(4) \oplus \mathcal{O}_{\proj^1}(6) \oplus \mathcal{O}_{\proj^1})$ is obtained from $W_0 \sqcup W_1$ by identifying $(u_0, [x_0, y_0, z_0])$ with $(u_1, [x_1, y_1, z_1])$ if and only if $u_0 \neq 0$ and 
$$ u_1= \frac{1}{u_0}, \; \; x_1=\frac{x_0}{u_0^4}, \; \; y_1=\frac{y_0}{u_0^6}, \; \; z_1=z_0.$$

The global sections $A$ resp.\ $B$ correspond to pairs of polynomial functions $(A_0, A_1)$, $(B_0, B_1)$ of degree $8$ resp.\ $12$
\begin{align*}
A_i \; : \; U_i \to \mathbb{A}^1, \; \; \; B_i \; : \; U_i \to \mathbb{A}^1, \; \; \text{for $i=0,1$}
\end{align*} 
such that $A_1(u)=u^8 A_0(\frac{1}{u})$ and $B_1(u)=u^{12} B_0(\frac{1}{u})$. The variety $S_{A,B}$ is then given by the equations
$$y_i^2z_i=x_i^3+A_i(u_i)x_iz_i^2+B_i(u_i)z_i^3 $$
on $W_i$ for $i=0,1$. 

\begin{eg} [\cite{kemeny-thesis}, Prop.\  7.1]
Let $a_1, \ldots, a_{12} \in \C$ be pairwise distinct. Set $$A=0 \in H^0(\mathcal{O}_{\proj^1}(8)), \; \; B= \prod_{i=1}^{12}(u_0-a_iu_1) \in H^0(\mathcal{O}_{\proj^1}(12)).$$ Denote the scheme $S_{A,B}$ by $X_{(a_i),0}$. Then the Jacobian criterion applied to the local description above gives that $X_{(a_i),0}$ is a smooth K3 surface. The fibration $X_{(a_i),0} \to \proj^1$ has $12$ singular fibres, which occur over the points $[a_i : 1] \in \proj^1$. Each singular fibre is a rational curve with one cusp and no other singularity.
\end{eg}

Even in the case where $(A,B) \in H^0(\mathcal{O}_{\proj^1}(8)) \times H^0(\mathcal{O}_{\proj^1}(12))$ do not define a smooth surface $S_{A,B}$, we may often produce a smooth elliptic K3 surface by desingularisation. For $p \in \proj^1$, let $\nu_p(A)$ resp.\ $\nu_p(B)$ be the order of vanishing of $A$ resp.\ $B$ at $p$. We have the following, again from \cite{miranda}:
\begin{prop}
Let $(A,B) \in H^0(\mathcal{O}_{\proj^1}(8)) \times H^0(\mathcal{O}_{\proj^1}(12))$ be such that $\Delta(A,B)$ is not the zero section. Assume
that for any $p \in \proj^1$ either $\nu_p(A) \leq 3$ or $\nu_p(B) \leq 5$ (or that both conditions hold). Then $S_{A,B}$ defines a surface with at worst rational double point singularities. The minimal desingularisation is an elliptic K3 surface $\widetilde{S}_{A,B} \to \proj^1$ with section.
\end{prop}
Note that we did \emph{not} claim that  $\widetilde{S}_{A,B} \to \proj^1$ is a Weierstrass fibration. This is in general not the case. There is however, a characterisation of the singular fibres $F$ of  $\widetilde{S}_{A,B}$ over $p \in \proj^1$ in terms of the singularities of $S_{A,B}$ over $p \in \proj^1$, due to Kodaira, \cite{kodaira-I}, \cite{kodaira-II}. The information is tabulated below, where $\chi(F)$ denotes the Euler number of $F$.

\begin{center}
\begin{tabular}{|l|p{5cm}|l|l|}
\hline
Kodaira type of $F$ & Description of $F$ & $\chi(F)$ & Singularity Type of $S_{A,B}$ \\ \hline
$I_0$ & smooth elliptic curve & $0$ & none \\ \hline
$I_1$ & irreducible, nodal rational curve & $1$ & none \\ \hline
$I_N$, $N \geq 2$ & cycle of $N$ smooth rational curves, meeting transversally & $N$ & $A_{N-1}$ \\ \hline
$I_N^*$, $N \geq 0$ & $N+5$ smooth rational curves meeting with dual graph $D_{N+4}$ & $N+6$ & $D_{N+4}$ \\ \hline
$II$ & a cuspidal rational curve & $2$ & none \\ \hline
$III$ & two smooth rational curves meeting at a point of order two & $3$ & $A_1$ \\ \hline
$IV$ & three smooth rational curves meeting at a point & $4$ & $A_2$ \\ \hline
$IV^*$ & seven smooth rational curves meeting with dual graph $\tilde{E}_6$ & $8$ & $E_6$ \\ \hline
$III^*$ & eight smooth rational curves meeting with dual graph $\tilde{E}_7$ & $9$ & $E_7$ \\ \hline
$II^*$ & nine smooth rational curves meeting with dual graph $\tilde{E}_8$ & $10$ & $E_8$ \\ \hline
 \end{tabular}
\end{center}
\clearpage

The following proposition is an easy application of Kodaira's classification of the singular fibres of $\widetilde{S}_{A,B}$.
\begin{prop}
Let $(A,B) \in H^0(\mathcal{O}_{\proj^1}(8)) \times H^0(\mathcal{O}_{\proj^1}(12))$ be such that $\Delta(A,B)$ is not the zero section. Assume
that for any $p \in \proj^1$ either $\nu_p(A) \leq 3$ or $\nu_p(B) \leq 5$ (or that both conditions hold). Assume further that $\Delta(A,B) \in H^0(\mathcal{O}_{\proj^1}(24))$ has $24$ distinct roots. Then $S_{A,B} \to \proj^1$ is smooth and all singular fibres are integral, rational curves with
one node and no other singularities. 
\end{prop}
\begin{proof}
As $\Delta(A,B)$ has $24$ distinct roots, $\widetilde{S}_{A,B}$ has $24$ singular fibres $F_1, \ldots, F_{24}$. Since
$24= \chi(\widetilde{S}_{A,B})=\sum_{i=1}^{24} \chi(F_i),$ then from the table above we see $\chi(F_i)=1$ for all $i$, $F_i$ is an
integral, rational curve with one node and no other singularities, and that $\widetilde{S}_{A,B}=S_{A,B}$.
\end{proof}
\begin{eg}
Let $(a_1, \ldots, a_{12}) \in \C^{12}$ be pairwise disjoint. Let $K \in \C$, and let $l:= -\sqrt[3]{\frac{27}{4}} \in \R$. Set
\begin{align*}
\alpha &= \prod_{i=1}^{12}(x_0-a_ix_1) \\
A&= K^2lx_1^8 \\
B&=\alpha+K^3x_1^{12}
\end{align*}
We have $\Delta(A,B)=-432 \alpha (\alpha+2K^3x^{12}_1)$. For a general choice of $K \neq 0$, $\alpha+2K^3x^{12}_1$ has $12$
distinct roots, all of which are distinct from the $a_i$.  Thus $S_{A,B}$ is a smooth K3 surface with $24$ singular fibres, all of which are
integral, nodal rational curves. We denote this K3 surface $X_{(a_i),K}$.
\end{eg}
\begin{remark}
Let $\mathcal{P}_g \to \mathcal{B}_g$ be the natural stack with fibre over a point representing a primitively polarised K3 surface $(X,L)$ equal to the linear system $|L|$. Let $\mathcal{V}_g^g \seq \mathcal{P}$ be the substack with fibre over $[(X,L)]$ parametrising integral $C \in |L|$ with precisely $g$ nodes and no other singularities; i.e.\ $C$ is a nodal rational curve. By studying rational stable maps to $X_{(a_i),K}$, any by considering the deformation $X_{(a_i),K} \to X_{(a_i),0}$, we were able to prove that the closure $\overline{\mathcal{V}}_g^g$ of $\mathcal{V}_g^g \seq \mathcal{P}$ is connected, \cite[Thm.\ 8.1]{kemeny-thesis}. Note that $\mathcal{V}_g^g$ is conjecturally irreducible, but this remains widely open and appears to be very difficult, \cite{dedieu}.
\end{remark}

We end this section with one final example which will play a role in the study of the Chow group of zero-cycles on a K3 surface.
\begin{eg} \label{X_{(2),8A_1}}
Consider the cubic equation
$$y^2=x^3+a(u_0)x^2+b(u_0)x $$
for general polynomials $a$ of degree four and $b$ of degree eight. After the coordinate change
$x_2=x+\frac{a(u_0)}{3}$, we can put this in Weierstrass form 
$$y^2=x_2^3+A(u_0)x+B(u_0)$$
where $A$ has degree eight and $B$ has degree twelve. Further, 
$\Delta(A,B)=b^2(a^2-4b)$. After homogenising $A,B$ we produce a elliptic K3 surface $S_{A,B}$ which has
rational double point singularities over the eight zeroes of $b$ and $I_1$ fibres (integral rational nodal curves) over the eight solutions to $a^2-4b=0$. 
The resolution $\widetilde{S}_{A,B}$ is a smooth K3 elliptic surface illustrated below. The eight singular fibres over $b=0$ are $I_2$ fibres, i.e.\ the union of two smooth rational curves meeting in two points, \cite[Prop.\ 4.2]{sar-gee}. In the diagram below (taken from \cite{huy-kem}) these are denoted $N_1, \ldots, N_8$. The fibration $\widetilde{S}_{A,B} \to \proj^1$ has two sections: the section at infinity $\sigma$ given by $x=z=0$ (where $z$ is the coordinate introduced in homogenisation) and the section $\tau$ given by $x=y=0$. The section $\tau$ has order two; thus translation by $\tau$ defines an involution $\widetilde{S}_{A,B} \to \widetilde{S}_{A,B}$. We will denote this K3 surface by $X_{(2),8A_1}$ as it corresponds to a general K3 surface with $2$-torsion in the Mordell-Weil group and eight fibres of type $A_1$ (i.e.\ corresponding to $A_1$ singularities of the unresolved surface $S_{A,B}$), see \cite{shimada-arxiv}.
\end{eg}
$$
\begin{picture}(150,100)
\put(-50,80){\qbezier(0,0)(-40,-40)(0,-80)}
\put(-50,80){\qbezier(-14,0)(26,-40)(-14,-80)}
\put(-30,0){$\ldots$}
\put(-50,80){\qbezier(60,0)(20,-40)(60,-80)}
\put(-50,80){\qbezier(46,0)(86,-40)(46,-80)}

\put(80,80){\qbezier(55,0)(62,-70)(63,-40)}
\put(80,80){\qbezier(55,-80)(62,-10)(63,-40)}
\put(150,00){$\ldots$}
\put(120,80){\qbezier(55,0)(62,-70)(63,-40)}
\put(120,80){\qbezier(55,-80)(62,-10)(63,-40)}
\put(-1.5,62){\line(1,0){190}}
\put(-60,62){\line(1,0){52}}
\put(-79,62){\line(1,0){10}}
\put(-50.1,62){\circle*{2}}
\put(9.8,62){\circle*{2}}
\put(177,62){\circle*{2}}
\put(137,62){\circle*{2}}
\put(80,65){\tiny\mbox{$\sigma$}}
\put(14.5,22){\line(1,0){175}}
\put(-45,22){\line(1,0){54}}
\put(-79,22){\line(1,0){28}}
\put(-66.1,22.1){\circle*{2}}
\put(-6.1,22.1){\circle*{2}}
\put(177.1,22.1){\circle*{2}}
\put(137,22.1){\circle*{2}}
\put(139.6,40.1){\circle*{2}}
\put(179.6,40.1){\circle*{2}}
\put(80,25){\tiny\mbox{$\tau$}}
\put(-81.5,40){\tiny\mbox{$N_1$}}
\put(-22.5,40){\tiny\mbox{$N_8$}}
\end{picture}
$$

\section{Constructions via Torelli} \label{TorelliEll}
The aim of this section is to show how the Torelli theorem can be used to construct K3 surfaces.
We start with the following lemma.
\begin{lem}
Let $X$ be a K3 surface. Suppose $X$ contains a smooth, integral elliptic curve $E$. Then $X$ is an elliptic surface.
\end{lem} 
\begin{proof}
Let $L:= \mathcal{O}_X(E)$. From the arithmetic genus formula $(E)^2=2g(E)-2$ we see $(E)^2=0$. Thus $\deg(L_{|_{E}})=0$
so $h^0(E,L) \leq 1$. From
$$ 0 \to \mathcal{O}_X \to L \to L_{|_{E}} \to 0,$$
we have $h^0(X,L) \leq 2$. From Riemann--Roch $h^0(X,L) \geq \chi(L)=2$, so $h^0(X,L)=0$. Since $E$ is smooth and integral,
and $g(E) \neq 0$, $L$ is base point free and thus $L$ induces a morphism $X \to \proj^1$ such that the generic fibre is smooth, elliptic. 
\end{proof}
In order to show that a K3 surface contains a smooth elliptic curve, it suffices to show that it contains a non-trivial, nef divisor
$D$ with $(D)^2=0$. 
\begin{lem}
Let $X$ be a K3 surface. Suppose $X$ contains a non-trivial, nef divisor $D$ with $(D)^2=0$. Then $D$ is base point free. Furthermore, there
exists a smooth, integral elliptic curve $E$ with $D \sim mE$ for an integer $m \geq 1$.
\end{lem}
\begin{proof}
See \cite[Prop.\ 3.10]{huy-lec-k3}.
\end{proof}
By combining the above lemma with the Torelli theorem, one can construct elliptic K3 surfaces. To be more precise, let $\Lambda$ be an even lattice of rank $\rho+1 \leq 10$ and signature $(1, \rho)$, and fix an element $L \in \Lambda$ with $(L)^2>0$. A \emph{$\Lambda$-polarised K3 surface} is a K3 surface $X$, together with a primitive embedding $i: \Lambda \hookrightarrow  \text{Pic}(X)$ with $i(L)$ big and nef. The moduli space of $\Lambda$-polarised K3 surfaces exists as a quasi-projective algebraic variety, is nonempty, and has at most two components, both of dimension $19-\rho$, which locally on the period domain are interchanged by complex conjugation, \cite{dolgachev}. Complex conjugation here means that a complex surface $X$ with complex structure $J$ is sent to $(X,-J)$.

The generic K3 surface $X$ in the moduli space of $\Lambda$ polarised K3 surface has $\text{Pic}(X) \simeq \Lambda$. Thus if $\Lambda$ contains an element $D$ with $(D)^2=0$, and if further $D \in \text{Pic}(X)$ is nef, then $X$ is an elliptic K3 surface, and $D$ is a multiple of the class of an elliptic curve.

\begin{remark}
In the case where $\Lambda$ has rank greater than ten, there is not automatically a unique lattice embedding of $\Lambda$ into the K3 lattice. Using computer aided techniques, however, it has in recent years become possible to show the existence of such embeddings. One example of this is the work of Shimada, \cite{shimada-arxiv}, which gives a complete list of all possible configurations of the singular fibres of elliptic K3 surfaces with section up to ADE singularity type, together with the torsion part of the Mordell--Weil group.
\end{remark}

Suppose now $X$ is a K3 surface, and assume $E \in \text{Pic}(X)$ is smooth and elliptic. The elliptic fibration $X \to \proj^1$ induced by $E$
need not have a section. We define an integer $l$, called the \emph{multisection index} as follows. The set 
$$ \{ (D \cdot E) \; : \; D \in \text{Pic}(X)\} $$ is an ideal of $\mathbb{Z}$; $l$ is defined to be its positive generator. Then $l=1$ if and only if $X$ has a section. 

There is a construction which associated to $X \to \proj^1$ an elliptic K3 surface $J(X) \to \proj^1$ with a section $\proj^1 \to J(X)$, see \cite[\S 11.4]{huy-lec-k3}, \cite[Ch.\ 5]{cossec-dolgachev}, or\cite[Ch.\ 1.5]{fried}. One calls $J(X) \to \proj^1$ the \emph{Jacobian fibration}. The singular fibres of $X \to \proj^1$ and $J(X) \to \proj^1$ are isomorphic. We have an isomorphism $J(X)$ to $M_{H}(v)$, the moduli space of $H$-semistable sheaves on $X$ with Mukai vector $v$, where $H$ is a general polarisation on $X$ and $v=(0,E,l)$. Note that for $H$ generic, $M_H(v)$ is irreducible, \cite[Cor.\ 10.3.5]{huy-lec-k3}.

The lattice $\text{Pic}(J(X))$ is fully determined by $\text{Pic}(X)$. In fact, one has the following result, \cite[Lem.\ 2.1]{keum}. 
\begin{prop} [Keum]
Let $X \to \proj^1$ be an elliptic K3 surface with multisection index $l$ and generic fibre $E$. Then $\text{Pic}(X)$ embeds into $\text{Pic}(J(X))$ with index $l$. The element $E \in \text{Pic}(J(X))$ is divisible by $l$, and $\text{Pic}(J(X))$ is generated as a lattice by $\text{Pic}(X)$ and $F/l$. Further, 
$$\det  \text{Pic}(X) = l^2 \det \text{Pic}(J(X)).$$
\end{prop}

Let $X_1 \to \proj^1$, $X_2 \to \proj^1$ be two elliptic K3 surfaces. We say two elliptic K3 surfaces $X_1 \to \proj^1$ and  $X_2 \to \proj^1$ are isomorphic if there is an isomorphism $X_1 \simeq X_2$ commuting with the elliptic fibrations (note that a single K3 surface may have multiple, non-isomorphic elliptic fibrations). The following lemma will be needed later.
\begin{lem} \label{tate-shaf}
Let $X \to \proj^1$ be an elliptic K3 surface and consider $J(X) \to \proj^1$. Then there are only countably many non-isomorphic elliptic K3 surfaces $Y \to \proj^1$ with $J(Y) \to \proj^1$ isomorphic to $J(X) \to \proj^1$.
\end{lem} 
\begin{proof}
The set of isomorphism classes of elliptic K3 surfaces $Y \to \proj^1$ with $J(Y) \to \proj^1$ isomorphic to $J(X) \to \proj^1$ are
described by the Tate--\v{S}hafarevi\v{c} group of $J(X)$, which in our case is isomorphic to $(\mathbb{Q} / \mathbb{Z})^{22-\rho(J(X))}$, where $\rho(J(X))$ is the Picard number of $J(X)$, \cite[Rem.\ 11.5.12]{huy-lec-k3}, also cf.\ \cite[\S 1.5.3]{fried}, \cite[Cor.\ 2.2]{grothendieck-brauer} and \cite[Thm.\ 0.1]{huybrechts-brauer}.
\end{proof}

\tocless\subsection{The elliptic K3 surface $Y_{\Omega_g}$}  \label{mukai} 
We now use the Torelli theorem to construct special K3 surfaces and curves on them with some very special Brill--Noether properties. The examples constructed will play a crucial role in the coming chapters. 

Let $g\geq 5$ be an integer, let $1 \leq d_1,d_2, \ldots, d_8 < \left \lfloor \frac{g+1}{2} \right \rfloor$ be integers, and consider first the rank ten lattice $\Omega_g$ with ordered basis $\{L, E, \Gamma_1, \ldots, \Gamma_8 \}$ and with intersection form given by:
\begin{itemize}
\item $(L \cdot L)=2g-2$
\item $(L \cdot E)=  \left \lfloor \frac{g+1}{2} \right \rfloor$
\item $(E \cdot E)=0$
\item $(\Gamma_i)^2=-2$ for $1 \leq i \leq 8$
\item $(E \cdot \Gamma_i)=0$ for $1 \leq i \leq 8$
\item $(L \cdot \Gamma_i)=d_i$ for $1 \leq i \leq 8$
\item $(\Gamma_i \cdot \Gamma_j)=0$ for $i \neq j$, $1 \leq i,j \leq 8$
\end{itemize}
It is easily seen that the above lattice has signature $(1,9)$ and is even.
\begin{lem} \label{lem-aaa}
 Let $g \geq 6$ be an integer and choose $1 \leq d_1, \ldots, d_8 < \left \lfloor \frac{g+1}{2} \right \rfloor$. There exists a K3 surface $Y_{\Omega_g}$ with $\text{Pic}(Y_{\Omega_g}) \simeq \Omega_g$. Furthermore, for any such K3 we may choose the ordered basis $\{L,E, \Gamma_1, \ldots, \Gamma_8 \}$ of $\Omega_g$ 
 in such a way that $L-E$ is big and nef and with $\Gamma_i$ and $E$ representable by smooth, integral curves for $1 \leq i \leq 8$. Further there is a smooth rational curve $\widetilde{\Gamma}_i \in |E-\Gamma_i|$.
\end{lem}
\begin{proof}
By the global Torelli theorem and from a result of Nikulin, the fact that this lattice has signature $(1,9)$ and is even implies that there exists a K3 surface $Y_{\Omega_g}$ with $\text{Pic}(Y_{\Omega_g}) \simeq \Omega_g$, \cite[Cor.\ 1.9, Cor.\ 2.9]{morrison-large} or \cite[Cor.\ 14.3.1]{huy-lec-k3}. By performing Picard--Lefschetz reflections and a sign change, we may assume that $L-E$ is big and nef, since $(L-E \cdot L-E)>0$, \cite[Prop.\ VIII.3.9]{barth} and \cite[Cor.\ 8.2.11]{huy-lec-k3}. Next, $(E \cdot E)^2=0$ and $(L-E \cdot E) =\left \lfloor \frac{g+1}{2} \right \rfloor>0$ which implies that $E$ is effective, \cite[Prop.\ VIII.3.6(i)]{barth}. We now want to show that the general element of $|E|$ is smooth and irreducible. By \cite[Prop.\ 2.6]{donat} and the fact that $E$ belongs to a basis of $\text{Pic}(Y_{\Omega_g})$ it is enough to show that $|E|$ is base-point free. It is enough to show that $E$ is nef, \cite[Lemma 2.3]{knut} or \cite[Prop.\ 2.3.10]{huy-lec-k3}. 

So it suffices to show there is no effective divisor $R$ with $(R)^2=-2$ and $(E \cdot R)<0$, \cite[Prop.\ VIII.3.6]{barth}. Suppose for a contradiction that such an $R$ exists. Write $R=xL+y E+ \sum_{i=1}^{8} z_{i} \Gamma_i$ for integers $x,y,z_{i}$.  As $(E \cdot R) = x \left \lfloor \frac{g+1}{2} \right \rfloor <0$ we have $x <0$. Then $(R-x L)^2=-2 \sum_{i=1}^8 z^2_{i}  \leq 0$. However, 
\begin{align*}
(R-x L)^2 &= -2+x^2(2g-2)-2x(L \cdot R) \\
&=-2+x^2(2g-2-2\left \lfloor \frac{g+1}{2} \right \rfloor)-2x(L-E \cdot R) \\
& >0
\end{align*}
for $x <0$ and $g \geq 6$.
Thus $|E|$ is an elliptic pencil. 

Next $\Gamma_1$ is effective, since $(\Gamma_1 \cdot L-E) >0$. We claim $\Gamma_1$ is integral. Otherwise, there would be an integral component $R$ of $\Gamma_1$ with $(R \cdot \Gamma_1) <0$, since $(\Gamma_1)^2=-2$. Further, $(R)^2=-2$, since $R$ is not nef. Write $R=xL+y E+\sum_{i=1}^{8} z_{i} \Gamma_i$. We have $(R \cdot E)=x \left \lfloor \frac{g+1}{2} \right \rfloor \geq 0$ so $x \geq 0$. Assume $x \neq 0$. Then we have $(R \cdot R+E) >0$ and $(R+E)^2>0$ so $R+E$ is big and nef, which contradicts that $(R+E \cdot \Gamma_1)=(R \cdot \Gamma_1)<0$. So $x=0$. But then $(R)^2=-2$ gives $\sum_{i=1}^8 z^2_{i}=1$, and $(R \cdot \Gamma_1)=-2 z_{1} <0$ so $z_{1}=1$ and $z_{i}=0$, $i>1$. Lastly, $(R \cdot L)=d_1+y \left \lfloor \frac{g+1}{2} \right \rfloor \geq 0$ so we must have $y \geq 0$ (as  $d_1< \left \lfloor \frac{g+1}{2} \right \rfloor$). Since $R$ is a smooth and irreducible rational curve, we must then have $y=0$ and $R=\Gamma_1$ (as the only effective divisor in $|R|$ is integral). Thus $\Gamma_1$ is integral. Likewise, $\Gamma_2, \ldots, \Gamma_8$ are integral.

Next, $\widetilde{\Gamma}_1$ is effective, since $(\widetilde{\Gamma}_1)^2=-2$ and $(\widetilde{\Gamma}_1 \cdot L-E) >0$. 
Let $R$ be an integral component of $\widetilde{\Gamma}_1$ such that $(R \cdot \widetilde{\Gamma}_1) <0$, $(R)^2=-2$. 
Writing $R=xL+y E+\sum_{i=1}^{8} z_{i} \Gamma_i$, we see as above $x=0$ and we must have $z_1=-1$ and $z_{i}=0$, $i>1$. Since $( R \cdot L)=-d_1+y \left \lfloor \frac{g+1}{2} \right \rfloor \geq 0$ we must have $y \geq 1$ and then $R \sim \widetilde{\Gamma}_1+(y-1)E$. Since $R$ is integral, this forces $R= \widetilde{\Gamma}_1$.
\end{proof}

\begin{lem} \label{I2}
Let $Y_{\Omega_g}$ and $\{L,E, \Gamma_1, \ldots, \Gamma_8 \}$ be as in the previous lemma, and let $Y_{\Omega_g} \to \proj^1$ be the fibration induced by $E$. If $Y_{\Omega_g}$ is a general $\Omega_g$-polarised K3 surface, then at least six of the reducible singular fibres $\Gamma_i+\widetilde{\Gamma_i}$ for $1 \leq i \leq 8$ are of the type $I_2$ (as opposed to the type $III$). We choose indices such that $\Gamma_i+\widetilde{\Gamma_i}$
is an $I_2$ fibre for $i \geq 3$.
\end{lem}
\begin{proof}
As $( xL+y E+\sum_{i=1}^{8} z_{i} \Gamma_i \cdot E)=x(L \cdot E)$, the elliptic fibration $Y_{\Omega_g} \to \proj^1$ induced by $E$ has multisection index $(L \cdot E)$.
Consider the Jacobian fibration $J(Y_{\Omega_g}) \to \proj^1$, which has the same singular fibres as $Y_{\Omega_g}$ (up to isomorphism). Let $\widetilde{\Omega_g}$ be the lattice generated by $\Omega_g$ and $$F:=E /(L \cdot E).$$ Then $J(Y_{\Omega_g})$ is an element of the global moduli space $M_{\widetilde{\Omega_g}}$ of $\widetilde{\Omega_g}$-polarized K3 surfaces, \cite[Lemma 2.1]{keum} and \cite[Def.\ p.1602]{dolgachev}. A result of Mukai states that one may describe $J(Y_{\Omega_g})$ as a moduli space $M_H(v)$ of $H$-semistable sheaves on $Y_{\Omega_g}$ for a generic polarization $H$ and Mukai vector $(0,E, (L \cdot E))$; see \cite[\S 11.4.2]{huy-lec-k3}, \cite[p.\ 2081]{keum}, \cite{mukai-moduli}. This construction can be done in families by \cite[Thm.\ 4.3.7]{huybrechts-sheaves}, so there is a holomorphic map 
\begin{align*}
\phi \; : \; U &\to M_{\widetilde{\Omega_g}} \\
(Y'_{\Omega_g}, \Omega_g \hookrightarrow \text{Pic}(Y'_{\Omega_g})) & \mapsto (J(Y'_{\Omega_g}), \widetilde{\Omega_g} \hookrightarrow \text{Pic}(J(Y'_{\Omega_g})))
\end{align*} 
defined in a Euclidean open subset $U$ about $(Y_{\Omega_g}, \Omega_g \hookrightarrow \text{Pic}(Y_{\Omega_g}))$ in the local moduli space (or period domain) of marked $\Omega_g$-polarised K3 surfaces. The Tate--\v{S}afarevi\v{c} group of  $J(Y_{\Omega_g})$ is countable by Lemma \ref{tate-shaf}. In particular, this implies that the fibres of $\phi$ are zero-dimensional by \cite[Cor.\ 11.5.5]{huy-lec-k3}. By Sard's theorem, $\phi(U)$ contains a Euclidean open set. Thus it suffices to show that there is a dense open set $V \seq M_{\widetilde{\Omega_g}}$ with the property that for any $(X, \widetilde{\Omega_g} \hookrightarrow \text{Pic}(X)) \in V$, the fibration $X \to \proj^1$ induced by $F$ has at least six $I_2$ fibres. 

Let $T$ be the trivial lattice of $J(Y_{\Omega_g}) \to \proj^1$, i.e.\ the lattice generated by $F$, any section $\sigma$ of  $J(Y_{\Omega_g}) \to \proj^1$ and the components of the reducible fibres which do not meet $\sigma$. We have
$$T \simeq \mathfrak{h} \oplus (-2)^8,$$
where $\mathfrak{h}$ denotes the hyperbolic lattice.
By choosing the basis $\{L-gF, F, \Gamma_1-d_1 F, \ldots, \Gamma_8-d_8 F  \}$, we see $\widetilde{\Omega_g}$ is isometric to $T$, or equivalently, the Mordell--Weil group of $J(Y_{\Omega_g})$ is trivial, \cite[Thm.\ 6.3]{schuett}. Suppose $(X, \widetilde{\Omega_g} \hookrightarrow \text{Pic}(X)) \in  M_{\widetilde{\Omega_g}}$ has the property that $X \to \proj^1$ induced by $F$ has at least six $I_2$ fibres. Then the same holds in a dense open set in each component of $M_{\widetilde{\Omega_g}}$ containing $(X, \widetilde{\Omega_g} \hookrightarrow \text{Pic}(X))$ as the condition that a fibre be nodal is Zariski open (i.e.\ for a flat, proper algebraic family of integral curves, the locus of non-nodal curves is a Zariski closed subset of the base). Furthermore, the same clearly holds for the complex conjugate $X^c \to \proj^1$. There are at most two components of $M_{\widetilde{\Omega_g}}$, which locally on the period domain are interchanged by complex conjugation, so this would complete the proof.

 Thus it suffices to find such an elliptic K3 surface $X \to \proj^1$. From \cite[Thm.\ 2.12]{shimada-arxiv} (published in compressed form as \cite{shimada-mich}), there exists an elliptic K3 surface $X \to \proj^1$ with section and torsion-free Mordell--Weil group such that $X$ has 10 singular fibres of type $A_1$, each of which are either of type $I_2$ or $III$. Using that the Euler number of $III$ is $3$, we find that $X$ can have at most $4$ fibres of type $III$, and thus has at least $6$ fibres of type $I_2$, which have Euler number $2$ (as the sum of the Euler numbers of the singular fibres must be $24$, the Euler number of $X$). There is an obvious primitive embedding of $T$ into the trivial lattice of $X$ such that $6$ of the generators correspond to components of the $I_2$ fibres avoiding the section. As the Mordell--Weil group of $X \to \proj^1$ is torsion free, we have a primitive embedding $T \hookrightarrow \text{Pic}(X)$. This completes the proof. 
\end{proof}

\begin{lem} \label{little-lem}
Let $Y_{\Omega_g}$ and $\{L,E, \Gamma_1, \ldots, \Gamma_8 \}$ be as in the previous lemma. If we assume $g > 7$, then $L-E$ is very ample (and hence $L$ is also very ample).
\end{lem}
\begin{proof}
Suppose the big and nef line bundle $L-E$ is not very ample. Then there exists either a smooth rational curve $R \seq Y_{\Omega_g}$ with $(L-E \cdot R) = 0$ or a smooth elliptic curve $F \seq Y_{\Omega_g}$ with $0 < (L-E \cdot F) \leq 2$ (or both exist), by \cite[Thm.\ 1.1]{knut} (set $k=1$ in Knutsen's theorem and note that $L-E$ is primitive). 

Assume firstly that $R$ as above exists; we may write $R=x_1L+y_1 E+ \sum_{i=1}^{8} z_{1,i} \Gamma_i$ for integers $x_1,y_1,z_{1,i}$. We have
\begin{align*}
-2 \sum_{i=1}^8 z^2_{1,i} &=(R-x_1L)^2 \\
&=-2+x_1^2(2g-2)-2x_1(L \cdot R) \\
&=-2+x_1^2(2g-2-2\left \lfloor \frac{g+1}{2} \right \rfloor)-2x_1(L-E \cdot R).
\end{align*}
By the above equality, using that $g \geq 8$ and $(L-E \cdot R)=0$, we find that $x_1=0$ and there exists some $j$ such that $z_{1,j}=\pm 1$, $z_{1,i}=0$ for $i \neq j$. Then $0=( L-E \cdot R)=\pm d_j+y_1 \left \lfloor \frac{g+1}{2} \right \rfloor$ which is impossible
for $1 \leq d_j < \left \lfloor \frac{g+1}{2} \right \rfloor$. 

So now suppose there is some smooth elliptic curve $F \seq Y_{\Omega_g}$ with $0 < (L-E \cdot F) \leq 2$. We may write $F=x_2 L+y_2 E+ \sum_{i=1}^{8} z_{2,i} \Gamma_i$ for integers $x_2,y_2,z_{2,i}$. We have $(F \cdot E)=x_2 (L \cdot E)$ and hence $x_2 >0$ (as $F \notin |E|$).  We calculate
\begin{align*}
-2 \sum_{i=1}^8 z^2_{2,i}  &= (F-x_2L)^2 \\
&= x_2^2(2g-2)-2x_2(L \cdot F) \\
&= x_2(x_2((2g-2)-2 \left \lfloor \frac{g+1}{2} \right \rfloor)-2 (L-E \cdot F)) 
\end{align*}
which is impossible for $g>7$, $0 \leq (L-E \cdot F) \leq 2$. Thus $L-E$ is very ample. Using Knutsen's criterion again, and the fact that $E$ is nef, we see that $L$ is likewise very ample.
\end{proof}

For the rest of the section we will assume $g$ is odd. The following technical lemma will be needed later in this section.
\begin{lem} \label{little-lem2}
Assume $g \geq 11$  is odd and let $Y_{\Omega_g}$ and $\{L,E, \Gamma_1, \ldots, \Gamma_8 \}$  be as in Lemma \ref{little-lem}. Then $L-2E$ is not effective. Further $(L-E)^2 \geq 8$ and there exists no effective divisor $F$ with $(F)^2=0$ and $(F \cdot L-E) \leq 3$.
\end{lem}
\begin{proof}
Suppose $L-2E$ is an effective divisor and let $D_1, \ldots D_k$ be its irreducible components. Write $D_i=x_iL+y_iE+\sum_{j=1}^8 z_{i,j} \Gamma_j$ for integers $x_i,y_i,z_{i,j}$. Then $0 \leq (D_i \cdot E)=x_i (L \cdot E)$ so that $x_i \geq 0$ and $\sum_i x_i=1$. Thus we may assume $x_1=1$ and $x_i=0$ for all $i \geq 2$. Now let $\widetilde{D}$ be any irreducible curve of the form $a E+\sum_{j=1}^8 b_j \Gamma_j$ for integers $a,b_j$ and suppose $\widetilde{D} \neq \Gamma_j$, $\forall 1 \leq j \leq 8$. Then $0 \leq (\widetilde{D} \cdot \Gamma_j)=-2b_j$ so $b_j \leq 0$ for all $j$. Since $(\widetilde{D})^2=-2 \sum_{j=1}^8 b^2_j \geq -2$ by \cite[Prop.\ VIII 3.6]{barth}, there is at most one $b_j$ such that $b_j \neq 0$, and in this case $b_j=-1$. Suppose firstly that all $b_j=0$. Then $\widetilde{D} \sim E$, since $\widetilde{D}$ is integral, and all effective divisors in $|aE|$ are a sum of $a$ divisors in $|E|$, \cite[Prop.\ 2.6(ii)]{donat}. Next suppose $b_j=-1$. Then $\widetilde{D}=aE-\Gamma_j$ and $a \geq 1$ since $a$ is effective. Thus $\widetilde{D}=\widetilde{\Gamma_j}+(a-1)E$ and $(\widetilde{D})^2=-2$, which implies $a=1$ since $\widetilde{D}$ is a smooth and irreducible rational curve (as the unique effective divisor in the linear system $|\widetilde{D}|$ is integral). 

Thus if $i \geq 2$, $D_i$ is either $E$, $\Gamma_j$ or $\widetilde{\Gamma}_j$, for some $j$. Since $\sum_i D_i=L-2E$, we see  $D_1=L-(2+m')E-\sum_{j=1}^8 n'_{1,j} \Gamma_j-\sum_{j=1}^8 n'_{2,j} \widetilde{\Gamma}_j$ for nonnegative integers $m'$ and $n'_{1,j}$, $n'_{2,j}$, $1 \leq j \leq 8$. Since $E=\Gamma_j+\widetilde{\Gamma}_j$, we may rewrite $D_1$ in the form $D_1=L-(2+m)E-\sum_{j=1}^8 (n_{1,j} \Gamma_j+n_{2,j} \widetilde{\Gamma}_j)$, where $m$, $n_{1,j}$, $n_{2,j}$ are nonnegative integers and if $n_{1,j_1} \neq 0$ for some $j_1$ then $n_{2,j_1}=0$, and likewise if $n_{2,j_2} \neq 0$ then $n_{1,j_2} = 0$.
But then one computes
\begin{align*}
(D_1)^2&=(L-2E-(mE+\sum_{j=1}^8 n_{1,j} \Gamma_j+n_{2,j} \widetilde{\Gamma}_j))^2 \\
&=-4-2\sum_{j=1}^8 (n_{1,j}^2+ n_{2,j}^2) -2(L \cdot mE+\sum_{j=1}^8 n_{1,j} \Gamma_j+n_{2,j} \widetilde{\Gamma}_j) \\
& \leq -4
\end{align*}
which is a contradiction (since $(D)^2 \geq -2$ for any integral curve $D$). This proves that $L-2E$ is non-effective.

We have $(L-E)^2=g-3 \geq 8$ for $g \geq 11$. For $g \geq 11$ and $a'>0$ one has $(g-3)a'-6>0$. From the proof of Lemma \ref{little-lem}, this implies there is no effective $F$ with $(F)^2=0$ and $(F \cdot L-E) \leq 3$. 
\end{proof}

\begin{lem} \label{lem-xyzw}
Let $D \in \Omega_g$ be an effective divisor with $(D)^2 \geq 0$, and assume $L-D$ is effective and $(L-D)^2 >0$. Then $D=cE$ for some integer $c \geq 0$.
\end{lem}
\begin{proof}
Write $D=xL+yE+\sum_{i=1}^8 z_i \Gamma_i$ for integers $x,y,z_i$.  One has $0 \geq (D-L \cdot E)=(x-1)(L \cdot E)$ so that $x \leq 1$. On the other hand $0 \leq (D \cdot E)=x(L \cdot E)$ so that $x \geq 0$. Thus $x=0$ or $x=1$. Suppose firstly that $x=1$. Then $0 < (D-L)^2=-2\sum_{i=1}^8 z_i^2$ which is a contradiction. Hence $x=0$. Then $0 \leq (D)^2=-2(\sum_{i=1}^8 z_i^2)$ and so $z_i=0$ for all $i$, as required.
\end{proof}

\section{Mukai's theory for curves on $Y_{\Omega_g}$} \label{muk-theory}
In this section we will apply a construction of Mukai to the special K3 surface $Y_{\Omega_g}$ in order to construct loci $Z \seq \mathcal{V}^0_g$ such that  for $x \in Z$ the fibre of  
$\eta : \mathcal{V}^0_g \to \mathcal{M}_{g}$ over $\eta(x)$ is zero-dimensional at $x$. This will be our basic tool for studying the generic finiteness of the morphism $\eta: \mathcal{T}^n_{g,k} \to \mathcal{M}_{p(g,k)-n}$. 

Our first task is to study Clifford index of curves on the elliptic K3 surfaces $Y_{\Omega_g}$ from Lemma \ref{little-lem2}.
For any smooth curve $C$ and $M \in \text{Pic}(C)$ with $\deg(M)=d$ and $h^0(M)=r+1$, let $\nu(M):=d-2r$. The \emph{Clifford index} of $C$ is defined by $$\nu(C):=\text{min}\{\nu(M) \; | \; M \in \text{Pic}(C) \text{ with $\deg(M) \leq g-1$, $h^0(M) \geq 2$}  \}.$$ Clifford's Theorem states that $\nu(C) \geq 0$ and $\nu(C)=0$ if and only if $C$ is hyperelliptic.
\begin{lem} \label{gon-omega}
Let $g \geq 11$ be odd, let $Y_{\Omega_g}$ and $\{L,E, \Gamma_1, \ldots, \Gamma_8 \}$ be as in Lemma \ref{little-lem2} and let $C \in |L|$ be a smooth curve. Then $\nu(C)=\frac{g+1}{2}-2$.
\end{lem}
\begin{proof}
Consider the line bundle $A:=\mathcal{O}_C(E)$. We have $h^0(A)=2$, since $h^1(L-E)=0$ by Kodaira's vanishing theorem. As $\deg(A)=\frac{g+1}{2}$, we see 
$\nu(C) \leq \frac{g+1}{2}-2$. Suppose for a contradiction that $\nu(C) < \frac{g+1}{2}-2$. From \cite[Lem.\ 8.3]{knut}, there is a smooth and irreducible curve $D \seq Y_{\Omega_g}$ with $0 \leq (D)^2 < \nu(C)+2$, $2(D)^2 < (D \cdot L)$ and $\nu(C)= (D \cdot L)-(D)^2-2$ (the sharp inequalities here are due to the fact that $L$ is primitive). We have $$(D-L \cdot D)=(D)^2-(D \cdot L)=-2-\nu(C) \leq -2$$ and $$(D-L)^2=-6-2\nu(C)-(D)^2+2g > -8-3\nu(C)+2g \geq 0,$$ as $\nu(C) \leq \frac{g+1}{2}-2$ and $g \geq 7$. Hence $L-D$ is effective and $D=cE$ for some $c \geq 0$ by Lemma \ref{lem-xyzw}. As $D$ is represented by a smooth curve, we must have $D=E$. But then $$\nu(C)=(E \cdot L)-(E)^2-2=\frac{g+1}{2}-2,$$ giving a contradiction.
\end{proof}

\begin{lem} \label{gon-omega-2}
Assume $g \geq 11$ is odd and let $Y_{\Omega_g}$ be a K3 surface with $\text{Pic}(Y_{\Omega_g}) \simeq \Omega_g$ as in Lemma \ref{little-lem2}. Let $M \in \Omega_{g}$ be an effective line bundle on $Y_{\Omega_g}$ satisfying 
\begin{itemize}
\item $0 \leq (M)^2 < \frac{g+1}{2}$
\item $2 (M)^2 < (M \cdot L)$
\item $\frac{g+1}{2}= (M \cdot L)-(M)^2.$
\end{itemize}
Then $M=E$.
\end{lem}
\begin{proof}
The above inequalities give $(M-L \cdot M) < 0$ and $$(M-L)^2=(M)^2+2g-2-2((M)^2+\frac{g+1}{2}) > 0$$ as $(M)^2 < \frac{g+1}{2} \leq g-3$ for $g \geq 7$. Thus $M=cE$ for $c >0$ by Lemma \ref{lem-xyzw}. The equation $\frac{g+1}{2}= (M \cdot L)-(M)^2$ gives $c=1$.
\end{proof}

\begin{lem} \label{unique-special}
Let $g \geq 11$ be odd and let $Y_{\Omega_g}$ and $\{L,E, \Gamma_1, \ldots, \Gamma_8 \}$ be as in Lemma \ref{little-lem2} and let $C \in |L|$ be a smooth curve. Suppose $A \in \text{Pic}(C)$ has $h^0(A)=2$ and $deg(A)= \frac{g+1}{2}$. Then $A \simeq E_{|_C}$.
\end{lem}
\begin{proof}
Let  $A \in \text{Pic}(C)$ with $h^0(A)=2$ and $deg(A)= \frac{g+1}{2}$. In particular, $\nu(A)=\nu(C)$. Then by \cite[Thm.\ 4.2]{donagi-morrison}, there is some effective divisor $D \seq Y_{\Omega_g} $ such that that $(D \cdot L) \leq g-1$, $L-D$ is effective, $\dim |D| \geq 1$, $D|_C$ achieves the Clifford index on $C$ and such that there exists a reduced divisor $Z_0 \in |A|$ of length $\frac{g+1}{2}$ with $Z_0 \seq D \cap C$. Further from the proof of \cite[Lemma 4.6]{donagi-morrison} and the paragraph proceeding it, we see $\frac{g+1}{2}= (D \cdot L)-(D)^2,$ using that $\nu(D|_C)= \frac{g+1}{2}-2$ by hypothesis. It then follows that all the conditions of Lemma \ref{gon-omega-2} are satisfied, so that $D \in |E|$. As $(E \cdot C)=\frac{g+1}{2}$ and $Z_0$ is reduced of length $\frac{g+1}{2}$,  we have $Z_0 = D \cap C$. Thus $Z_0 \in |E_{|_C}|$ which forces $A \simeq E_{|_C}$.
\end{proof}

The following lemma may be extracted from work of Mukai, cf.\ \cite[\S3]{mukai-duality}, \cite[Lem.\ 2]{mukai-genus11} (although it never appears in this precise form). Despite the simple proof, this lemma is actually rather fundamental, since it gives an example of  a K3 surface $S$ and a divisor $D \seq S$ such that the K3 surface $S$ can be reconstructed merely from the curve $D$ together with a special divisor $A \in \text{Pic}(D)$.
\begin{lem}[Mukai] \label{muk-lem}
Let $g \geq 11$ be odd and let $S$ be a K3 surface with $L, M \in \text{Pic}(S)$ such that $L^2=2g-2$, $M^2=0$ and $(L \cdot M)=\frac{g+1}{2}$. Let $D \in |L|$ be smooth and set $A:=M_{|_D}$. Further, let $A^{\dagger}=(L-M)_{|_D}$ be the adjoint of $A$ and set $k:=h^0(A^{\dagger})-1$. Assume $L-M$ is very ample and that there is no integral curve $F \seq S$ with $(F)^2=0$ and $(F \cdot L-M)=3$. Assume further that $M$ is represented by an integral curve and that $L-2M$ is not effective. Then $A^{\dagger}$ is very ample, and $S$ is the quadric hull of the embedding $\phi_{A^{\dagger}}: D \hookrightarrow \proj^k$ induced by $A^{\dagger}$. In particular, the linear system $|L|$ contains only finitely many curves $D'$ such that we have an isomorphism $\iota :  D' \simeq D$ with $\iota^*A \simeq M_{|_{D'}}$.
\end{lem}
 \begin{proof}
 As there exists an integral curve in $|M|$, we know $h^1(M)=0$. From the exact sequence
 $$ 0 \to M^* \to L \otimes M^* \to  A^{\dagger} \to 0 $$
 we see that restriction induces an isomorphism $H^0(S,L\otimes M^*) \simeq H^0(D,A^{\dagger})$, which implies that $A^{\dagger}$ is very ample (as $L \otimes M^*$ is) and that $\phi_{A^{\dagger}}: D \hookrightarrow \proj^k$ is the composition of $D \hookrightarrow S$ and $\phi_{L-M}: S \hookrightarrow \proj^k$. Now let $Q \seq \proj^k$ be a quadric containing $D$. Then we claim that $Q$ contains $S$. Indeed, otherwise there is some effective divisor $T \seq S$ such that $D+T=Q \cap S$ so that $D+T \in |2L-2M|$. But then $T \in |L-2M|$, contradicting that $L-2M$ is not effective by assumption. Thus the quadric hull of $D$ coincides with the quadric hull of $S \hookrightarrow \proj^k$. But the quadric hull of $S$ is simply $S$ from \cite[Thm.\ 7.2]{donat}. The assumption $g \geq 11$ is needed to ensure $(L-M)^2 \geq 8$ which is required in Saint-Donat's theorem (note that since $L-M$ is very ample by assumption it is not hyperelliptic in the sense of \cite[\S 4.1, Thm.\ 5.2]{donat}). Finally, by the argument above, we have that, for any curve $D' \in |L|$ such that there exists an isomorphism $\iota :  D' \simeq D$ with $\iota^*A \simeq M_{|_{D'}}$, there is an automorphism $f: S \to S$ induced by an automorphism of $\proj^k$ such that $f(D')=D$. Since $S$ is a K3 surface, we deduce that there are only finitely many such $D'$.
 \end{proof}
 
 Putting all the pieces together, we have the following proposition:
\begin{prop} \label{good-dim-prop}
Let $g \geq 11$ be odd. Let $T$ be a smooth and irreducible scheme with base point $0 \in T$. Let $\mathcal{Z} \to T$ be a flat family of K3 surfaces together with an embedding $\phi: \mathcal{Z} \hookrightarrow T \times \proj^g$. Assume $\phi_0: \mathcal{Z}_0 \to \proj^{g}$ is the embedding $ Y_{\Omega_g} \to \proj^g$ induced by $|L|$, where $ Y_{\Omega_g}$ is as in Lemma \ref{little-lem2}. Let $H \seq \proj^g$ be a hyperplane and $C:= H \cap Y_{\Omega_g}$ a smooth curve. Assume that for all $t \in T$, we have a smoothly varying family of hyperplanes $H_t$ with $H_0=H$ and $\mathcal{Z}_t \cap H_t \simeq C$. Then for $t \in T$ general there is an isomorphism $ \psi_t: \mathcal{Z}_t \simeq Y_{\Omega_g}$ and an automorphism $f_t \in Aut(\proj^g)$ such that $f_t \circ \phi_t = \phi_0 \circ \psi_t$.  Furthermore, $f_t$ sends $H_t$ to $H$.
\end{prop}
\begin{proof}
For $t \in T$ general we have a primitive embedding $j: \text{Pic}(\mathcal{Z}_t) \hookrightarrow \text{Pic}(Y_{\Omega_g})$ with $j(\mathcal{O}_{\mathcal{Z}_t}(1))=L$. By hypothesis, each fibre $\mathcal{Z}_t$ contains a curve isomorphic to $C$ 
which has Clifford index $\frac{g+1}{2}-2$ by Lemma \ref{gon-omega}. By \cite[Lem.\ 8.3]{knut} there is some smooth and irreducible divisor $M' \in \text{Pic}(\mathcal{Z}_t)$ satisfying
$0 \leq (M')^2 < \frac{g+1}{2}$, $2 (M')^2 < (M' \cdot \mathcal{O}_{\mathcal{Z}_t}(1))$ and $\frac{g+1}{2}= (M' \cdot \mathcal{O}_{\mathcal{Z}_t}(1))-(M')^2$. As in the proof of Lemma \ref{gon-omega}, these conditions ensure that $\mathcal{O}_{\mathcal{Z}_t}(1)-M'$ is effective. Moreover, $M:=j(M')$ satisfies the conditions of Lemma \ref{gon-omega-2}, so that $M=E$. Thus $(M')^2=0$ and hence $M'$ is a smooth elliptic curve. By lemmas \ref{little-lem} and \ref{little-lem2}, $L-E$ is very ample,  $(L-E)^2 \geq 8$, $L-2E$ is not effective and there exists no effective $F$ with $(F)^2=0$ and $(F \cdot L-E)=3$; thus the same holds for $\mathcal{O}_{\mathcal{Z}_t}(1)-M' \in \text{Pic}(\mathcal{Z}_t)$ for $t$ close to $0$. Thus by Lemma \ref{muk-lem} and since $(\mathcal{O}_{\mathcal{Z}_t}(1)-M')_{|_C} \simeq K_C(E^*_{|_C})$  by Lemma \ref{unique-special} we have $ \mathcal{Z}_t \simeq Y_{\Omega_g}$. Furthermore, the embedding  $C \hookrightarrow Z_t$ is identified under this isomorphism with the natural embedding of $C$ into the quadric hull of the embedding $\phi_{K_C(E^*_{|_C})}: C \hookrightarrow \proj^k$. In particular, the embedding is independent of $t$ (up to the action of projective transformations on $\proj^g$). Thus for the general $t$ there is some $f_t \in Aut(\proj^g)=Aut(|L|)$ with $f_t(Z_t)=Y_{\Omega_g}\seq \proj^g$ such that $f_t$ sends $H_t$ to $H$.
\end{proof}
As a direct consequence we have the following corollary.
\begin{cor} \label{bij-cor}
Let $g \geq 11$ be odd and let $Y_{\Omega_g}$ and $\{L,E, \Gamma_1, \ldots, \Gamma_8 \}$ be as in Lemma \ref{little-lem2} and let $C \in |L|$ be a smooth curve. Then the fibre of the morphism $\eta: \mathcal{T}_g^0 \to \mathcal{M}_{g}$ over $[C]$ is zero-dimensional at $[(i:C \hookrightarrow Y_{\Omega_g}, L)]$.
\end{cor}
\begin{remark}
In fact, $\eta : \mathcal{T}_g^0 \to \mathcal{M}_{g}$ is birational onto its image for $g \geq 11$, $g \neq 12$, \cite{ciliberto-lopez-miranda}. Also see \cite[\S 10]{mukai-nonabelian}, \cite{sernesi-brill} for an approach in the odd genus case which more closely resembles the above.
\end{remark}

\chapter{The moduli map} \label{finny}
In this chapter we prove generic finiteness of the moduli map 
$$\eta: \mathcal{T}^m_{g,k} \to \mathcal{M}_{p(g,k)-m}$$ on one component for all but finitely many values of $p(g,k)-m $. The idea is to start with
the case of a smooth, primitive, genus $11$ curve $C$ lying on a K3 surface, i.e.\ the case $k=1, g=11, m=0$. The result has already been established for these values in Corollary \ref{bij-cor}. We then build up from this result by adding rational components to the curve $C$ and partially normalising at some of the nodes.
Using that rational curves on a K3 surface are rigid, this ultimately gives generic finiteness on one component, for all but finitely many values of $p(g,k)-m $.
\section{The moduli map}
We start by reducing to the case where the number of nodes is large.
\begin{lem} \label{red-largen}
Let $n \leq m$ with $p(g,k)-m>0$. Assume that there is a component $I_m \seq \mathcal{T}^m_{g,k}$ such that 
$${\eta}_{|_{I_m}}: I_m \to \mathcal{M}_{p(g,k)-m}$$ is generically finite and assume that for the general $[f: B \to X] \in I_m$, $B$ is non-trigonal. Then there exists a component $I_n \seq \mathcal{T}^n_{g,k}$
such that  $${\eta}_{|_{I_n}}: I_n \to \mathcal{M}_{p(g,k)-n}$$ is generically finite. For the general $[f': C \to Y] \in I_n$, $C$ is non-trigonal.
\end{lem}
\begin{proof}
Let $[(f: B \to X,L)] \in I_m$ be a general point. By \cite[Thm.\ B]{ded-sern}, $f(B)$ is a nodal curve (with precisely $m$ nodes). Thus there is an integral, nodal curve $B'$ with $m-n$ nodes with normalization $\mu: B \to B'$ such that $f$ factors through a morphism $\tilde{g}: B' \to X$. The fibre of $\eta$ over the stable curve $[B'] \in \overline {\mathcal{M}}_{p(g,k)-n}$ is zero-dimensional near $[(\tilde{g},L)]$ as otherwise we could compose with $\mu$ to produce a one dimensional family near $[(f: B \to X,L)]$ in the fibre of $\eta$ over $[B]$. Since $[(\tilde{g},L)]$ lies in the closure of $\mathcal{T}^n_{g,k}$ by Lemma \ref{defo-nodal-lemma}, we see that there exists a component $I_n \seq \mathcal{T}^n_{g,k}$
such that  $${\eta}_{|_{I_n}}: I_n \to \mathcal{M}_{p(g,k)-n}$$ is generically finite. As the normalization of $B'$ is non-trigonal by assumption, $B'$ must lie outside the image of 
$$\overline{\mathcal{H}}_{3, p(g,k)-n} \to \overline{\mathcal{M}}_{p(g,k)-n}, $$
in the notation of \cite[Thm.\ 3.150]{harris-morrison}. Thus, for the general $[f': C \to Y] \in I_n$, $C$ is non-trigonal.
\end{proof}

The following proposition gives a criterion for generic finiteness of the morphism
$$\eta: \mathcal{T}^n_{g,k} \to \mathcal{M}_{p(g,k)-n}$$ on one component $I$. The idea is to assume we have an unramified map $f_0: C_0 \to X$ representing a point in $\mathcal{W}^{n'}_{g'}$ such that finiteness of $\eta$ holds near the point representing $f_0$. If we then build a new morphism $f: C_0 \cup \proj^1 \to X$ by finding a rational curve $f(\proj^1)$ in $X$, and if we further assume $C_0 \cup \proj^1$ is a stable curve (i.e.\ $\proj^1$ intersects $C_0$ in at least three points), then by rigidity of rational curves in $X$, one sees easily that finiteness of $\eta$ holds near the point representing $f$, where $n,k$ are such that $f$ represents a point in $\mathcal{W}^n_{g,k}$.
\begin{prop} \label{finiteness-criterion}
Assume there exists a polarised K3 surface $(X,L)$ and an unramified stable map $f: B \to X$ with
$f_*(B) \in |kL|$. Assume:
\begin{enumerate}
\item $[(f: B \to X, L)]$ lies in the closure of $\mathcal{T}^{n}_{g,k}$.
\item There exists an integral, nodal component $C \seq B$ of arithmetic genus $p' \geq 2$ such that $f_{|_C}$ is an unramified morphism $j: C \to X$, birational onto its image. Let $k'$ be an integer such that there is a big and nef line bundle $L'$ on $X$ with $j_*(C) \in |k'L'|$, and let $g'=\frac{1}{2} (L')^2+1$.
\item The fibre of the morphism $\eta: \mathcal{W}^{n'}_{g',k'} \to \overline{\mathcal{M}}_{p'}$ over $[C]$ is zero-dimensional near $[(j: C \to X, L')]$, where $n'=p(j(C))-p'$.
\item If $\tilde{C}$ is the normalization of $C$, then $\tilde{C}$ is non-trigonal.
\item If $D \seq B$ is a component, $D \neq C$, then $D$ has geometric genus zero.
\item The stabilization morphism $B \to \hat{B}$ is an isomorphism in an open subset $U \seq B$ such that $C \seq U$.
\end{enumerate} 
Then there exists a component $I \seq \mathcal{T}^n_{g,k}$ such that ${\eta}_{|_I}$ is generically finite and $[(f: B \to X, L)]$ lies in the closure of $I \seq \mathcal{W}^{n}_{g,k}$. For the general $[f': B'\to X'] \in I$, $B'$ is non-trigonal.
\end{prop}
\begin{proof}
We have a morphism $\eta: \mathcal{W}^{n}_{g,k} \to \overline{\mathcal{M}}_{p(g,k)-n}$. By assumption $1$, it suffices to show $\eta^{-1}([\hat{B}])$ is zero-dimensional near $[(f: B \to X, L)]$. Let $S$ be a smooth, irreducible, one-dimensional scheme with base point $0$, and suppose we have a commutative diagram
 \[
\xymatrix{
\mathcal{B} \ar[r]^{\tilde{g}}  \ar[rd]_{\pi_1} & \mathcal{X} \ar[d]^{\pi_2} \\
&S
}
\] 
with $\tilde{g}$ proper, $\pi_1$, $\pi_2$ flat and with $\tilde{g}_s: \mathcal{B}_s \to \mathcal{X}_s$ an unramified stable map to a K3 surface for all $s$, with $\tilde{g}_0=f$. Further assume $\hat{\mathcal{B}_s} \simeq \hat{B}$. For any $s \in S$, we have open subsets $U_s \seq \mathcal{B}_s$, $V_s \seq \hat{B}$ with the stabilization map inducing isomorphisms $U_s \simeq V_s$ and such that $\hat{B} \setminus V_s$ has zero-dimensional support. By assumption $6$, $C \seq V_0$, and thus for $s$ near $0$, $C \seq V_s \simeq U_s$. Thus, after performing a finite base change about $0 \in S$, there exists an irreducible component $\mathcal{C} \seq \mathcal{B}$, such that $\mathcal{C}_s \simeq C$, for all $s$ near $0$. We have a commutative diagram
 \[
\xymatrix{
\mathcal{C} \ar[r]^{h}  \ar[rd]_{{\pi_1}_{|_{\mathcal{C}}}} & \mathcal{X} \ar[d]^{{\pi_2}} \\
&S
}
\]
where ${\pi_1}_{|_{\mathcal{C}}}$ is flat and $h:=\tilde{g}_{|_{\mathcal{C}}}$. Since $h_0=j$, assumption $3$ gives $\mathcal{X}_s \simeq X$ and $h_s=j$ for all $s$. In particular, $\tilde{g}_s$ is a one-dimensional family of unramified morphisms into a \emph{fixed} K3 surface. From assumption $5$, the fact that rational curves on a K3 surface are rigid and since $\tilde{g}_s$ is unramified, we conclude that $\mathcal{B}_s \simeq B$ and $\tilde{g}_s: B \to X$ is independent of $s$. Thus  $\eta^{-1}(\hat{B})$ is zero-dimensional near $[(f: B \to X, L)]$. Hence there exists a component $I \seq \mathcal{T}^n_{g,k}$ such that ${\eta}_{|_I}$ is generically finite and $[(f: B \to X, L)]$ lies in the closure of $I \seq \mathcal{W}^{n}_{g,k}$. Since $\tilde{C}$ is non-trigonal, the stabilization $\hat{B}$ must lie outside the image of 
$$\overline{\mathcal{H}}_{3, p(g,k)-n} \to \overline{\mathcal{M}}_{p(g,k)-n}.$$ Thus, for the general $[f': B'\to X'] \in I$, $B'$ is non-trigonal.
\end{proof}

We will apply the above criterion to prove generic finiteness of $\eta$ on one component, for various bounds on $p(g,k)-n$. We first consider the case $k=1$. To begin, we will need an easy lemma. Let $p >h \geq 8$ be integers, and let $l,m$ be nonnegative integers with
$$p-h= \left \lfloor \frac{h+1}{2} \right \rfloor l+m $$
and $0 \leq m < \left \lfloor \frac{h+1}{2} \right \rfloor $. Define:
\begin{align*}
s_1:= \begin{cases}
   p-h-1, & \text{if $m=0$ or $m=\left \lfloor \frac{h+1}{2} \right \rfloor -1$}\\
    p-h+1, & \text{otherwise}.
  \end{cases}
\end{align*} 
Let $P_{p,h}$ be the rank three lattice generated by elements $\{M, R_1, R_2 \}$ and with intersection form given with respect to this ordered basis by:
\[ \left( \begin{array}{ccc}
2h-2 & s_1 & 3 \\
s_1 & -2 & 0 \\
3 & 0 & -2 \end{array} \right).\]
The lattice $P_{p,h}$ is even of signature $(1,2)$.
\begin{lem} \label{onenodal-lem-a}
Let $p>h \geq 8$. There exists a K3 surface $S_{p,h}$ with $\text{Pic}(S_{p,h}) \simeq P_{p,h}$ as above such that the classes $M, R_1, R_2 $ are each represented by integral curves and with $M$ very ample. If $h $ is odd and at least $11$, then for $D \in |M|$ general the fibre of $\eta: \mathcal{T}_h^0 \to \mathcal{M}_{h}$ is zero-dimensional at $[(i:D \hookrightarrow S_{p,h}, M)]$ and $D$ is non-trigonal. Furthermore, any divisor of the form $-xM+yR_1+zR_2 $ for integers $x,y,z$ with $x>0$ is not effective.
\end{lem}
\begin{proof}
Consider the K3 surface $Y_{\Omega_h}$ from Lemma \ref{lem-aaa}. We choose $d_1:=(L \cdot \Gamma_1)$ to be $m+1$ if $0 < m <  \left \lfloor \frac{h+1}{2} \right \rfloor  -1$ and $d_1=\left \lfloor \frac{h+1}{2} \right \rfloor  -1$ if $m=0$ and $d_1=\left \lfloor \frac{h+1}{2} \right \rfloor  -2$ if $m=\left \lfloor \frac{h+1}{2} \right \rfloor -1$. Further set $d_2=3$ and let all other $d_i$ be arbitrary integers in the range $1 \leq d_i < \left \lfloor \frac{h+1}{2} \right \rfloor$. We define a primitive embedding
$j: P_{p,h} \hookrightarrow \Omega_h$ as follows. If $m=0$ (so that $l \geq 1$ as $p>h$), we define the embedding via
$M \mapsto  L$, $R_1 \mapsto (l-1)E+\Gamma_1$, $R_2 \mapsto \Gamma_2$. If $m \neq 0$ we define the embedding via
$M \mapsto  L$, $R_1 \mapsto lE+\Gamma_1$, $R_2 \mapsto \Gamma_2$.
Let $M_{P_{p,h}}$ be the moduli space of ample, $P_{p,h}$-polarised K3 surfaces, \cite[Def.\ p.1602]{dolgachev}. The moduli space $M_{P_{p,h}}$ has dimension $17=19-2$, \cite[Prop.\ 2.1]{dolgachev}.  Let $M_1$ be a component containing the ample, $P_{p,h}$-polarised K3 surface $[Y_{\Omega_h}]$. Then the general point of $M_1$ represents a projective K3 surface $S_{p,h}$ with $\text{Pic}(S_{p,h}) \simeq P_{p,h}$, \cite[Cor.\ 1.9, Cor.\ 2.9]{morrison-large}. Further,  $M, R_2 $ are each represented by integral curves and $M$ is very ample by Lemmas \ref{lem-aaa} and \ref{little-lem}. If $h $ is odd and at least $11$, the statement about the fibres of $\eta$ follows from Corollary \ref{bij-cor} by semicontinuity, and the non-trigonality follows from Lemma \ref{gon-omega}. 

We now claim that any divisor of the form $-xM+yR_1+zR_2$ for integers $x,y,z$ with $x>0$ is not effective. By degenerating $S_{p,h}$ to $Y_{\Omega_h}$ as above, it suffices to show that $-xL +y j(R_1)+z\Gamma_2 \in \text{Pic}(Y_{\Omega_h})$ is not effective. But this is clear, since the rank ten lattice $\text{Pic}(Y_{\Omega_h})$ contains the class of a smooth, integral, elliptic curve $E$ with $( E \cdot -xL +y j(R_1)+z\Gamma_2)=-x (E \cdot L) <0$.

The $-2$ class $R_1$ is effective since $(R_1)^2=-2$, $(R_1 \cdot M) >0$. It remains to show that $R_1$ is integral. Let $D$ be any integral component of $R_1$ with $(D)^2=-2$, $(D \cdot R_1)<0$. Writing $D= xM+yR_1+zR_2$ we see that $x=0$ (as $R_1-D$ is effective). Thus $(D \cdot R_1) <0$ implies $y>0$ and $-2=(D)^2=-2(y^2+z^2)$ forces $z=0$; thus $D=R_1$ is integral.
\end{proof}
\begin{thm} \label{finiteness}
Assume $g \geq 11$, $n \geq 0$, and let $0 \leq r(g) \leq 5$ be the unique integer such that 
$$g-11 = \left \lfloor \frac{g-11}{6} \right \rfloor 6 +r(g).$$
Define 
\begin{itemize}  
\item $l_g:=12$, if $r(g)=0$.
\item $l_g:=13$, if $1 \leq r(g) <5$.
\item $l_g:=15$ if $r(g)=5$.
\end{itemize}
Then there is a component $I \seq \mathcal{T}^n_{g}$ such that $${\eta}_{|_I}: I \to \mathcal{M}_{g-n}$$ is generically finite for
$g-n \geq l_g$. For the general $[f': B' \to X'] \in I$, $B'$ is non-trigonal.
\end{thm}
\begin{proof}
From Lemma \ref{red-largen} it suffices to prove the result for the maximal value of $n$. Assume $g-n = l_g$ if $r(g) \neq 0$ and for $g-n = 15$ if $r(g)=0$. 
Set $p=g$, $h=11$ and consider the lattice $P_{g,11}$ and K3 surface $S_{g,11}$ from Lemma \ref{onenodal-lem-a}. Set $\epsilon=1$ if $r(g)=0$ or $r(g)=5$ and $\epsilon=0$ otherwise. Then $(M+R_1+\epsilon R_2)^2=2g-2$ and $M+R_1+\epsilon R_2$ is ample. Let $D \in |M|$ be general and consider the curve $D_1=D \cup R_1 \cup \epsilon R_2$ where all intersections are transversal. Choose any subset of $s_1-3$ distinct points of $D \cap R_1$ and let $f: B \to D_1$ be the partial normalization at the chosen points. Then $B$ has arithmetic genus $l_g$ if $r(g) \neq 0$ and genus $15$ if $r(g)=0$. Further, $f$ satisfies the conditions of Proposition \ref{finiteness-criterion}. Then there is a component $I \seq \mathcal{T}^n_{g}$ such that $${\eta}_{|_I}: I \to \mathcal{M}_{g-n}$$ is generically finite for
$g-n \geq l_g$ if $r(g) \neq 0$ and for $g-n \geq 15$ if $r(g)=0$. For the general $[f': B' \to X'] \in I$, $B'$ is non-trigonal.

We now wish to improve the bound in the case $r(g)=0$.
Recall from Section \ref{mukai} the lattice $\Omega_{11}$ with ordered basis $\{L, E, \Gamma_1, \ldots, \Gamma_8 \}$. Thus the general $C \in |L|$ is a smooth, genus $11$ curve and $(L \cdot E)=6$. Note that $\widetilde{\Gamma}_i \sim L-E$ is a class satisfying $(\widetilde{\Gamma}_i)^2=-2$, $(\widetilde{\Gamma}_i \cdot E)=0$, $(\widetilde{\Gamma}_i \cdot \Gamma_i)=2$. From Lemma \ref{lem-aaa}, $\widetilde{\Gamma}_i$ is represented by an integral class. Further $\widetilde{\Gamma}_i+\Gamma_i$ is an $I_2$ singular fibre of $|E|$ for any $i \geq 3$. We will denote by $x_i$ and $y_i$ the two nodes of $\widetilde{\Gamma}_i+\Gamma_i$.

Set $m:=\left \lfloor \frac{g-11}{6} \right \rfloor$ and assume $r(g)=0$. Consider the primitive, ample line bundle $H:=L+\left \lfloor \frac{g-11}{6} \right \rfloor E$, which satisfies $(H)^2=2g-2$. Let $C \in |L|$ be a general smooth curve which meets $\Gamma_1$ transversally. Let $B$ be the union of $C$ with $2m$ copies of $\proj^1$ as in the diagram below.
\\
$$
\setlength{\unitlength}{1cm}
\begin{picture}(8,5)
\put(1,0){\line(0,1){4}}
\put(0.5,1.5){\line(1,0){2}}
\put(2,1){\line(0,1){2}}
\put(1.5,2.5){\line(1,0){2}}
\put(3.7,2.5){$\ldots$}
\put(4.5,2){\line(0,1){2}}
\put(0.5,3.5){\line(1,0){5}}

\put(0,0.4){\line(1,0){7}}

\put(2,0){\mbox{$C$}}
\put(0.45,0.9){\tiny\mbox{$R_{1,1}$}}
\put(0.75,1.3){\tiny\mbox{$p$}}
\put(0.75,3.3){\tiny\mbox{$q$}}
\put(1.1,1.6){\tiny\mbox{$R_{1,2}$}}
\put(1.45,2.2){\tiny\mbox{$R_{2,1}$}}
\put(2.3,2.7){\tiny\mbox{$R_{2,2}$}}
\put(3.9,3.1){\tiny\mbox{$R_{m,1}$}}
\put(4.7,3.7){\tiny\mbox{$R_{m,2}$}}
\end{picture}
$$
\\
Thus $B=C \cup R_{1,1} \cup R_{1,2} \ldots \cup R_{m,1} \cup R_{m,2}$, where each $R_{i,j}$ is smooth and rational, all intersections are transversal and described as follows for $m \geq 2$: $R_{i,j} \cap C= \emptyset$ unless $(i,j)=(1,1)$, $R_{1,1}$ intersects $C$ in one point, and $R_{i,j}$ intersects $R_{k,l}$ in at most one point, with intersections occuring if and only if, after swapping $R_{i,j}$ and $R_{k,l}$ if necessary, we have ($i=k$ and $l=j+1$), ($k=i+1$ and $l \neq j$) or ($(i,j)=(1,1)$ and $(j,k)=(m,2)$). The arithmetic genus of $B$ is $12$. For $m=1$, $B=C \cup R_{1,1} \cup R_{1,2}$ where $C$ intersects $R_{1,1}$ transverally in one point and $R_{1,1} \cap C = \emptyset$ and $R_{1,1}$ intsersects $R_{1,2}$ transversally in two points.

There is then a unique unramified morphism $f: B \to Y_{\Omega_{11}}$ such that $f_{|_C}$ is a closed immersion, $f_{|_{R_{i,1}}}$ is a closed immersion with image $\Gamma_1$ and $f_{|_{R_{i,2}}}$ is a closed immersion with image $\widetilde{\Gamma_1}$, for $1 \leq i \leq m$ and where $(f(p),f(q))=(x_1,y_1)$. Thus $f_*{B} \in |H|$. Note that the stabilization $\hat{B}$ has two components, namely it is the union of $C$ with a rational curve with one node. Thus the claim holds in the $r=0$ case by Proposition \ref{finiteness-criterion}, together with Lemma \ref{stablem} and Corollary \ref{bij-cor}. 
\end{proof}
We now turn to the nonprimitive case.
\begin{lem} \label{easynontriglemma}
Let $\Delta$ be an smooth, one-dimensional algebraic variety over $\C$, with $0 \in \Delta$ a closed point. Let $\mathcal{C} \to \Delta$ be a flat family of nodal curves,
with general fibre integral and such that $$ \mathcal{C}_0 = B \cup \Gamma_1 \cup \ldots \cup \Gamma_k,$$
with $B$ smooth and non-trigonal, and $\Gamma_i$ smooth curves for $1 \ldots i \leq k$, with all intersections transversal. Then if $\tilde{\mathcal{C}}_t$ is the normalization of $\mathcal{C}_t$, $t \in \Delta$ general, then $\tilde{\mathcal{C}}_t$ is non-trigonal.
\end{lem}
\begin{proof}
We go through the first steps of the usual stable reduction procedure, \cite[\S X.4]{arbarello-II}. Let $\mu : \tilde{\mathcal{C}} \to \mathcal{C}$ be the normalization of the integral surface $\mathcal{C}$. Then $\mu$ is finite and birational, and further is an isomorphism outside the preimage of singular points in the fibres of $\mathcal{C}$. Further, $\tilde{\mathcal{C}} \to \Delta$ is a flat-family of curves by \cite[Prop.\ 4.3.9]{liu}. As $\tilde{\mathcal{C}}$ has isolated singularities, $\tilde{\mathcal{C}}_t$ must be smooth for $t \in \Delta$ general, and must be the normalization of $\mathcal{C}_t$, as $\mu$ is finite and birational. Since $\mu_0: \tilde{\mathcal{C}}_0 \to \mathcal{C}_0$ is finite, and an isomorphism outside singular points of $\mathcal{C}_0$, the components of $\tilde{\mathcal{C}}_0$ are isomorphic to $B$, $\Gamma_i$. Further, $\mu_0$ factors through the total normalization of $\mathcal{C}_0$ (which equals the disjoint union of the components). This forces $\tilde{\mathcal{C}}_0$ to itself be nodal. Further $\tilde{\mathcal{C}}_0$ is connected, by taking a desingularization of the surface $\tilde{\mathcal{C}}$ and applying \cite[Thm.\ 8.3.16]{liu} followed by \cite[Cor.\ 8.3.6]{liu}. Since $B$ is smooth and non-trigonal, the stabilization of $\tilde{\mathcal{C}}_0$ lies outside the image of $\overline{\mathcal{H}}_{3,p}$, where $p$ is the arithmetic genus of $\tilde{\mathcal{C}}_0$. It follows that $\tilde{\mathcal{C}}_t$ is non-trigonal.
\end{proof}

\begin{lem} \label{hjz}
Let $Z_a$ be a general K3 surface with Picard lattice $\Lambda_a$ generated by elements $D,F,\Gamma$ giving the intersection matrix
\[ \left( \begin{array}{ccc}
2a-2 & 6 & 1 \\
6 & 0 & 0 \\
1 & 0 & -2 \end{array} \right)\]
Assume that $14 \leq a \leq 19$.
Then we may pick the basis such that $D$, $F$, $\Gamma$ are all effective and represented by integral, smooth curves with $D$ ample. Further, there is an unramified stable map $f_a: B_a \to Z_a$, birational onto its image, with $B_a$ an integral, nodal curve of arithmetic genus $13$ for $14 \leq a \leq 15$ and $15$ for $16 \leq a \leq 19$, such that $f_{a*}(B_a) \in |D|$ and $f_a$ satisfies the conditions of Proposition \ref{finiteness-criterion}. Further, there is an integral, rational nodal curve $F_0 \in |F|$ which meets $f_a(B_a)$ transversally, and $\Gamma$ meets $f_a(B_a)$ transversally in one point.
\end{lem}
\begin{proof}
 Let $M_{\Lambda_a}$ be the moduli space of pseudo-ample, $\Lambda_a$-polarised K3 surfaces, \cite{dolgachev}. This has at most two components, which locally on the period domain are interchanged via complex conjugation. Consider the lattice $\Omega_{11}$ with ordered basis $\{L, E, \Gamma_1, \ldots, \Gamma_8 \}$ and set $d_8=1$. For $d_1, d_2 \geq 3$, let $H$ be the primitive, ample line bundle 
 $H=L+\Gamma_1 + \epsilon \Gamma_2$, where $\epsilon=0$ for $14 \leq a \leq 15$ and $\epsilon=1$ for $16 \leq a \leq 19$. Choose $3\leq d_1, d_2 \leq 5$ such that $H^2=2a-2$; it is easily checked that all six possibilities can be achieved. There is a primitive lattice embedding $$\Lambda_a \hookrightarrow \Omega_{11}$$ defined by $D \mapsto H$, $\Gamma \mapsto \Gamma_8$, $F \mapsto E$. 
 
 Let $Y_{\Omega_{11}}$ be any K3 surface with $\text{Pic}(Y_{\Omega_{11}}) \simeq \Omega_{11}$, and choose the basis $\{L, E, \Gamma_1, \ldots, \Gamma_8 \}$ as in Lemma \ref{lem-aaa}. Consider the curve $C \cup \Gamma_1 \cup \epsilon \Gamma_2 \in |H|$, where $C \in |L|$ is a general smooth curve. By partially normalizing at all nodes other than three on $C \cup \Gamma_1$ and three on $C \cup \Gamma_2$ (for $\epsilon \neq 0$), we construct an unramified stable map $\bar{f}_a: \bar{B}_a \to Y_{\Omega_{11}}$ with $\bar{f}_{a*} (\bar{B}_a) \sim H$ which is birational onto its image and satisfying the conditions of Proposition \ref{finiteness-criterion}. Note that $\bar{B}_a$ has arithmetic genus $13$ for $14 \leq a \leq 15$ and $15$ for $16 \leq a \leq 19$. After deforming $\bar{f_a}$ to an unramified stable map $f_a: B_a \to Z_a$, we find $B_a$ must become integral, since it is easily checked that $Z_a$ contains no smooth rational curves $R$ with $(R \cdot F)=(R \cdot \Gamma)=0$. Further, the normalization $\tilde{B}_a$ is non-trigonal by Lemma \ref{easynontriglemma}.
 
 Thus the claim on $f_a$ follows from the proof of Proposition \ref{finiteness-criterion}. 
Note that the $I_2$ fibre $\Gamma_7 + \widetilde{\Gamma}_7$ must deform to an integral, nodal, rational curve on $Z_a$, since $Z_a$ contains no smooth rational curves which avoid $F$ and $\Gamma$. If $\Omega_{11} \hookrightarrow \text{Pic}(Y_{\Omega_{11}})$ is the embedding as in Lemma \ref{lem-aaa}, and if $Y^c_{\Omega_{11}}$ is the conjugate K3 surface, then we obviously have an embedding $\Omega_{11} \hookrightarrow \text{Pic}(Y^c_{\Omega_{11}})$ satisfying the conclusions of Lemma \ref{lem-aaa}. Thus the claim holds for the general pseudo-ample, $\Lambda_a$-polarised K3 surface $Z_a$.
\end{proof}
\begin{remark} \label{slightgen-nodal2}
In the notation of the above proof, we have $h^0(N_{\bar{f}_a}) \leq p(\bar{B}_a)$ from Lemma \ref{stablem}. It thus follows that 
$h^0(N_{f_a}) \leq p(B_a)$ by semicontinuity.
\end{remark}

\begin{lem} \label{aghzt}
Let $1 \leq d \leq 5$ be an integer and consider the rank five lattice $K_d$ with basis $\{ A,B,\Gamma_1,\Gamma_2, \Gamma_3 \}$ giving the intersection matrix
\[ \left( \begin{array}{ccccc}
-2 & 6 & 3 & 2 & d \\
6 & 0 & 0 & 0 & 0 \\
3 & 0 & -2 & 0 & 0\\
2 & 0 & 0 & -2 & 0 \\
d & 0 & 0 & 0 & -2
 \end{array} \right). \]
 Then $K_d$ is an even lattice of signature $(1,4)$. There exists a K3 surface $Y_{K_d}$ with
 $\text{Pic}(Y_{K_d}) \simeq K_d$, and such that the classes $\{ A,B,\Gamma_1,\Gamma_2, \Gamma_3 \}$ are all represented by nodal, reduced curves such that the nodal curve $A$ meets $\Gamma_1, \Gamma_2$ and $\Gamma_3$ transversally. Further, $A+B$ is big and nef.
\end{lem}
\begin{proof}
Let $Y_1, Y_2$ be smooth elliptic curves and consider the Kummer surface $\widetilde{Z}$ associated to $Y_1 \times Y_2$. Let $P_1, P_2, P_3, P_4$ be the four $2$-torsion points of $Y_1$ and let $Q_1, Q_2, Q_3, Q_4$ be the $2$-torsion points of $Y_2$. Let $E_{i,j} \seq \widetilde{Z}$ denote the exceptional divisor over $P_i \times Q_j$, let $T_i \seq \widetilde{Z}$ denote the strict transform of $(P_i \times Y_2) / \pm$ and let
$S_j$ denote the strict transform of $(Y_1 \times Q_j) / \pm$. We also denote by $F$ a smooth elliptic curve of the form $x \times Y_2$, where $x \in Y_1$ is a non-torsion point. It may help the reader to consult the diagram on \cite[p.\ 344]{mori-mukai}, to see the configuration of these curves. We set 
\begin{align*}
\widetilde{A} &:= S_1+E_{1,1}+T_1+E_{1,2}+S_2+E_{1,3}+S_3 \\ 
\widetilde{B} &:=F+S_4+E_{2,4}+T_2+E_{2,1}+E_{2,2}+E_{2,3} \\ 
\widetilde{\Gamma}_1 &:=T_3+E_{3,1}+E_{3,2}+E_{3,3}\\
\widetilde{\Gamma}_2&:=T_2+E_{2,1}+E_{2,2} \\ 
\widetilde{\Gamma}_3 &= \left\{ \begin{array}{ll}
         T_4+\sum_{i=1}^d E_{4,i}, & \mbox{if $1 \leq d \leq 3$}\\
         T_4+\sum_{i=1}^{d-3} E_{4,i}+E_{4,4}+S_4+E_{2,4}+T_2+E_{2,1}+E_{2,2}+E_{2,3}, & \mbox{if $4 \leq d \leq 5$}.\end{array} \right.
\end{align*}
Then $\widetilde{A}$, $\widetilde{B}$, $\widetilde{\Gamma}_1, \widetilde{\Gamma}_2, \widetilde{\Gamma}_3$ generate $K_d $ (to simplify the computations, use that a tree of $-2$ curves has self-intersection $-2$). To see that this gives a primitive embedding of $K_d $ in $\text{Pic}(\widetilde{Z})$ we compute the intersections with elements of $\text{Pic}(\widetilde{Z})$; for $J \in \text{Pic}(\widetilde{Z})$, define $(J \cdot K_d)$ to be the quintuple
$((J \cdot \widetilde{A}), (J \cdot \widetilde{B}), (J \cdot \widetilde{\Gamma}_1), (J \cdot \widetilde{\Gamma}_2),(J \cdot \widetilde{\Gamma}_3))$. Then one computes
$(T_1 \cdot K_d)=(1,0,0,0,0)$, $(T_3 \cdot K_d)=(0,0,1,0,0)$, $(E_{4,1} \cdot K_d)=(1,0,0,0,-1)$,
$(E_{2,4} \cdot K_d)=(0,0,0,1,0)$, $(E_{2,3} \cdot K_d)=(1,-1,0,1, c)$, where $c$ is either $0$ or $-1$, depending on $d$. Thus $K_d$ is primitively embedded in $\text{Pic}(\widetilde{Z})$. Further, all intersections of  $\widetilde{B}$, $\widetilde{\Gamma}_1, \widetilde{\Gamma}_2, \widetilde{\Gamma}_3$ with $\widetilde{A}$ are transversal . Note that for any (rational) component $R \seq \widetilde{A}+\widetilde{B}$, $(R \cdot \widetilde{A}+\widetilde{B}) \geq 0$. Thus $\widetilde{A}+\widetilde{B}$ is big and nef. Hence the claim holds by degenerating to $\widetilde{Z}$.
\end{proof}

\begin{lem} \label{aghzt-2}
Let $1 \leq d \leq 5$ be an integer and consider the K3 surface $Y_{K_d}$ from Lemma \ref{aghzt}. Then the classes $\{ A,B,\Gamma_1,\Gamma_2, \Gamma_3 \}$ are all represented by integral curves. 
\end{lem}
\begin{proof}
We will firstly show that $B$ is nef, and hence base point free. Indeed, there would otherwise exist
an effective divisor $R=xA+yB+z\Gamma_1+w\Gamma_2+u\Gamma_3$ for integers $x,y,z,w,u$, with $(R)^2=-2$ and $(R \cdot B) < 0$, i.e.\ $x<0$. Thus $R-xA$ is effective and $(R-xA \cdot A)=(R-xA \cdot A+B) \geq 0$,
since $A+B$ is nef. But then
\begin{align*}
-2=(R)^2 &= ((R-xA)+xA)^2 \\
&=-2(z^2+w^2+u^2+x^2)+2x(R-xA \cdot A)
\end{align*}
So we must have $z=w=u=0$, $x=-1$ and $(R+A \cdot A)=0$. But then $R=-A+yB$ and $(R+A \cdot A)=0$ gives $y=0$. But this contradicts that $A$ is effective.

We next show that each $\Gamma_i$ is integral. Let $R$ be any irreducible component of $\Gamma_i$ with $(R \cdot \Gamma_i) <0$, $(R)^2=-2$ (such a component exists). Write $R=xA+yB+z\Gamma_1+w\Gamma_2+u\Gamma_3$. We have $x \geq 0$ as $(R \cdot B) \geq 0$. Assume $x \neq 0$. Then
$(R \cdot R+B)>0$, $(R +B)^2 >0$, so that $R+B$ is big and nef, contradicting that $(R+B \cdot \Gamma_i)=(R \cdot \Gamma_i) <0$. Thus $x=0$ and $R=yB+z\Gamma_1+w\Gamma_2+u \Gamma_3$. Since $(R)^2=-2$, $(R \cdot \Gamma_i) <0$, we have $R=yB+\Gamma_i$. Since $(R \cdot A+B)=(R \cdot A) \geq 0$, we must have $y \geq 0$ (for $i=3$, we need here that $d <6$). Since the only effective divisor in $|R|$ is integral (and equal to $R$), we must have $y=0$. Thus $\Gamma_i=R$ is integral.

To show that $A$ is integral, let $R_1, \ldots, R_s$ be the components of the effective $-2$ curve $A$, and write $R_i=x_iA+y_iB+z_i\Gamma_1+w_i\Gamma_2+u_i\Gamma_3$ for integers $x_i,y_i,z_i,w_i,u_i$. Intersecting with $B$ shows we have $x_i \geq 0$ for all $i$. Thus there is precisely one component, say $R_1$ with $x_1 \neq 0$ and further $x_1=1$. Now, choose a component $R_i$ with $(R_i)^2=-2$, $(R_i \cdot A)<0$. Firstly assume $i \neq 1$, so that $x_i=0$. Intersecting with the integral curves $\Gamma_j$, $1 \leq j \leq 3$ (and noting $R_i \neq \Gamma_j$ as $(R_i \cdot A) <0$), we have $z_i, w_i, u_i \leq 0$. From $(R_i)^2=-2$, we see $R_i=y_iB-\Gamma_j$ for some $1 \leq j \leq 3$. Intersecting with $A$ gives $6y_i-k <0$ for some $1 \leq k \leq 5$, and thus $y_i \leq 0$ which is a contradiction to the effectivity of $R_i$.

In the second case, assume $(R_1)^2=-2$, $(R_1 \cdot A)<0$, with $R_1=A+y_1B+z_1\Gamma_1+w_1\Gamma_2+u_1\Gamma_3$. We compute
\begin{align*}
-2 &= (R_1)^2\\
&= ((R_1-A)+A)^2 \\
&=-2(y_1^2+z_1^2+w_1^2+u_1^2+1)+2((R_1 \cdot A)+2) \\
&= -2(y_1^2+z_1^2+w_1^2+u_1^2)+2((R_1 \cdot A)+1).
\end{align*} 
Thus we have either $(R_1 \cdot A)=-2$ and $R_1=A+y_1B$ or $(R \cdot A)=-1$, $R_1=A+y_1B \pm \Gamma_j$ for some $1\leq j \leq 3$. In the first case, $(A+y_1B \cdot A)=-2$ implies $y_1=0$ so $A=R_1$ is integral. In the second case, $(A+y_1B \pm \Gamma_j \cdot A)=-1$ implies $-1=-2+6y_1 \pm k$, for $1 \leq k=(A \cdot \Gamma_j) \leq 5$. The only possibilities are $y_1=0$, $k=1$, $R_1=A+\Gamma_j$, contradicting that all effective divisors in $|R_1|$ are integral, or $y_1=1$, $k=5$, $R_1=A+B-\Gamma_j$. Since $(B-\Gamma_i)^2=-2$, $(B-\Gamma_j \cdot A+B)>0$, we have that $B-\Gamma_j$ is effective, so once again this contradicts that all effective divisors in $|R_1|$ are integral.
\end{proof}

\begin{lem} \label{essential-def}
Let $M_{\Lambda_a}$ denote the moduli space of pseudo-ample $\Lambda_a$-polarised K3 surfaces, with the lattice $\Lambda_a$ as in Lemma \ref{hjz}, for $14 \leq a \leq 19$. Then there is a nonempty open subset $U \seq M_{\Lambda_a}$ such that for $[Y_a] \in U$  we may pick the basis $\{ D, F, \Gamma \}$ such that 
there is an integral, nodal, rational curve $R_a \in |D-2F-\Gamma|$ such that $R_a$ meets $\Gamma$ transversally in three points.
\end{lem}
\begin{proof}
Set $y:=a-14$, so that $0 \leq y \leq 5$ by assumption.
If we change the basis of $\Lambda_a$ to $\{ D-2F-\Gamma, F, \Gamma \}=\{X,Y,Z \}$, we see $\Lambda_a$ is isometric to the lattice 
$\bar{\Lambda}_a$ with intersection matrix
\[
\left( \begin{array}{ccc}
2y-2 & 6 & 3 \\
6 & 0 & 0 \\
3 & 0 & -2 \end{array} \right).\] Let $K_d$ be the lattice from Lemma \ref{aghzt}.
For appropriate choices of $d$ there is a primitive lattice embedding
$\bar{\Lambda}_a \hookrightarrow K_d$, given by
\begin{align*}
X &\mapsto  A+\epsilon_1\Gamma_3+\epsilon_2 \Gamma_2 \\
Y &\mapsto B \\
Z &\mapsto \Gamma_1
\end{align*}
where if $a =14$ we set $\epsilon_1=\epsilon_2=0$ and $d$ arbitrary, if $15 \leq a < 19$ we set $d=y+1$, $\epsilon_1=1$, $\epsilon_2=0$ and if $a=19$ we set $d=5$ and $\epsilon_1=\epsilon_2=1$. By Lemmas \ref{aghzt}, \ref{aghzt-2}, we have that $A,\Gamma_2, \Gamma_3 $ are represented by smooth rational curves intersecting transversally on $Y_{K_d}$. In all cases, the divisor $D=X+2Y+Z$ is big and nef, since if $R \in \{ \Gamma_1, \epsilon_1 \Gamma_3, \epsilon_2 \Gamma_2, A \}$, $(R \cdot D) \geq 0$. By partially normalising nodes we may produce an unramified, stable map $f: \tilde{C} \to Y_{K_d}$, where $\tilde{C}$ is a genus zero union of smooth rational curves and $f(\tilde{C})=A+\epsilon_1 \Gamma_3+\epsilon_2 \Gamma_2$. By \cite[\S 2, Rem.\ 3.1]{kool-thomas}, we may deform $f$ horizontally to a stable map with target a small deformation of $Y_{K_d}$ in the moduli space of $\Lambda_a$-polarised K3 surfaces (also cf.\ Proposition \ref{relative-dominant}).
Thus, after deforming $Y_{K_d}$ to a K3 surface $Y_a$ with $\text{Pic}(Y_a) \simeq \Lambda_a$, we can deform $A+\epsilon_1 \Gamma_3+\epsilon_2 \Gamma_2$ to a nodal rational curve $\bar{R}_a$ which meets $\Gamma$ transversally in three points. Furthermore, $\bar{R}_a$ is integral, since $\Lambda_a$ contains no $-2$ curves which have zero intersection with $F, \Gamma$. 
\end{proof}

\begin{thm} \label{finiteness-nonprim}
Assume $k \geq 2$, $g \geq 8$. Set $m:= \left \lfloor \frac{g-5}{6} \right \rfloor$ and let $0 \leq r(g) \leq 5$ be the unique integer such that
$$ g-5=6m+r(g).$$
Define:
\begin{itemize}  
\item $l_g:=15$, if $r(g)=3,4$, $m$ odd and/or $k$ even.
\item $l_g:=16$, if $r(g)=3,4$, $m$ even and $k$ odd.
\item $l_g:=17$, if $r(g)=5$, $m$ odd and/or $k$ even.
\item $l_g:=18$, if $r(g)=5$, $m$ even and $k$ odd.
\item $l_g:=17$, if $r(g)\leq 2$, $m$ even and/or $k$ even.
\item $l_g:=18$, if $r(g) \leq 2$, $m$ odd and $k$ odd.
\end{itemize}
Then there is a component $I \seq \mathcal{T}^n_{g,k}$ such that $${\eta}_{|_I}: I \to \mathcal{M}_{p(g,k)-n}$$ is generically finite for
$p(g,k)-n \geq l_g$. For the general $[f': B'\to X'] \in I$, $B'$ is non-trigonal.
\end{thm}
\begin{proof}
Consider the $\Lambda_a$-polarised K3 surface $Y_a$ from Lemma \ref{essential-def} and let $\{ D, F, \Gamma \}$ be as in the lemma. Set $m:= \left \lfloor \frac{g-5}{6} \right \rfloor \geq 0$ and $$m'= \begin{cases} m-1 &\mbox{if } r(g) \geq 3 \\ 
m-2 & \mbox{if } r(g) \leq 2. \end{cases}. 
$$ We have $m' \geq -1$ for $g \geq 8$. Consider the primitive, ample line bundle $H=D+m'F $. We choose
$$a= \begin{cases} 11+r(g) &\mbox{if } r(g) \geq 3 \\ 
17+r(g) & \mbox{if } r(g) \leq 2. \end{cases}. 
$$
Then $(H)^2=2g-2$ for $g \geq 8$. Let $f_a: B_a \to Y_a$ resp.\ $R_a$ be the unramified stable map, resp.\ rational curve from  lemmas \ref{hjz} resp.\ \ref{essential-def}. Set $l=k m'+2(k-1)$ which is nonnegative for $k \geq 2$. We have an effective decomposition
$$kH \sim f_a(B_a)+(k-1) R_a+(k-1)\Gamma +lF_0 ,$$
where $F_0 \in |F|$ is an integral, nodal rational curve as in lemma \ref{hjz}.  We will prove the result by constructing an unramified stable map $f: B \to Y_a$ with $f_* (B)= f_a(B_a)+(k-1)R_a+(k-1)\Gamma +lF_0$ satisfying the conditions of Proposition \ref{finiteness-criterion}.

Assume firstly $m'$ is even and/or $k$ is even, so $l$ is even, and set $s=l/ 2$. Let $\proj^1 \to F_0$ be the normalization morphism and let $p,q$ be the points over the node. Let $x$ be the point of intersection of $f_a(B_a)$ and $\Gamma$ and let $y,z$ be distinct points in $\Gamma \cap R_a$. We may pick the points to ensure $y \neq x, z \neq x$. Define $B$ as the union of $B_a$ with $l+2(k-1)$ copies of $\proj^1$ and with transversal intersections as in the following diagram:
$$
\setlength{\unitlength}{1cm}
\begin{picture}(14,5)
\put(1,0){\line(0,1){4}}
\put(0.5,1.5){\line(1,0){2}}
\put(2,1){\line(0,1){2}}
\put(1.5,2.5){\line(1,0){2}}
\put(3.7,2.5){$\ldots$}
\put(4.5,2){\line(0,1){2}}
\put(0.5,3.5){\line(1,0){5}}

\put(8,0){\line(0,1){4}}
\qbezier(8.5,2)(4,1)(10.5,1.5)
\put(10,1){\line(0,1){2}}
\put(9.5,2.5){\line(1,0){2}}
\put(11.7,2.5){$\ldots$}
\put(12.5,2){\line(0,1){2}}
\put(12.2,3.5){\line(1,0){2}}

\put(0,0.4){\line(1,0){14}}

\put(5,-0.1){\mbox{$B_a$}}
\put(0.45,0.9){\tiny\mbox{$F_{1,1}$}}
\put(0.75,1.3){\tiny\mbox{$p$}}
\put(0.75,3.3){\tiny\mbox{$q$}}
\put(1.1,1.6){\tiny\mbox{$F_{1,2}$}}
\put(1.75,1.6){\tiny\mbox{$q$}}
\put(1.45,2.2){\tiny\mbox{$F_{2,1}$}}
\put(2.3,2.7){\tiny\mbox{$F_{2,2}$}}
\put(3.9,3){\tiny\mbox{$F_{s,1}$}}
\put(4.3,3.3){\tiny\mbox{$p$}}
\put(4.7,3.7){\tiny\mbox{$F_{s,2}$}}

\put(7.8,0.2){\tiny\mbox{$x$}}
\put(7.6,0.7){\tiny\mbox{$\Gamma_1$}}
\put(8.6,1.2){\tiny\mbox{$R_{a,1}$}}
\put(9.6,2.3){\tiny\mbox{$\Gamma_2$}}
\put(10.6,2.3){\tiny\mbox{$R_{a,2}$}}
\put(11.8,3.1){\tiny\mbox{$\Gamma_{k-1}$}}
\put(13.2,3.6){\tiny\mbox{$R_{a,k-1}$}}
\put(7.8,1.22){\tiny\mbox{$y$}}
\put(7.8,1.7){\tiny\mbox{$z$}}
\end{picture}
$$
Then there is a unique unramified morphism $f: B \to Y_a$ with $f_* (B)= f_a(B_a)+(k-1)R_a+(k-1)\Gamma +lF_0$ which restricts to the normalization $\proj^1 \to R_a$ on all components marked $R_{a,i}$, restricts to $f_a$ on $B_a$, sends all components marked $F_{i,j}$ to $F_0$, all components marked $\Gamma_i$ to $\Gamma$ and which takes points marked $x$ (resp.\ $y,z,p,q$) to $x$ (resp.\ $y,z,p,q$). 

We now claim that if $B_0 \seq B$ is a connected union of components containing $R_{a,k-1}$ with $f_*(B_0) \in |nH|$ then $n=k$ and $B_0 = B$. If $c_1D+c_2F+c_3\Gamma$ is a divisor linearly equivalent to $nH$, then intersecting with $F$ shows $c_1=n$. Now $R_{a} \in |D-2F-\Gamma|$, whereas $H =D+(m-1)F$ for $m \geq 0$. Thus $f_*(B_0) \in |nH|$ shows that the connected curve $B_0$ cannot coincide with the component $R_{a,k-1}$, and thus must also contain $\Gamma_{k-1}$. Repeating this argument, one sees readily that $B_0$ must contain $\sum_{i=1}^{k-1} (\Gamma_i+R_{a,i})+B_a$. We then get $n=k$ as required, which forces $B_0=B$. 

Using Remarks \ref{slightgen-nodal} and \ref{slightgen-nodal2}, one sees $h^0(N_{f}) \leq p(B)$. For any component $C \neq B_a \seq B$, $f(C)$ meets $f(B_a)$ properly. Thus it follows from Proposition \ref{prim-cor} that the conditions of Proposition \ref{finiteness-criterion} are met. Note that the arithmetic genus of $B$ is $l_g$. 

Now assume $m'$ is odd and $k$ is odd. Let $a,b,c \in f_a(B_a) \cap F_0$ be distinct points. Let $B$ be as in the diagram below.
$$
\setlength{\unitlength}{0.9cm}
\begin{picture}(14,5)
\qbezier(0,2)(2.3,-4)(1.5,4)
\qbezier(-1.5,-0.5)(-0.8,4)(0,2)
\put(1.5,1.5){\line(1,0){2}}
\put(3,1){\line(0,1){2}}
\put(2.5,2.5){\line(1,0){2}}
\put(4.7,2.5){$\ldots$}
\put(5.5,2){\line(0,1){2}}

\put(8,0){\line(0,1){4}}
\qbezier(8.5,2)(4,1)(10.5,1.5)
\put(10,1){\line(0,1){2}}
\put(9.5,2.5){\line(1,0){2}}
\put(11.7,2.5){$\ldots$}
\put(12.5,2){\line(0,1){2}}
\put(12.2,3.5){\line(1,0){2}}

\put(-2,0.4){\line(1,0){17}}

\put(6,-0.1){\mbox{$B_a$}}
\put(0.45,0.9){\tiny\mbox{$F_{1}$}}
\put(1.75,1.3){\tiny\mbox{$p$}}
\put(2.1,1.6){\tiny\mbox{$F_{2}$}}
\put(2.75,1.6){\tiny\mbox{$q$}}
\put(2.45,2.2){\tiny\mbox{$F_{3}$}}
\put(3.3,2.7){\tiny\mbox{$F_{4}$}}
\put(4.9,3){\tiny\mbox{$F_{l}$}}

\put(-1.2,0.2){\tiny\mbox{$a$}}
\put(0.4,0.2){\tiny\mbox{$b$}}
\put(1.45,0.2){\tiny\mbox{$c$}}

\put(7.8,0.2){\tiny\mbox{$x$}}
\put(7.6,0.7){\tiny\mbox{$\Gamma_1$}}
\put(8.6,1.2){\tiny\mbox{$R_{a,1}$}}
\put(9.6,2.3){\tiny\mbox{$\Gamma_2$}}
\put(10.6,2.3){\tiny\mbox{$R_{a,2}$}}
\put(11.8,3.1){\tiny\mbox{$\Gamma_{k-1}$}}
\put(13.2,3.6){\tiny\mbox{$R_{a,k-1}$}}
\put(7.8,1.22){\tiny\mbox{$y$}}
\put(7.8,1.7){\tiny\mbox{$z$}}
\end{picture}
$$
\\
Then as before there is an unramified morphism $f: B \to Y_a$ with  $f_* (B)= f_a(B_a)+(k-1)R_a+(k-1)\Gamma +lF_0$ satisfying the conditions of Proposition \ref{finiteness-criterion}, and $B$ has arithmetic genus $l_g$. 
\end{proof}

The following lemma will be needed for Theorem \ref{marked-wahl-k3}.
\begin{lem} \label{coh-van}
Assume $p(g,k)$ and $n$ are such that there is a component $I \seq \mathcal{T}^n_{g,k}$ such that the morphism
${\eta}_{|_I}: I \to \mathcal{M}_{p(g,k)-n}$ is generically finite. Then for the general $[(f: B \to X,L)] \in I$, we have
$$H^0(B, f^*(T_X))=0.$$
\end{lem}
\begin{proof}
Let $T^n_{g,k}(X,L)$ denote the fibre of $\mathcal{T}^n_{g,k} \to \mathcal{B}_g$ over $[(X,L)]$. 
The finiteness of ${\eta}_{|_I}$ at $[(f: B \to X,L)] $ obviously implies that the morphism
\begin{align*}
r_{n,k} \;: \;T^n_{g,k}(X,L) & \to \mathcal{M}_{p(g,k)-n} \\
[f: B \to X] & \mapsto [B]
\end{align*}
is finite near $[f]$. The claim $H^0(B, f^*(T_X))=0$ then follows from the exact sequence of sheaves on $B$
$$ 0 \to T_{B} \to f^*(T_X) \to N_f \to 0$$
and the fact that the coboundary morphism $H^0(B, N_f) \to H^1(B, T_{B})$ corresponds to the differential of $r_{n,k}$.
\end{proof}

\chapter{Chow groups and Nikulin involutions} \label{nik-inv}
Let $X$ be a projective K3 surface and let $f: X \to X$ be a Nikulin involution; i.e.\ an automorphism of order two with $f^* \omega= \omega$ for all $\omega \in H^{2,0}(X)$ (or, equivalently, $f^*$ acts trivially on the transcendental lattice $T(X)$). Then the Bloch--Beilinson conjectures imply in particular the following conjecture
\begin{conjecture} \label{bloch}
$f^*$ acts as the identity on the Chow group $CH^2(X)$ of zero-cycles on $X$.
\end{conjecture}
In this chapter we will show that if the genus of $X$ is odd then this conjecture holds on one component of the moduli space of K3 surfaces with a Nikulin involution. 

To explain this, we first need to outline the construction of the moduli spaces of K3 surfaces admitting a Nikulin involution, taken from \cite{sar-gee}. This is explained in detail in Section \ref{ModNik}. Of particular importance will be special elliptic K3 surfaces with a Nikulin involution which lie in the boundary on one component of the moduli space. These elliptic K3 surfaces have already appeared as Example \ref{X_{(2),8A_1}} in Section \ref{WeierEll}, where they were constructed via the Weierstrass equation and denoted $X_{(2),8A_1}$.
In Section \ref{BloBei} we use stable maps to deform to $X_{(2),8A_1}$ and thereby conclude that Conjecture \ref{bloch} holds on one component of the moduli space of K3 surfaces admitting a Nikulin involution. 

\section{Moduli spaces of K3 surfaces admitting a Nikulin involution} \label{ModNik}
We begin with some lattice theory from \cite{sar-gee}. Consider the lattice $$\Phi_g := \mathbb{Z}L \oplus E_8(-2),$$
where $L^2=2g-2$, $g \geq 3$. If $g$ is odd, then there is a unique lattice $ \Upsilon_g$ which is an over lattice of $\Phi_g$ with $ \Upsilon_g / \Phi_g \simeq \mathbb{Z} / 2\mathbb{Z}$ and such that $E_8(-2)$ is a primitive sublattice of $\Upsilon_g$ and $L$ is primitve in $\Upsilon_g$. Equivalently, $\Upsilon_g$ is the lattice generated by $\Phi_g$ and the element $L/2 + v/2$, where $v \in E_8(-2)$ is a nonzero element with
$$ v^2 \equiv 4\epsilon \; (8),$$
with $\epsilon:=0$ if $g \equiv 1 \; (4)$ and $\epsilon:=1$ of $g \equiv 3 \; (4)$.

There is an involution $j : \Phi_g \to \Phi_g$ which acts at $1$ on $\mathbb{Z}L$ and $-1$ on $E_8(-2)$. This extends to an involution $j: \Upsilon_g \to \Upsilon_g$. The following is due to Nikulin \cite{nikulin-finite} and van Geemen--Sarti \cite[\S 2]{sar-gee}.
\begin{thm} \label{sarti-geemans-thm}
Let $\Lambda$ be either of the lattices $\Phi_g$ or $\Upsilon_g$. Let $X$ be a K3 surface with a primitive embedding $\Lambda \hookrightarrow \text{Pic}(X)$ such that $L$ is big and nef. Then $X$ admits a Nikulin involution $f: X \to X$ such that $f^*_{|_\Lambda}=j$ and $f^*_{|_{\Lambda^{\perp}}}=\text{id}$, for $\Lambda^{\perp} \seq H^2(X,\mathbb{Z})$. Conversely, if $X$ admits a Nikulin involution $f$ then there is a primitive embedding $\Lambda \hookrightarrow \text{Pic}(X)$ such that $L$ is big and nef, where $\Lambda$ is either $\Phi_g$ or $\Upsilon_g$, for some $g$. Further, $f^*_{|_\Lambda}=j$ and $f^*_{|_{\Lambda^{\perp}}}=\text{id}$.
\end{thm}
Notice that the conditions $f^*_{|_\Lambda}=j$ and $f^*_{|_{\Lambda^{\perp}}}=\text{id}$ uniquely determine the action of $f^*$ on $H^2(X,\mathbb{Z})$, and thus they uniquely determine the involution $f$. If $X$ admits a Nikulin involution $f$, then there may be several different lattices $\Lambda \hookrightarrow \text{Pic}(X)$ of type $\Phi_g$ or $\Upsilon_g$ with $f^*_{|_\Lambda}=j$ and $f^*_{|_{\Lambda^{\perp}}}=\text{id}$ (for instance, we will see later that $X_{(2),8A_1}$ from Example \ref{X_{(2),8A_1}} has this property). Different choices of primitive embeddings $\Lambda \hookrightarrow \text{Pic}(X)$ correspond to choosing different $f$-invariant polarisations $L$; in particular there is no choice if the Picard number of $X$ is the minimal value nine.

To study deformations of K3 surfaces admitting a Nikulin involution, it suffices to consider the moduli space $M_{\Lambda}$ of $\Lambda$-polarised K3 surfaces, where $\Lambda$ is either of the lattices $\Phi_g$ or $\Upsilon_g$, as in \cite{dolgachev}. As $\Lambda$ has rank nine, $M_{\Lambda}$  is nonempty and has at most two irreducible components. In fact we have:
\begin{prop}
Let $\Lambda$ be either of the lattices $\Phi_g$ or $\Upsilon_g$ and let $M_{\Lambda}$ be the moduli space of $\Lambda$-polarised K3 surfaces. Then $M_{\Lambda}$ is irreducible.
\end{prop}
\begin{proof}
From \cite[Cor.\ 5.2]{dolgachev} there is a primitive embedding $\Lambda \hookrightarrow L_{K3}$, unique up to isometry of $L_{K3}$, where $L_{K3}$ is the K3 lattice. Consider the orthogonal $\Lambda^{\perp} \seq L_{K3}$. From \cite[\S 5]{dolgachev}, it suffices to show that $\Lambda^{\perp}$ admits the hyperbolic lattice $U$ as an orthogonal summand.

We now mimic the proof of \cite[Prop.\ 2.7]{sar-gee}. We have $L_{K3} \simeq U^3 \oplus E_8(-1)^2$. There is a primitive embedding $i: E_8(-2) \hookrightarrow E_8(-1)^2$ such that $i(E_8(-2))$ is the set $\{ (x,-x) \; : \; x \in E_8(-1)  \}$. We may choose a basis $R, F$ of $U$ with $(R)^2=-2$, $(F)^2=0$, $(R \cdot F)=1$. Let $U_i$ denote the $i$-th copy of $U$ in $U^3 \oplus E_8(-1)^2$, $i \leq 3$. We have a primitive embedding $ \mathbb{Z} L \hookrightarrow U_1$ given by $L \mapsto R+gF$. Thus we have a primitive embedding $\Phi_g \hookrightarrow L_{K3}$, unique up to isometry, such that $U_2 \oplus U_3$ is an orthogonal summand of $\Phi_g^{\perp}$. In particular, $M_{\Lambda}$ is irreducible when $\Lambda=\Phi_g$.

It remains to consider the case $\Lambda=\Upsilon_g$. Write $g=2n+1$ for $n \in \mathbb{Z}$. Let $\alpha \in E_8(-1)$ with $\alpha^2=-2$ if
$n$ is odd and $\alpha^2=-4$ is $n$ is even. As before, identify $E_8(-2)$ with $\{ (x,-x) \; : \; x \in E_8(-1)  \}$ and let $v=(\alpha, -\alpha) \in E_8(-2)$.
Let $L=(2u, \alpha, \alpha) \in U_1 \oplus E_8(-1)^2$ where $u=R+\frac{n+3}{2}F$ for $n$ odd and $u=R+\frac{n+4}{2}F$ for $n$ even. Then
$\frac{1}{2}(L+v)=(u, \alpha, 0)$, and the lattice generated by $\frac{1}{2}(L+v)$, $E_8(-2)$ is a primitive sublattice of $U^3 \oplus E_8(-1)^2$ isomorphic to $\Upsilon_g$. Once again $U_2 \oplus U_3$ is an orthogonal summand of $\Upsilon_g^{\perp}$, so in particular $U$ is an orthogonal summand of $\Upsilon_g^{\perp}$, and $M_{\Upsilon_g}$ is irreducible.
\end{proof}

Let $X$ be a K3 surface with a Nikulin involution $f: X \to X$. Then $X$ has eight (reduced) fix points, see \cite[Ch. 15]{huy-lec-k3}. Let $G \simeq \mathbb{Z} / 2\mathbb{Z}$ be the group of automorphisms generated by $f$. The quotient
$X / G$ has eight rational double point singularities over the fixed points. Let $Y \to X / G$ be the minimal desingularisation; surfaces arising in this way are called \emph{Nukulin surfaces}. Then $Y$ is a projective K3 surface, and it contains eight disjoint rational curves over the nodes of $X / G$. Thus $Y$ has Picard number at least nine. We define the \emph{Nikulin lattice} to be the rank eight lattice $\mathfrak{N}$ generated by $e_1, \ldots, e_8$ together with $\frac{1}{2}(e_1 +\ldots +e_8)$, with $(e_i)^2=-2$ for all $i$ and $(e_i \cdot e_j)=0$ if $i \neq j$, $1 \leq i,j \leq 8$.

We now follow the lattice theory from \cite{garbagnati-sarti-evenset} (see also \cite[Prop.\ 2.7]{sar-gee}). For arbitrary $g' \geq 3$, set $$\Phi'_{g'}:=\mathbb{Z}M \oplus \mathfrak{N},$$
where $(M)^2=2g'-2$. In the case where $g'$ is odd, let $\Upsilon'_{g'}$ be the rank nine lattice generated by $\Phi'_{g'}$ together with the class $\frac{1}{2}(M+v)$, where $v \in \mathfrak{N}$, $\frac{v}{2} \notin \mathfrak{N}$ is a class with
\begin{itemize}
\item $(v)^2 \equiv 0 \; (4) $
\item $(v \cdot e_i) \equiv 0 \; (2) \; \; \text{for all\;} 1 \leq i \leq 8$
\item $(M)^2+(v)^2 \equiv 0 \; (8)$.
\end{itemize}
Note that the lattice $\Upsilon'_{g'}$ does indeed exist and is unique up to isometry.
Let $\Lambda'$ be either of the lattices $\Phi'_{g'}$ or $\Upsilon'_{g'}$, and let $M_{\Lambda'}$ be the moduli space of $\Lambda'$ polarised
K3 surfaces, i.e.\ K3 surfaces $Y$ with a primitive embedding $\Lambda' \hookrightarrow \text{Pic}(Y)$ such that $M$ is big and nef. 
\begin{prop}
Suppose $Y$ is a K3 surface with a primitive embedding $\Lambda' \hookrightarrow \text{Pic}(Y)$ such that the elements $e_i \in  \mathfrak{N}$
are the classes of integral rational curves for $1 \leq i \leq 8$. Then $\frac{1}{2}(e_1 +\ldots +e_8)$ defines a double cover $\tilde{Y} \to Y$, branched over $e_1+\ldots+e_8$. The surface $\tilde{Y}$ has eight disjoint $-1$-curves over the $e_i$ and blowing these down produces a Nikulin K3 surface $(X,f)$ such that $Y$ is the minimal desingularisation of the quotient of $X$ by $f$. The rational map $X \dashrightarrow Y$ is the quotient by $f$.
\end{prop}
\begin{proof}
See \cite[\S 1]{garbagnati-sarti-evenset}.
\end{proof}
So if $Y$ is a K3 surface with a primitive embedding $\Lambda' \hookrightarrow \text{Pic}(Y)$ such that the elements $e_i \in  \mathfrak{N}$
are the classes of integral rational curves, then $Y$ is the minimal desingularisation of a Nikulin surface $(X,f)$. The next proposition relates $\text{Pic}(Y)$ to $\text{Pic}(X)$.
\begin{prop} \label{lattice-quotient}
Let $\Lambda$ be either of the lattices $\Phi_g$ or $\Upsilon_g$. Let $[X] \in M_{\Lambda}$ be a polarised K3 surface with corresponding Nikulin involution $f$, and let $Y$ be the desingularisation of the quotient of $X$ by $f$. Then
$[Y] \in M_{\Lambda'}$ where:
\begin{align*}
 \Lambda'&=\Phi'_{\frac{g+1}{2}} \text{ if $g$ is odd and $\Lambda=\Upsilon_g$} \\ 
\Lambda'&=\Upsilon'_{2g-1} \text{ if $\Lambda=\Phi_g$}.
\end{align*}
\end{prop}
\begin{proof}
See \cite[Cor.\ 2.2]{garbagnati-sarti-evenset}.
\end{proof}

Recall the elliptic K3 surface $X_{(2),8A_1}$ from Example \ref{X_{(2),8A_1}} and let $f: X_{(2),8A_1} \to X_{(2),8A_1}$ be the Nikulin involution
induced by translation by the section $\tau$. Let $F$ denote the class of the generic elliptic fibre.
\begin{lem} \label{ample-bloch}
Consider the $f$-invariant line bundle $L=(e+1)F+\sigma+\tau$ on $X_{(2),8A_1}$ for $e \geq 1$, where $\sigma$, $\tau$ denote the sections and $F$ denote the class of the generic elliptic fibre. Then $L$ is primitive. If $e \geq 1$, $L$ is big and nef and if $e \geq 2$ then $L$ is ample.
\end{lem}
\begin{proof}
Let $N$ be an irreducible component of an $I_2$ fibre. Then $(N \cdot L)=1$ which implies that $L$ is primitive.
We have $(L)^2=4e>0$, for $e \geq 1$. Further $L$ has non-negative intersection with all of its components, e.g.\ $(L \cdot \sigma)=(L \cdot \tau)=e-1 \geq 0$ for $e \geq 1$. Thus $L$ is big and nef for $e \geq 1$. If $e \geq 2$ and if $C$ is a curve which dominates $\proj^1$ via the fibration, then $(L \cdot C)>0$; in particular this holds for the sections $\sigma$, $\tau$. Further, if $C$ is now an irreducible component of a fibre, then we again have $(L \cdot C) >0$. Indeed, this is clear if $C \simeq F$ is smooth elliptic and it further holds if $C$ is a component of an $I_2$ fibre, since such a component intersects either $\sigma$ or $\tau$ transversally in one point. It follows that $L$ is ample for $e \geq 2$.
\end{proof}

\begin{prop} \label{type-bloch}
Consider the $f$-invariant, primitive line bundle $L=(e+1)F+\sigma+\tau$ on $X_{(2),8A_1}$ for $e \geq 1$, where $\sigma$, $\tau$ denote the sections and $F$ denotes the class of the generic elliptic fibre. Let $\Lambda$ be the smallest primitive sublattice of $\text{Pic}(X_{(2),8A_1})$ containing $L$ and 
$(H^2(X_{(2),8A_1}, \mathbb{Z})^{f})^{\perp} \simeq E_8(-2)$. Then $\Lambda \simeq \Upsilon_{2e+1}$.
\end{prop}
\begin{proof}
By definition $$\Lambda=(\mathbb{Z}L \oplus (H^2(X_{(2),8A_1}, \mathbb{Z})^{f})^{\perp}) \otimes_{\mathbb{Z}} \mathbb{Q} \cap \text{Pic}(X_{(2),8A_1}).$$ We have $(L)^2=4e$. From the proof of \cite[Prop.\ 2.2]{sar-gee}, the fact that $\Lambda$ is an even, rank 9 lattice of signature $(1,8)$ which contains $\mathbb{Z}L$ and $E_8(-2)$ as primitive sublattices forces $\Lambda \in \{\Phi_{2e+1}, \Upsilon_{2e+1} \}$. So we have to show
$\Lambda \neq \Phi_{2e+1}$.

Let $Y_{(2),8A_1}$ be the minimal desingularisation of the quotient of $X_{(2),8A_1}$ by $f$. Then $Y_{(2),8A_1}$ is itself an elliptic K3 surface with $8$ fibres of type $I_1$ and $8$ fibres of type $I_2$, where the configuration of the singular fibres is obtained by interchanging the $I_1$ and $I_2$ fibres of $X_{(2),8A_1}$, \cite[\S 4]{sar-gee}. Further the fixed points of $f$ correspond to the $8$ nodes of the $I_1$ fibres of $X_{(2),8A_1}$. The quotient $\pi: X_{(2),8A_1} \dashrightarrow Y_{(2),8A_1}$ is defined on an open subset containing the sections $\sigma, \tau$ and the general fibre $F$. Let $\hat{\sigma}$ resp.\ $\hat{F}$ denote the class of the (reduced) images $\pi(\sigma)$ resp.\ $\pi(F)$ in $\text{Pic}(Y_{(2),8A_1})$. Set $M=(e+1)\hat{F}+\hat{\sigma}$. Then $\pi^*(M)=L$ and the invariant line bundle $L$ descends to a line bundle on the (singular) quotient $X_{(2),8A_1} / G$ for $G=<f>$. 

Let $X_t$ be a one-dimensional, flat family of $\Lambda$-polarised K3 surfaces specialising to $X_0=X_{(2),8A_1}$ and such that $\text{Pic}(X_t) \simeq \Lambda$ for $t \neq 0$. Let $\hat{X}_t$ be the quotient of $X_t$ by the Nikulin involution $f_t$, let $Y_t \to \hat{X}_t$ be the minimal desingularisation, and let $\pi_t: X_t \dashrightarrow Y_t$ be the quotient. Let $L_t$ denote the image of $L \in \Lambda \simeq \text{Pic}(X_t)$. For $t$ close to zero and for a general divisor $D_t \in |L_t |$, $D_t$ avoids the fixed points of $f_t$, since this holds for $t=0$. For $t$ general we have a primitive embedding $\text{Pic}(Y_t) \hookrightarrow \text{Pic}(Y_{(2),8A_1})$, so the relation $\pi_*(D_0) \simeq 2 M$ gives a line bundle $M_t:=\frac{1}{2} \pi_*D_t \in \text{Pic}(Y_t)$. For $t$ general and $D'_t \in |M_t|$ general, $D'_t$ does not meet the exceptional divisors of $Y_t$, as this holds for $t=0$, and hence $M_t$ descends to a line bundle on $\hat{X}_t$. We have $L_t \simeq \pi^*M_t$, from the fact that this relation holds for $t=0$. Thus $L_t$ descends to a line bundle on the quotient $\hat{X}_t$ for $t$ general. Since $X_t$ has Picard number nine for $t \neq 0$, \cite[Prop.\ 2.7]{sar-gee} applies and shows that this is only possible in the case $\Lambda= \Upsilon_{2e+1}$.
\end{proof}

\section{The Bloch--Beilinson conjectures for Nikulin involutions} \label{BloBei}
The starting point for our study of the action of Nikulin involutions on the Chow group of points on a K3 surface is the following proposition.
\begin{prop} \label{chow-starting-pt}
Let $f: X \to X$ be a symplectic automorphism of finite order on a projective K3 surface over $\C$. Assume there exists a dominating family of integral
curves $C_t \seq X$ of geometric genus one, such that for generic $t$ the following two conditions are satisfied:
\begin{enumerate}
\item $C_t$ avoids the fixed points of $f$.
\item $C_t$ is invariant under $f$.
\end{enumerate}
Then $f^*=\text{id}$ on $CH^2(X)$.
\end{prop}
\begin{proof}
The set of points $Z_1$ of $x \in X$ such that $[x]=[f(x)] \in CH^2(X)$ is a countable union of closed subsets of $X$ by \cite[Lemma 3]{mumf-chow}. Thus
it suffices to prove that if $x \in X$ is general in the sense that $x$ lies outside a countable union $Z_2$ of proper closed subsets, then $[x]=[f(x)] \in CH^2(X)$. Indeed assume there exists such a $Z_2 \neq X$ and and suppose for a contradiction that $Z_1 \neq X$. Then $X=Z_1 \cup Z_2$ would be a countable union of proper closed subsets, which is impossible over $k=\C$ (proof: restrict to a generic, very-ample divisor and use that $\C$ is uncountable). It thus suffices to show that if $C_t \seq X$ satisfies conditions 1.\ and 2.\ as above then $[x]=[f(x)] \in CH^2(X)$ for all $[x] \in C_t$. Repeating the argument above, it suffices to show $[x]=[f(x)]$ for $x \in C_t$ generic (or even general).

Assume  $C_t$ satisfies conditions 1.\ and 2.\ and let $E \to C_t$ be the normalisation. Then ${f_t}_{|_{C_t}} : C_t \to C_t$ is a finite automorphism. By the universal property of normalisation (see \cite[\S 4.1.2]{liu}), this induces a finite automorphism $\tilde{f}: E \to E$. As ${f_t}_{|_{C_t}}$ is fixed point free, so is $\tilde{f}$. Let $G$ be the finite group of automorphisms of $E$ generated by $\tilde{f}$ and set $D=E / G$. By the Hurwitz formula, $D$ is a smooth curve of genus one. We may choose origins for $E$, $D$, such that the projection $\pi: E \to D$ is a morphism of elliptic curves. So we can view $D$ as a quotient of $E \simeq \text{Pic}^0(E)$ by a finite group $\Gamma \seq \text{Pic}^0(E)$. Thus two points over the fibre of $\pi$ differ by an element of $\Gamma$. Hence $\mathcal{O}_E(x-\tilde{f}(x))$ is a torsion line bundle for any $x \in E$. The composition $E \to X$ of the normalisation morphism with $C_t \hookrightarrow X$ induces a group homomorphism $CH^1(E) \simeq \text{Pic}(E) \to CH^2(X)$. Thus if $x \in C_t$ is a smooth point, $[x-f(x)] \in CH^2(X)$ is a torsion element. Since $CH^2(X)$ is torsion free by Roitman's theorem, see \cite[Ch. 22]{voisin-theorie}, this forces $[x-f(x)]=0$ as required.
\end{proof}

Recall the elliptic K3 surface $X_{(2),8A_1}$ from Example \ref{X_{(2),8A_1}} and let $f: X_{(2),8A_1} \to X_{(2),8A_1}$ be the Nikulin involution $f$
induced by translation by the section $\tau$. Let $L=(e+1)F+\sigma+\tau$ be the $f$-invariant line bundle on $X_{(2),8A_1}$, where $\sigma$, $\tau$ denote the sections and $F$ denotes a generic elliptic fibre. We let $N$ be an $I_2$ fibre of $X_{(2),8A_1}$. As usual, let $Y_{(2),8A_1}$
be the desingularisation of the quotient of $X_{(2),8A_1}$ by $f$, and let $\pi: X_{(2),8A_1} \dashrightarrow Y_{(2),8A_1}$ be the rational quotient map.
Since $N$, $F$, $\sigma$ and $\tau$ avoid the fixed points, they lie in the domain of definition of $\pi$, so we may define $\hat{N}:=\pi(N)$, $\hat{F}:= \pi(F)$, $\hat{\sigma}=\pi(\sigma)=\pi(\tau)$. Note that $\hat{N}$ is a fibre of type $I_1$, \cite[\S 4]{sar-gee}. Now consider the divisor
$$C:= eN+F+\sigma+\tau \in |L| $$
on $X_{(2),8A_1}$.
This divisor is $f$-invariant and avoids the fixed points of $f$. Consider further the divisor
$$\hat{C}:=e \hat{N}+\hat{F}+\hat{\sigma} \in |M|,$$
on $Y_{(2),8A_1}$, where $M := \mathcal{O}_{Y_{(2),8A_1}}((e+1)\hat{F}+\hat{\sigma})$. As $\hat{C}$ avoids the exceptional divisors of $Y_{(2),8A_1}$, we may treat it as a divisor of the singular quotient $ \hat{X}_{(2),8A_1}$ of $X_{(2),8A_1}$ by $f$. We have $\pi^*{M} \simeq L$ and $\pi^{-1}(\hat{C}) = C$.

From Prop \ref{lattice-quotient}, we know that there is a primitive embedding $\Phi'_{e+1} \hookrightarrow \text{Pic}(Y_{(2),8A_1})$. In fact, $M$ is perpendicular to the exceptional divisors and is primitive, since there is a component $R$ of an $I_2$ fibre with $(R \cdot M)=1$. So we have a primitive embedding $\mathbb{Z}M \oplus \mathfrak{N} \hookrightarrow \text{Pic}(Y_{(2),8A_1})$ where $\mathfrak{N}$ is the Nikulin lattice (the classes $e_1, \ldots, e_8$ are sent to the eight components of $I_2$ fibres which avoid $\hat{\sigma}$). Note that if $e \geq 2$, then $M$ is big and nef and $3M-e_1-\ldots-e_8$ is ample.
\begin{lem}
There exists an unramified stable map $$h: D \to  Y_{(2),8A_1}$$ with image contained in $\hat{C}$ and $h_*{D} \simeq M$ and such that 
\begin{itemize}
\item Each component of $D$ is smooth.
\item $D$ has arithmetic genus one.
\end{itemize}
\end{lem}
\begin{proof}
Let $D$ be the unique connected, nodal curve with components $D_1, \ldots, D_{e+2}$ such that
$$ D_1 \simeq \hat{F}, \; D_2 \simeq \hat{\sigma}, \; D_i \simeq \proj^1 \text{  for $3 \leq i \leq e+2$}$$
and such that $D_l \cap D_{l+1} = \{ p_l \}$ for some point $p_l \in D$ with $p_1=\hat{F} \cap \hat{\sigma}$, and $D_l \cap D_{m} = \emptyset$ for
$m \geq l+2$. The stable map $h$ is defined to be the unique unramified morphism $h: D \to  Y_{(2),8A_1}$ with $h(D) \seq \hat{C}$,  $h_*{D} \simeq M$ and such that
$h_{|_{D_1}}$ is the closed immersion $\hat{F} \hookrightarrow Y_{(2),8A_1}$, $h_{|_{D_2}}$ is the closed immersion $\hat{\sigma} \hookrightarrow Y_{(2),8A_1}$ and $h_{|_{D_i}}$ is the normalisation $ \proj^1 \to \hat{N}$. This is illustrated in the diagram below (taken from \cite{huy-kem}):
$$\setlength{\unitlength}{.013in}
\begin{picture}(124,100)
\put(42,40){\tiny\mbox{$D_1$}}
\put(22,65){\tiny\mbox{$D_2$}}
\put(-4,85){\tiny\mbox{$D_3$}}
\put(-94,85){\tiny\mbox{$D_n$}}
\put(-150,80){\qbezier(60,0)(40,-0)(75,-80)}
\put(-150,80){\qbezier(60,0)(80,-0)(61.5,-41.4)}
\put(-150,80){\qbezier(59.6,-45)(56,-60)(43,-80)}
\put(-125,80){\qbezier(60,0)(40,-0)(75,-80)}
\put(-125,80){\qbezier(60,0)(80,-0)(61.5,-41.4)}
\put(-125,80){\qbezier(59.6,-45)(56,-60)(43,-80)}
\put(-77.8,6.4){\circle*{2}}
\put(9.8,62){\circle*{2}}
\put(40,62){\circle*{2}}
\put(-52.8,6.4){\circle*{2}}
\put(-40,40){\mbox{$\ldots$}}
\put(40,5){\line(0,1){70}}
\put(5.7,62){\line(1,0){50}}
\put(90,55){\tiny\mbox{$\to$}}
\put(-60,80){\qbezier(60,0)(40,-0)(75,-80)}
\put(-60,80){\qbezier(60,0)(80,-0)(61.5,-41.4)}
\put(-60,80){\qbezier(59.6,-45)(56,-60)(43,-80)}
\put(-12.7,6.4){\circle*{2}}
{  \linethickness{0.28mm}
\put(110,80){\qbezier(60,0)(40,-0)(65,-80)}
\put(110,80){\qbezier(60,0)(80,-0)(55,-80)}
\put(170,16.6){\circle*{3}}

}
\put(230,5){\line(0,1){70}}
\put(174.7,62){\line(1,0){70}}
\put(178.9,62.1){\circle*{2}}
\put(230,62.1){\circle*{2}}
\put(200,65){\tiny\mbox{$\hat{\sigma}$}}
\put(180.5,40){\tiny\mbox{$e\hat{N}$}}
\put(232.5,40){\tiny\mbox{$\hat{F}$}}
\end{picture}
$$
\end{proof}

\begin{lem}
Assume $e \geq 2$. Let $T$ be a smooth, integral scheme of finite type over $\C$ with distinguished point $0 \in T$. Let $\mathcal{Y} \to T$ be a projective, flat-family of $\Phi'_{e+1}$-polarised K3 surfaces such that $\mathcal{Y}_0 \simeq Y_{(2),8A_1}$ and 
$$\Phi'_{e+1} \simeq \text{Pic}(\mathcal{Y}_t)$$ for $t \neq 0 \in T$. Assume that $e_i \in \text{Pic}(\mathcal{Y}_t)$ are the classes of integral $(-2)$-curves for all $t \in T$, $1 \leq i \leq 8$, and that the image of $M \in \Phi'_{e+1}$ gives a big and nef class $M_t \in \text{Pic}(\mathcal{Y}_t)$ for all $t$. For any $t \in T$, consider the Nikulin surface $(\mathcal{X}_t, f_t)$ with quotient map $\pi_t: \mathcal{X}_t \dashrightarrow \mathcal{Y}_t$ defined by $\frac{1}{2}(e_1+\ldots+e_8)$. Then, there exists a smooth curve $T'$ with a morphism $T' \to T$, dominant about the origin, such that the stable map $h: D \to  Y_{(2),8A_1}$ from the previous lemma deforms to a family of stable maps $$h_t: D_t \to \mathcal{Y}_t$$ 
for $t \in T'$. Furthermore, for $t$ general
\begin{itemize}
\item $h_t$ is unramified and $h_t(D_t)$ is an integral curve of geometric genus one which avoids the exceptional divisors on $\mathcal{Y}_t$.
\item $\pi^{-1}_t(h_t(D_t))$ is an integral, geometric genus one curve which is $f_t$ invariant.
\end{itemize}
\end{lem}
\begin{proof}
After replacing $T$ by an \'etale base change, we may assume we have effective, $T$-relatively flat divisors $E_i \in \text{Pic}(\mathcal{Y})$ for $1 \leq i \leq 8$ with $E_{i,s} \simeq e_i \in \text{Pic}(\mathcal{Y}_s)$, and we may assume that $\frac{1}{2}(E_1+\ldots+E_8)$ is the class of a $T$-flat line bundle. We may further assume that we have a $T$-flat line bundle $\mathcal{M} \in \text{Pic}(\mathcal{Y})$ with $\mathcal{M}_t \simeq M_t$. Then $\frac{1}{2}(E_1+\ldots+E_8)$ defines a flat family $\mathcal{X} \to T$ of Nikulin surfaces, with quotient maps $\pi_t: \mathcal{X}_t \dashrightarrow \mathcal{Y}_t$.

From Proposition \ref{relative-dominant}, the stable map $h: D \to  Y_{(2),8A_1}$ from the previous lemma deforms to a one-dimensional family of stable maps $$h_t: D_t \to \mathcal{Y}_t$$ for $t \in T'$, where $T' \to T$ is dominant about $0 \in T$. The curve $D_t$ has arithmetic genus one and, for general $t \in T$, $h_t$ is unramified and $h_t(D_t)$ does not meet the exceptional divisors. Thus we may consider ${h_t}_*(D_t)$ as a divisor on $\hat{X}_t$, where $ \hat{X}_t$ denotes the singular quotient of $\mathcal{X}_t$ by $f_t$. Since we are assuming $\Phi'_{e+1} \simeq \text{Pic}(\mathcal{Y}_t)$ for $t \neq 0 \in T$, we have that $\text{Pic}(\hat{X}_t) \simeq \mathbb{Z} M_t$, which implies that both $D_t$ and $h_t(D_t)$ are integral curves of geometric genus one, for $t \neq 0$ general. 

Now consider the curve $\pi^{-1}_t(h_t(D_t)) \in \mathcal{X}_t$. This lies in the unramified locus of $\pi_t: \mathcal{X}_t \to \hat{X}_t$, so there are two possibilities. In the first possibility, for general $t$, $\pi^{-1}_t(h_t(D_t))$ is an integral, geometric genus one curve which is $f_t$ invariant, and we are done. In the second case, $\pi^{-1}_t(h_t(D_t))$ has two component $C_{1,t}$ and $C_{2,t}$ which are interchanged by $f_t$ for $t$ general. But then, if $T^o \seq T'$ is an affine open subset about $0$, there exists a finite and surjective base change $S \to T^o$ such that $C_{1,t}$ and $C_{2,t}$ can be extended to deformations $\mathcal{C}_1 \to S$ and $\mathcal{C}_2 \to S$, cf.\ \cite[\href{http://stacks.math.columbia.edu/tag/0551}{Tag 0551}]{stacks-project}. Further, $C_{1,t}$ and $C_{2,t}$ deform (over $S$) to connected subcurves $C_1$ and $C_2$ of $$C=\pi^{-1}(\hat{C})=eN+F+\sigma+\tau \seq X_{(2),8A_1},$$ with $C_1+C_2=C$. Without loss of generality, assume $C_1$ contains $F$. Since $C_{2,t}=f_t(C_{1,t})$, this implies $C_2$ contains $F=f_t(F)$, contradicting that $F$ occurs with multiplicity one in $C$.
\end{proof}

\begin{thm} \label{chow-thm}
Let $X$ be a generic $\Upsilon_{2e+1}$-polarised K3 surface for any integer $e \geq 2$, and let $f$ be the Nikulin involution on $X$ corresponding to the embedding $\Upsilon_{2e+1} \hookrightarrow \text{Pic}(X)$. Then Conjecture \ref{bloch} holds for $f: X \to X$.
\end{thm}
\begin{proof}
Let $Y$ be the desingularisation of the quotient $\hat{X}$ of $X$ by $f$ and $\pi: X \to \hat{X}$ the quotient map. As $X$ is generic, by the previous  
lemma there exists an unramified stable map $h': D \to Y$ such that $D$ is integral of arithmetic genus one, $h'(D)$ avoids the exceptional divisors and $\pi^{-1}(h'(D))$ is integral (there exist projective families $\mathcal{Y} \to T$ as in the lemma since $3M-e_1-\ldots-e_8 \in \text{Pic}(Y_{(2), 8A_1})$ is ample). Since $h'_*D$ is primitive (as $L$ is primitive on $X_{(2),8A_1}$), $h'$ must be birational to its image. From Proposition \ref{relative-dominant}, there is a one-dimensional family of non-isomorphic stable maps $$h'_t: D_t \to Y$$ to $Y$ with $h'_0=h'$. For $t$ general, we have that $h'_t$ is unramified and birational to its image, which implies that $h'_t(D_t)$ cannot be constant for $t$, and thus $h'_t(D_t)$ is a dominating family of integral elliptic curves on $Y$. For $t$ general, $h'_t(D_t)$ avoids the exceptional divisors of $Y$ and $\pi^{-1}(h'_t(D_t))$ is integral. Thus $\pi^{-1}(h'_t(D_t))$ gives a dominating family of integral, elliptic (by Hurwitz) curves which are $f$-invariant and avoid the fixed points. Thus the claim follows from Proposition \ref{chow-starting-pt}.
\end{proof}
\begin{cor} \label{bloch-thm-onethird}
Conjecture \ref{bloch} holds for an arbitrary $\Upsilon_{2e+1}$-polarised K3 surface $(X,f)$.
\end{cor}
\begin{proof}
For $e=1$, the result holds from \cite{pedrini}, so assume $e \geq 2$. Let $(X,f)$ be an arbitrary $\Upsilon_{2e+1}$-polarised K3 surface. Any closed point $x \in X$ is a specialisation of points $x_t \in \mathcal{X}_t$, such that $(\mathcal{X}_t, f_t)$ is a generic $\Upsilon_{2e+1}$-polarised K3 surface. As $[x_t]=[f_t(x_t)] \in CH^2(\mathcal{X}_t)$ from Theorem \ref{chow-thm}, specialisation gives $[x]=[f(x)]$, cf. \cite[Lemma 3]{mumf-chow}.
\end{proof}
\begin{remark}
These results have been subsequently extended to cover \emph{all} K3 surfaces admitting a Nikulin involution, see \cite{voisin-bloch}.
\end{remark}

\begin{remark}
Any symplectic automorphism $f$ of prime order on a K3 surface has order $2,3,5$ or $7$.
The arguments given here can be extended to show that Conjecture \ref{bloch} holds on certain components of the moduli space of 
pairs $(X,f)$ of K3 surfaces with a symplectic automorphism of order $p=3,5$ or $7$, assuming the invariant polarisation $L$ satsifies $(L)^2=2ep$ for some integer $e$. See \cite[\S 5]{huy-kem}. Also see the follow-up paper \cite{huy-derived}, which proves the result in full generality (i.e.\ for arbitrary symplectic automorphisms of finite order) using derived categories.
\end{remark}

\chapter{The marked Wahl map} \label{sec:markedwahl}
Recall the following definition from \cite{wahl-curve}: let $V$ be any smooth projective variety, and let $R$ be a line bundle on $V$. Then there is a linear map, called the \emph{Gaussian}:
\begin{align*}
\Phi_R \; : \; \bigwedge^2 H^0(V,R) & \to H^0(V, \Omega_V (R^2)) \\
 s \wedge t & \mapsto sdt-tds.
\end{align*} 
 In the case $R = \omega_V$, this map is called the \emph{Wahl map}. For $V=C$ a smooth curve, and $T \seq C$ a marking, we call $\Phi_{\omega_C(-T)}$ the \emph{marked Wahl map}. When $T= \emptyset$, the Wahl map $\Phi_{\omega_C}$ was first studied in \cite{wahl-jac}, where it appears in connection to the deformation theory of the cone over the curve $C$, embedded in projective space via the canonical line bundle. In particular, Wahl shows that $\Phi_{\omega_C}$ is surjective for most complete intersection curves, but not for curves lying on a K3 surface (see also \cite{beauville-merindol} for a simplified proof). This gave the first known obstruction for a curve to lie on a K3 surface. \footnote{Other obstructions have since been discovered, \cite[\S 10]{mukai-nonabelian}, \cite{voisin-acta}.} 
 \section{The marked Wahl map}
 In this section we will use an approach inspired by \cite{wahl-plane-nodal} to study the marked Wahl map in the case where $C$ is a smooth curve and $T \neq \emptyset$. We will be particularly interested in the case where $C$ is the normalisation of a nodal curve lying on a K3 surface and $T$ is the divisor over the nodes. In this case we will show that the non-surjectivity of $\Phi_{\omega_C(-T)}$ gives an obstruction for a marked curve to have a nodal model on a K3 surface, in such a way that $T$ is the (unordered) marking over the nodes.

We begin with the following lemma, which is a special case of \cite[Lem.\ 3.3.1]{greuel}:
\begin{lem} \label{jjji}
Let $x_1, \ldots, x_n, y_1, \ldots, y_m$ be distinct, generic points of $\proj^2$ and let $d$ be a positive integer satisfying
$$3n+6m < \frac{d^2+6d-1}{4}-\left \lfloor \frac{d}{2} \right \rfloor .$$
Then there exists an integral curve $C \seq \proj^2$ of degree $d$ with nodes at $x_i$, ordinary singular points of multiplicity $3$ at $y_j$ for $1 \leq i \leq n$, $1 \leq j \leq m$ and no other singularities.
\end{lem}
Let $C \seq \proj^2$  be an integral curve of degree $d$ with nodes at $x_i$, ordinary singular points of multiplicity three at $y_j$ for $1 \leq i \leq n$, $1 \leq j \leq m$ and no other singularities, as in the lemma above. Let $\pi: S \to \proj^2$ be the blow-up at $x_1, \ldots, x_n, y_1, \ldots, y_m$, let $E_X$ be the sum of the exceptional divisors over $x_i$ for $1 \leq i \leq n$ and let $E_Y$ be the sum of the exceptional divisors over $y_j$ for $1 \leq j \leq m$. Denote by $D$ the strict transform of $C$, and let $T \seq D$ be the marking $E_X \cap D$. Note that $D$ is smooth, since all singularities are ordinary. Set $$M=\mathcal{O}_S((d-3)H-2E_X-2E_Y),$$ where $H$ denotes the pull-back of the hyperplane of $\proj^2$. Note that 
$$K_D \sim (D+K_S)_{|_D} \sim (dH-2E_X-3E_Y)+(E_X+E_Y-3H),$$ and this gives
$$ M_{|_D}\simeq K_D(-T).$$  We therefore have the following commutative diagram
\begin{align} \label{gauss-diag}
\xymatrix{
\bigwedge^2 H^0(S, M) \ar[r]^{\; \; \; \; \Phi_{M} \; \; \; \;}  \ar[d] & H^0(S, \Omega_{S}(M^2)) \ar[d]^g \\
\bigwedge^2 H^0(D, K_{D}(-T)) \ar[r]^{\; \; \; \; \; W_{D,T} \; \; \; \; \; \; }  &  H^0(D, K^3_{D}(-2T)).
}
\end{align}
where $\Phi_{M}$ is the Gaussian, \cite[\S 1]{wahl-curve} and $ W_{D,T} $ is the marked Wahl map of $(D,T)$. Here $g$ denotes the composition of the natural maps
$H^0(S, \Omega_{S}(M^2)) \to H^0(D,{\Omega_{S}(M^2)}_{|_D})$ and $H^0(D,{\Omega_{S}(M^2)}_{|_D}) \to H^0(D,\Omega_{D}(M^2))$. We aim to show that $W_{D,T}$ is surjective. We will firstly show that $g$ is surjective. The main tool we will need is the Hirschowitz criterion, \cite{hirschowitz}: 
\begin{thm}[Hirschowitz] \label{hirsch}
Let $p_1, \ldots, p_t$ be generic points in the plane, and assume $m_1, \ldots, m_t, d$ are nonnegative integers satisfying
$$ \sum_{i=1}^t \frac{m_i(m_i+1)}{2} < \left \lfloor \frac{(d+3)^2}{4} \right \rfloor .$$
Then $H^1(\proj^2, \mathcal{O}_{\proj^2}(d) \otimes I_{p_1}^{m_1} \otimes \ldots \otimes I_{p_t}^{m_t})=0$.
\end{thm}
\begin{lem} \label{gauss-lem1}
Assume $d \geq 8$, $3n+m < \left \lfloor \frac{(d-5)^2}{4} \right \rfloor$. Then $H^1(S,\Omega_S((d-6)H-2E_X-E_Y))=0.$
\end{lem}
\begin{proof}
The relative cotangent sequence twisted by $ (d-6)H-2E_X-E_Y$ gives
$$0 \to \pi^* \Omega_{\proj^2}((d-6)H-2E_X-E_Y) \to \Omega_{S}((d-6)H-2E_X-E_Y) \to \omega_{E_X}(2) \oplus \omega_{E_Y}(1) \to 0 .$$ Thus it suffices to show
$$H^1(\proj^2, \Omega_{\proj^2}(d-6) \otimes I_X^2 \otimes I_Y )=H^1(S, \pi^* \Omega_{\proj^2}((d-6)H-2E_X-E_Y) )=0$$ 
where $I_X=I_{x_1} \otimes \ldots \otimes I_{x_n}$ and $I_Y=I_{y_1} \otimes \ldots \otimes I_{y_m}$.
Twisting the Euler sequence by $\omega_{\proj^2}$ gives a short exact sequence
$$0 \to \omega_{\proj^2} \to \mathcal{O}_{\proj^2}(-2)^{\oplus 3} \to \Omega_{\proj^2} \to 0 ,$$
where we have used the standard identification
$$\Omega_{\proj^2} \simeq T_{\proj^2} \otimes \omega_{\proj^2},$$
from \cite[Ex.\ II.5.16(b)]{har}. As $H^2(\proj^2, \omega_{\proj^2}(d-6) \otimes I_X^2 \otimes I_Y)=0$ for $d \geq 7$, it suffices to show $H^1(\proj^2,  \mathcal{O}_{\proj^2}(d-8)\otimes I_X^2 \otimes I_Y )=0$. This follows from Theorem \ref{hirsch} and the assumption $d \geq 8$, $3n+m < \left \lfloor \frac{(d-5)^2}{4} \right \rfloor$.
\end{proof}
\begin{lem} \label{gauss-lem2}
Let $D$ be as above and assume $3n+m < \left \lfloor \frac{(d-3)^2}{4} \right \rfloor$. Assume further that $m \geq 10$, so that $d \geq 10$. Then $$H^1(D,\mathcal{O}_D((d-6)H-2E_X-E_Y))=0.$$
\end{lem}
\begin{proof}
We have an exact sequence
$$ 0 \to \mathcal{O}_S(-6H+2E_Y) \to \mathcal{O}_S((d-6)H-2E_X-E_Y) \to \mathcal{O}_D((d-6)H-2E_X-E_Y) \to 0.  $$
By the Hirschowitz criterion, $H^1(S,\mathcal{O}_S((d-6)H-2E_X-E_Y))=0$, as we are assuming $d \geq 6$, $3n+m < \left \lfloor \frac{(d-3)^2}{4} \right \rfloor$. Thus it suffices to show $H^2(S,\mathcal{O}_S(-6H+2E_Y) )=0$. By Serre duality, $h^2(S,\mathcal{O}_S(-6H+2E_Y) )=h^0(S,\mathcal{O}_S(3H+E_X-E_Y))$.
We have 
$$ 0 \to \mathcal{O}_S(3H-E_Y) \to \mathcal{O}_S(3H+E_X-E_Y) \to \mathcal{O}_{E_X}(-1) \to 0$$
and so it suffices to show $H^0(S,\mathcal{O}_S(3H-E_Y))=0$. But $H^0(S,\mathcal{O}_S(3H-E_Y))=H^0(\proj^2, \mathcal{O}(3) \otimes I_Y)=0$, since $m \geq 10$, and any ten general points do not lie on any plane cubic (as the space of plane cubics has dimension nine).
\end{proof}
\begin{cor} \label{f-surj}
Let $x_1, \ldots, x_n, y_1, \ldots, y_m$ be distinct, general points of $\proj^2$ with $m \geq 10$ and let $d \geq 21$ be a positive integer satisfying
$$3n+6m < \left \lfloor \frac{(d-5)^2}{4} \right \rfloor.$$ Let  $C \seq \proj^2$ be an integral curve as in Lemma \ref{jjji}. Let $S \to \proj^2$ denote the blow-up of $\proj^2$ at $x_1, \ldots, x_n, y_1, \ldots, y_m$, and let $D \seq S$ denote the strict transform of $C$. Then the map $g$ from Diagram (\ref{gauss-diag}) is surjective.
\end{cor}
\begin{proof}
Note that $ \left \lfloor \frac{(d-5)^2}{4} \right \rfloor < \frac{d^2+6d-1}{4}-\left \lfloor \frac{d}{2} \right \rfloor $ for $d \geq 5$ so that such a curve $C$ exists. Let $M$ be the line bundle defined above Diagram (\ref{gauss-diag}).
We have short exact sequences
$$ 0 \to \Omega_S((d-6)H-2E_X-E_Y) \to \Omega_S(M^2) \to \Omega_S(M^2)_{|_D} \to 0 $$ and
$$0 \to \mathcal{O}_D ((d-6)H-2E_X-E_Y) \to  \Omega_S(M^2)_{|_D} \to \Omega_D(M^2) \to 0 .$$ 
The map $f$ is the composition of the natural maps $H^0(S,\Omega_S(M^2)) \to H^0(S, \Omega_S(M^2)_{|_D})$ and 
$H^0(D,\Omega_S(M^2)_{|_D} ) \to H^0(D, \Omega_D(M^2) )$, so the claim follows from lemmas \ref{gauss-lem1} and \ref{gauss-lem2}.
\end{proof}

We now wish to show that the Gaussian $\Phi_M$ from Diagram \ref{gauss-diag} is surjective. We start by recalling one construction of Gaussian maps from \cite[\S 1]{wahl-curve}. Let $X$ be a smooth, projective variety, and $L \in \text{Pic}(X)$ a line bundle. Let $Y \to X \times X$ be the blow-up of the diagonal $\Delta$, and let $F$ denote the exceptional divisor. There is a short exact sequence of sheaves on $X \times X$
$$ 0 \to I_{\Delta}^2 \to I_{\Delta} \to \Delta_* \Omega_X \to 0.$$ Twisting the above sequence by $L \boxtimes L$ produces a short exact sequence
$$ 0 \to I_{\Delta}^2 (L \boxtimes L) \to I_{\Delta}(L \boxtimes L) \to \Delta_* (\Omega_X(L^2)) \to 0$$ and upon taking cohomology we get a map
$$ \widetilde{\Phi}_L : H^0(X \times X,I_{\Delta}(L \boxtimes L)) \to H^0(X, \Omega_X(L^2)).$$ Now $H^0(X \times X,I_{\Delta}(L \boxtimes L)) $ may be identified with the kernel $\mathcal{R}(L,L)$ of the
multiplication map $H^0(X,L) \otimes H^0(X,L) \to H^0(X,L^2)$, and we have $\bigwedge^2 H^0(X,L) \seq  \mathcal{R}(L,L)$ by sending $s \wedge t$ to $s \otimes t-t\otimes s$. Further, $\Phi_L$ is the restriction of $\widetilde{\Phi}_L$ to $\bigwedge^2 H^0(X,L)$, and it is easily verified that both $\Phi_L$ and $\widetilde{\Phi}_L$ have the same image in $H^0(X, \Omega_X(L^2))$. Thus, to verify the surjectivity of $\Phi_L$, it suffices to show 
$$H^1(X \times X,I_{\Delta}^2 (L \boxtimes L))=H^1(Y, L_1+L_2-2F)=0$$
where $L_1$ and $L_2$ denote the pull-backs of $L$ via the projections $pr_i: Y \to X \times X \to X$, for $i=1,2$.

Following \cite{cili-corank}, \cite{wahl-plane-nodal}, we now wish to use the Kawamata--Viehweg vanishing theorem to show $H^1(Y, L_1+L_2-2F)=0$.
\begin{prop}[\!\! \cite{cili-corank}] \label{kv-van}
Let $X$ be a smooth projective surface, which is not isomorphic to $\proj^2$. Assume $L \in \text{Pic}(X)$ is a line bundle such that there exist three very ample line bundles $M_1, M_2, M_3$ with $L-K_X \sim M_1+M_2+M_3$. Then the Gaussian $\Phi_L$ is surjective.
\end{prop}
\begin{proof}
For a line bundle $A$ on $X$, we denote by $A_i \in \text{Pic}(Y)$ the pullback via the projection $pr_i: Y \to X \times X \to X$, for $i=1,2$.
By the above discussion, it suffices to show $H^1(Y, L_1+L_2-2F)=0$. As the diagonal $\Delta \seq X \times X$ has codimension two, we have $K_Y \simeq g^*K_X+F$, \cite[Exercise II.8.5]{har}. Thus we see $H^1(Y,L_1+L_2-2F)=H^1(Y,(L-K_X)_1+(L-K_X)_2-3F+K_Y)$, and so by the Kawamata--Viehweg vanishing theorem it suffices to show
$(L-K_X)_1+(L-K_X)_2-3F$ is big and nef. Since we have
\begin{align*}
(L-K_X)_1+(L-K_X)_2-3F &= (M_{1,1}+M_{1,2}-F) + (M_{2,1}+M_{2,2}-F) \\
& + (M_{3,1}+M_{3,2}-F), 
\end{align*}
it suffices to show that $M_{i,1}+M_{i,2}-F$ is big and nef for $1 \leq i \leq 3$. Now $H^0(Y,M_{i,1}+M_{i,2}-F)$ is the kernel $\mathcal{R}(M_i,M_i)$ of the
multiplication map $H^0(X,M_i) \otimes H^0(X,M_i) \to H^0(X,M_i^2)$, and we have an injective map $\bigwedge^2 H^0(X,M_i) \hookrightarrow \mathcal{R}(M_i,M_i)$ 
sending $s \wedge t$ to $s \otimes t-t\otimes s$. Thus $\bigwedge^2 H^0(X,M_i) $ induces a sublinear system of $|M_{i,1}+M_{i,2}-F|$ which induces a rational map
\begin{align*}
\psi_i \; : \; Y &\to Gr(1, \proj(H^0(M_i)^*)) \\
(x,y) &\mapsto \overline{\phi_i(x) \phi_i(y)} 
\end{align*}
where $\phi_i: X \hookrightarrow \proj(H^0(M_i)^*)$ is the embedding induced by $M_i$, and where $\overline{\phi_i(x) \phi_i(y)} $ denotes the line through $\phi_i(x) $ and $\phi_i(y)$. By viewing $(x,y) \in F$ as a pair $x \in X$, $y \in T_{X,x}$, one sees that the map $\psi_i $ is in fact globally defined, and hence $M_{i,1}+M_{i,2}-F$ is nef. To see that it is big, it suffices to show that $\psi_i$ is generically finite, i.e.\ we need to show that there exist points $x,y \in X$ such that $\phi_i(X)$ does not contain the line $\overline{\phi_i(x) \phi_i(y)}$. But if this were not the case $\phi_i(X)$ would be a linear space, contrary to the hypotheses.
\end{proof}
We now return to the situation of the blown-up plane. We start with the following:
\begin{thm}[\!\! \cite{hirsch-zeit}] \label{hirsch-almeida}
Let $p_1, \ldots, p_k$ be generic distinct points in the plane, and let $\pi: S \to \proj^2$ be the blow-up. Let $E \seq S$ be the exceptional divisor, and let $H$ be the pull-back of the hyperplane class on $\proj^2$. If we assume
$d \geq 5$ and $k+6 \leq \frac{(d+1)(d+2)}{2},$
then $dH-E$ is very ample on $S$.
\end{thm}
Putting everything together, we deduce:
\begin{prop} \label{mark-surj-main}
Let $x_1, \ldots, x_n, y_1, \ldots, y_m$ be distinct, generic points of $\proj^2$ with $m \geq 10$ and let $d \geq 24$ be a sufficiently large integer, so that both of the following two conditions are satisfied
\begin{enumerate}
\item $3n+6m < \left \lfloor \frac{(d-5)^2}{4} \right \rfloor$,
\item $n+m+6 \leq \frac{(d-6)(d-3)}{18}$.
\end{enumerate}
Let  $C \seq \proj^2$ be an integral curve of degree $d$ with nodes at $x_i$, ordinary singular points of multiplicity $3$ at $y_j$ for $1 \leq i \leq n$, $1 \leq j \leq m$ and no other singularities. Let $S \to \proj^2$ denote the blow-up of $\proj^2$, and let $D \seq S$ denote the strict transform of $C$. Then the marked Wahl map $W_{D,T}$ is surjective, where $T$ is the divisor over the nodes of $C$. Furthermore, $h^0(D,\mathcal{O}_D(T))=1$.
\end{prop}
\begin{proof}
We will firstly show that $W_{D,T}$ is surjective. We have already seen in Corollary \ref{f-surj} that the map $f$ from Diagram \ref{gauss-diag} is surjective, thanks to the assumption $3n+6m < \left \lfloor \frac{(d-5)^2}{4} \right \rfloor$. Thus it suffices to show that $\Phi_M$ is surjective, where
$M=\mathcal{O}_S((d-3)H-2E_X-2E_Y)$. Now $M-K_S \sim (d-6)H-3E_X-3E_Y$ can be written as the sum
\begin{align*}
M-K_S &= (\lfloor \frac{d-6}{3}\rfloor H-E_X-E_Y)+ (\lfloor \frac{d-6}{3}\rfloor H-E_X-E_Y)\\
&+ ((d-6-2\lfloor \frac{d-6}{3}\rfloor)H-E_X-E_Y).
\end{align*}
Since we are assuming $d \geq 24$, $n+m+6 \leq \frac{(d-6)(d-3)}{18}$, Theorem \ref{hirsch-almeida} shows that $M-K_S$ may be written as a sum of three very ample line bundles
(use $\lfloor \frac{d-6}{3}\rfloor \geq \frac{d-9}{3}$). Thus Proposition \ref{kv-van} implies that the Gaussian $\Phi_M$ is surjective.

For the second statement, note that the short exact sequence
$$0 \to \mathcal{O}_S \to \mathcal{O}_S(E_X) \to \mathcal{O}_{E_X}(-1) \to 0 $$
gives that $h^0(S, \mathcal{O}_S(E_X) )=1$. From the sequence
$$ 0 \to \mathcal{O}_S(3E_X+3E_Y-dH) \to \mathcal{O}_S(E_X) \to \mathcal{O}_D(T) \to 0$$
it suffices to show $H^1(S,\mathcal{O}_S(3E_X+3E_Y-dH))=0$. But $dH-3E_X-3E_Y$ is a sum of three very ample line bundles, so it is big and nef, so that this follows from the Kawamata--Viehweg vanishing theorem and Serre duality.

\end{proof}
As an immediate consequence we have:
\begin{thm} \label{infinite}
Fix any integer $l \in \mathbb{Z}$. Then there exist infinitely many integers $h(l)$, such that the general marked curve $[(C,T)] \in \widetilde{\mathcal{M}}_{h(l),2l}$ has surjective marked Wahl map.
\end{thm}
\begin{proof}
Consider the curve $D \seq S$ from Proposition \ref{mark-surj-main}, applied to $n=l$ and $m=10$ (choose any $d$ satisfying the hypotheses of the proposition). Let $h(l)$ denote the genus of $D$. In an open subset about  $[(D,T)] \in \widetilde{\mathcal{M}}_{h(l),2l}$, we have $h^0(D,\mathcal{O}_D(T))=1$ and thus 
$h^0(D,K_D(-T))=\chi(K_D(-T))+1$ is locally constant. Further, the equality $h^0(D,K_D^3(-2T))=\chi(K_D^3(-2T))$ holds, since $\deg(K_D^3(-2T)) >2h(l)-2$. Thus the claim follows immediately from Proposition \ref{mark-surj-main} and semicontinuity.
\end{proof}
\begin{remark}
In our example $(D,T)$, we have that $l$ is of the order $\frac{h}{8}$, where $h$ is the genus of $D$. Indeed, if $n=l$ is large, we can take $d^2$ to be approximately
$18l$, so that $g(C)=\frac{1}{2}(d-1)(d-2) \sim 9l$ and $h \sim 8l$. Thus one would expect that the marked Wahl map of a general marked curve in $\widetilde{\mathcal{M}}_{h,2l}$ is surjective, so long as $l$ is at most of the order $\frac{h}{8}$.
\end{remark}

We will now study the marked Wahl map for curves arising via the normalisation of nodal curves on K3 surfaces. 
\begin{thm} \label{marked-wahl-k3}
Assume $g-n \geq 13$ for $k=1$ or $g \geq 8$ for $k >1$, and let $ n \leq \frac{p(g,k)-2}{5}$. Then there is an irreducible component $I^0 \seq \mathcal{V}^n_{g,k}$ such that for a general $[(f: C \to X,L)] \in I^0$ the marked Wahl map $W_{C,T}$ is nonsurjective, where $T \seq C$ is the divisor over the nodes of $f(C)$.
\end{thm}

\begin{proof}  
Let $I \seq \mathcal{T}^n_{g,k}$ be the irreducible component from Theorem \ref{finiteness} (in the case $k=1$) or Theorem \ref{finiteness-nonprim} (in the case $k>1$) and set $I^0=I \cap \mathcal{V}^n_{g,k}$, which is nonempty from \cite[Lemma 3.1]{chen-rational}.
Let $\pi: Y \to X$ be the blow up of the K3 surface $X$ at the nodes of $f(C)$. There is a natural closed immersion $C \seq Y$. Let $E \seq Y$ denote the sum of the exceptional divisors, and let $M= \mathcal{O}_Y(C)$. We have $K_{C}=(M+E)_{|_{C}}$ by the adjunction formula.
Consider the following commutative diagram:
\[
\xymatrix{
\bigwedge^2 H^0(Y, M) \ar[r]^{\Phi_{M} \; \; \; \;}  \ar[d]^h & H^0(Y, \Omega_{Y}(M^2)) \ar[d]^g \\
\bigwedge^2 H^0(C, K_{C}(-T)) \ar[r]^{\; \; \; W_{C,T} \; \; \; \; }  &  H^0(C, K^3_{C}(-2T)).
}
\]
 where the top row is a Gaussian on $Y$. 
The map $h$ is surjective, as $H^1(Y, \mathcal{O}_Y)=0$.
Suppose for a contradiction that the marked Wahl map $W_{C,T}$ were surjective. Then $g$ would be surjective, and hence the natural map
$$ H^0(C, \Omega_{Y|_{C}}(M^2)) \to H^0(C, K^3_{C}(-2T))$$ would also be surjective.  Now consider the short exact sequence
$$0 \to M_{|_{C}} \to \Omega_{Y|_{C}}(M^2) \to K^3_{C}(-2T) \to 0 .$$ Since 
$$H^1(C,M)=H^1(C,K_{C}(-E))=H^0(C,E) \neq 0,$$
the surjectivity of $g$ would imply that $h^1(C,\Omega_{Y|_{C}}(M^2))=h^0(C,T_{Y|_{C}}(2E-K_{C})) \neq 0$.
However $$H^0(C,T_{Y|_{C}}(2E-K_{C})) \seq H^0(C, f^*(T_X)(2E-K_{C}))\seq H^0(C, f^*(T_X))=0$$
from Lemma \ref{coh-van}, and since $K_{C}-2E$ is effective for $n \leq \frac{p(g,k)-2}{5}$. So this is a contradiction and hence $W_{C,T}$ is nonsurjective.

\end{proof}

\chapter[Brill--Noether theory for nodal curves on K3]{Brill--Noether theory for nodal curves on K3 surfaces}
In this section we consider two related questions on the Brill--Noether theory of nodal curves on a K3 surface.  Let $D \seq X$ be a nodal curve on a K3 surface, and let $C:= \tilde{D}$ be the normalisation of $D$. In Section \ref{BNP-nodal}, we consider the Brill--Noether theory of the smooth curve $C$. The main tool used here is the Lazarsfeld--Mukai bundle, \cite{lazarsfeld-bnp}, and degenerations to special K3 surfaces with disjoint rational curves. In Section \ref{rat} we consider the Brill--Noether theory for the nodal curve $D$. The techniques used in this section are more sophisticated and build upon recent work of O'Grady, \cite{ogrady}. The idea is to use rational curves to construct a (singular) Lagrangian in an appropriate moduli space of vector bundles on a K3 surface.

\section{Brill--Noether theory for smooth curves with a nodal model on a K3 surface} \label{BNP-nodal}
In this section we will apply an argument from \cite{lazarsfeld-bnp} to the K3 surface $S_{p,h}$ as in Lemma \ref{onenodal-lem-a} in order to study the Brill--Noether theory for smooth curves with a primitive nodal model on a K3 surface.
\begin{lem}
Consider the K3 surface $S_{p,h}$ as in Lemma  \ref{onenodal-lem-a}. There is no expression $M=A_1+A_2$, where $A_1$ and $A_2$ are effective divisors with $h^0(Y,\mathcal{O}(A_1)) \geq 2$ and $h^0(Y,\mathcal{O}(A_2)) \geq 2$.
\end{lem}
\begin{proof}
We first claim that any effective divisor of the form $D=aR_1+bR_2$, for integers $a,b$, must have $a,b \geq 0$. Suppose for a contradiction that $a < 0$. Clearly we must have $b > 0$. Thus there is some integral component $D_1$ of $D$ with $(D_1 \cdot R_2) <0$, as $(D \cdot R_2)=-2b <0$. Thus $D_1 \sim R_2$. Repeating this argument on $D-R_2$, we see that $b R_2$ is a summand of $D$. But then $D-bR_2=aR_1$ is effective, which is a contradiction as $a<0$. Thus $a \geq 0$. Likewise $b \geq 0$. Furthermore, this argument also shows that all integral components of any effective divisor of the form $D=aR_1+bR_2$ are linearly equivalent to either $R_1$ or $R_2$. In particular, $D$ is rigid.

Suppose $M=A_1+A_2$ is an expression as above. Write $A_1=x_1M+\sum_{i=1}^2 y_{1,i} R_i$ and $A_2=x_2M+\sum_{i=1}^2 y_{2,i} R_i$ for integers $x_i,y_{i,j}$ for $i=1,2$, $1 \leq j \leq 2$. We have $x_1,x_2 \geq 0$ by Lemma  \ref{onenodal-lem-a} and $x_1+x_2=1$, and assume $x_1 \geq x_2$.  Thus we must have $x_2=0$, which gives  $h^0(Y,\mathcal{O}(A_2)) \leq 1$ (as the divisor $\sum_{i=1}^2 y_{2,i} \Gamma_i $ is rigid if $y_{2,i} \geq 0$ for all $i$ and not effective if there is some $j$ with $y_{2,j} <0$). 
\end{proof}

Let $C \seq X$ be a smooth curve on a K3 surface $X$. Let $M \in \text{Pic}(C)$ be a globally generated line bundle such that $\omega_C \otimes M^*$ is also globally generated. We denote by $F_{C,M}$ the vector bundle on $X$ defined as the kernel
of the evaluation map $H^0(C,M) \otimes_{\C} \mathcal{O}_X \twoheadrightarrow M$. Let $G_{C,M}$ be the dual bundle of $F_{C,M}$, this is globally generated from the exact sequence
$$ 0 \to H^0(M) \otimes \mathcal{O}_X \to G_{C,M} \to \omega_C \otimes M^* \to 0$$
(using $H^1(\mathcal{O}_X)=0$). The following generalisation of \cite[Lemma 1.3]{lazarsfeld-bnp} is well-known, see \cite[Remark 3.1]{fkp-old}.
\begin{lem}
In the above situation, assume further that there is no expression $\mathcal{O}(C) \simeq L_1 \otimes L_2$, where $L_1$ and $L_2$ are effective line
bundles on $X$ with $h^0(L_i) \geq 1+s_i$ for $i=1,2$, where $s_i \geq 1$ are integers satisfying $s_1+s_2=h^0(C,M)$. Then $F_{C,M}$ is a simple vector bundle.
\end{lem}
\begin{proof}
We follow the proof of \cite[Ch.7, Prop.2.2]{huy-lec-k3}. The bundle $F_{C,M}$ is simple if and only if its dual $G_{C,M}$ is simple. Suppose $G_{C,M}$ were not simple. Then there would exist a non-trivial endomorphism $\psi: G \to G$ with non-trivial kernel. Set $K:= \text{im}(\psi)$, $L_1:=\det(K)$ and $L_2:=\det((G/K))$. Set $s_1:= \rank(K)$ and $s_2:=\rank((G/K)/T)$, where $T$ is the maximal torsion subsheaf of $G/K$. Clearly $s_i \geq 1$ for $i=1,2$ and $s_1+s_2=\rank(G)=h^0(C,M)$. So it suffices to prove $h^0(L_i) \geq 1+s_i$ for $i=1,2$. As is explained in \cite[Sec.7, Prop.2.2]{huy-lec-k3}, if we pick a sufficiently positive divisor $D$ on $X$ we have $h^0(D,L_i|_D) \geq s_i +1$ (as $c_1(T)$ is effective). On the other hand, if $D$ is sufficently positive  then $D-L_i$ is big and nef, so that $H^0(X,L_i) \twoheadrightarrow H^0(D,L_i|_D)$ and thus $h^0(L_i) \geq 1+s_i$ for $i=1,2$.
\end{proof}

\begin{cor} \label{BNP-cor}
Consider a K3 surface $S_{p,h}$ as in Lemma  \ref{onenodal-lem-a}. Let $C \in |M|$ be a general smooth curve. Then $C$ is Brill--Noether--Petri general.
\end{cor}
\begin{proof}
This follows immediately from the proof of the main theorem in \cite{lazarsfeld-bnp} and the above lemma.
\end{proof}
Putting all the pieces together, we get the following result.
\begin{prop} \label{BNP-theorem}
Assume $g-n \geq 8$. 
Then there exists an irreducible component $\mathcal{J} \seq \mathcal{V}^n_g$ such that for $[(f:D \to X,M)] \in \mathcal{J}$ general, $D$ is Brill--Noether--Petri general.
\end{prop}
\begin{proof}
Set $h=g-n$, $p=g$. The case $n=0$ is \cite{lazarsfeld-bnp}, so we may assume $p>h$.  Let $l,m$ be the unique nonnegative integers such that $$p-h= \left \lfloor \frac{h+1}{2} \right \rfloor l+m $$
and $0 \leq m < \left \lfloor \frac{h+1}{2} \right \rfloor $. Set $\epsilon=1$ if $m=0$ or $m=\left \lfloor \frac{h+1}{2} \right \rfloor -1$ and $\epsilon=0$ otherwise. Then $(M+R_1+\epsilon R_2)^2=2g-2$, where $M$, $R_1$, $R_2$ are a basis of $P_{p,h}$ as in Lemma  \ref{onenodal-lem-a}. The claim then follows from the proof of Theorem \ref{finiteness}, by deforming to the curve $$R:=D \cup R_1 \cup \epsilon R_2 $$ on $S_{p,h}$, where $D \in |M|$ is general, marked at all nodes other than one point from $D \cap R_i$ for $i=1,2$. Note that the partial normalisation of $R$ at the marked nodes is an unstable curve, and the stabilisation is isomorphic to $D$, which is Brill--Noether--Petri general by Corollary \ref{BNP-cor}.
\end{proof}

\section{Brill--Noether theory for nodal rational curves on K3 surfaces} \label{rat}
In this section we will denote by $X$ a K3 surface with 
$\text{Pic}(X) \simeq \mathbb{Z}L$, $(L)^2=2g-2$ with $g \geq 2$, and $C \in |L|$ will denote a fixed \emph{rational} curve (not necessarily nodal). Let $\bar{J}^d(C)$ denote the compactified Jacobian of degree $d$, rank one torsion free sheaves. For the fundamental theory of the compactified Jacobian of an integral curve on a smooth surface, see \cite{cpt-jac-2}, \cite{cpt-jac-1}, \cite{rego}, \cite{altman-compact}. In particular, $\bar{J}^d(C)$ is a projective, integral, local complete intersection scheme of dimension $g$ containing $\text{Pic}^d(C)$ as a dense open subset. Consider the \emph{generalized Brill--Noether loci}
$$\overline{W}^r_d(C) := \{ \text{$A \in \bar{J}^d(C)$ with $h^0(A) \geq r+1$} \}$$ which can be given a determinantal scheme structure, cf.\ \cite[Pg.\ 10]{gomez}, \cite[\S 2.2]{bhosle-parawes}. There is an open subset $W^r_d(C) \seq \overline{W}^r_d(C) $ parametrizing line bundles. We will denote by $\rho(g,r,d)$ the Brill--Noether number $g-(r+1)(g-d+r)$.

The following comes from the proof of \cite[Remark 2.3(i)]{bhosle-parawes}, although it may have been known to experts earlier (note that if $A \in \bar{J}^d(C)$ for $d>2g-2$, then $h^1(A)=0$, \cite[Prop.\ I.4.6]{cpt-jac-2}). The proof is essentially the same as in the smooth case.
\begin{thm} \label{bhosle-bn}
Each irreducible component of $\overline{W}^r_d(C)$ has dimension at least $\rho(g,r,d)$.
\end{thm}
If $\rho(g,r,d) > 0$, then under our hypotheses $\overline{W}^r_d(C)$ is connected, \cite[Thm.\ 1]{gomez}. If $\rho(g,r,d) \geq 0$, then $\overline{W}^r_d(C)$ is nonempty since we can deform $C$ to a smooth curve in $|L|$ and $h^0$ is upper semicontinuous, cf.\ \cite[\S 3.4]{gomez}.

Let $V_{d,r} \seq \overline{W}_d^r(C) $ be the open locus parametrizing sheaves $A$ which are globally generated and with $h^0(A)=r+1$. Assume $V_{d,r} \neq \emptyset$. We will begin by proving that $\dim V_{d,r} \leq \rho(g,r,d)$ (in particular $V_{d,r}= \emptyset$ if $\rho(g,r,d) < 0$). 

Fix a vector space $\mathbb{H}$ of dimension $r+1$ and let 
$P^r_d \to V_{d,r}$ parametrize pairs $(A,\lambda)$ where $A \in V_{d,r}$ and $\lambda$ is a surjection of $\mathcal{O}_X$ modules
$$ \lambda : \mathbb{H} \otimes \mathcal{O}_X \to A$$ inducing an isomorphism $\mathbb{H} \simeq H^0(A)$. Two such surjections are identified if they differ by multiplication by a nonzero scalar. Thus $P^r_d$ is a PGL(r+1) bundle over $V_{d,r}$.

Let $(A,\lambda) \in P^r_d$. Then $Ker \lambda$ is a vector bundle $F$ of rank $r+1$, \cite[\S3.2]{gomez}. We have $det(F) \simeq L^*$, $deg(c_2(F))=d$, $h^0(F)=h^1(F)=0$ and $h^2(F)=r+1+(g-d+r)$, cf. \cite[\S1]{lazarsfeld-bnp}. Note that $g-d+r =h^1(A) \geq 0$. Further, for any rank one, torsion-free sheaf $A$ on $C$ we may define an `adjoint' $A^{\dagger}$, which is a rank one torsion-free sheaf with $(A^{\dagger})^{\dagger}=A$. From the short exact sequence
$$0 \to F \to  \mathbb{H} \otimes \mathcal{O}_X \to A \to 0$$
we may form the dual sequence
$$0 \to \mathbb{H}^* \otimes \mathcal{O}_X \to F^* \to A^{\dagger} \to 0.$$

The following lemma is a slight generalisation of \cite[Cor.\ 9.3.2]{huy-lec-k3}:
\begin{lem}
Assume $\text{Pic}(X) \simeq \mathbb{Z}L$ as above and let $(A,\lambda) \in P^r_d$. Then the vector bundle $F=Ker \lambda$ is stable.
\end{lem}
\begin{proof}
For any vector bundle $H \seq \mathcal{O}_X^{\oplus a}$ and any integer $1 \leq s \leq rk(H)$, we have $h^0(\bigwedge^s H^*) \geq 1$. Indeed, we have $\bigwedge^s H \seq \bigwedge^s \mathcal{O}_X^{\oplus a} \simeq  \mathcal{O}_X^{\oplus b}$ for some integer $b$ and then $\mathcal{E}nd_{\mathcal{O}_X}(\bigwedge^s H) \simeq \bigwedge^s H \otimes \bigwedge^s H^* \seq (\bigwedge^s H^*)^{\oplus b}$. Taking global sections gives $h^0(\bigwedge^s H^*) \geq 1$ (as $id \in H^0(\mathcal{E}nd_{\mathcal{O}_X}(\bigwedge^s H)))$.

Now let $F' \seq F \seq \mathbb{H} \otimes \mathcal{O}_X $ be a locally free subsheaf  of $F$ with $rk(F')=r'<r+1$. From the above, $h^0(det(F'^*)) \geq 1$ and $h^0(\bigwedge^{r'-1}F'^*) \geq 1$ (if $r' >1$). As $\text{Pic}(X) \simeq \mathbb{Z}L$, we have $det(F')=kL^*$ for some $k \geq 0$. We claim $k \neq 0$. If $k=0$ and $r'=1$, then $F' \simeq \mathcal{O}_X$ which contradicts that $h^0(F)=0$. If $k=0$, $r' >0$, then $F' \simeq \bigwedge^{r'-1}F'^* \otimes det(F')$ gives $h^0(F') \geq 1$ which is again a contradiction. So we have $det(F')=kL^*$ for $k >0$ and $det(F)=L^*$, which implies $\deg(F')/ rk(F') < \deg(F) / rk(F)$ as required.
\end{proof}

Let $M_v$ be the moduli space of stable sheaves on $X$ with Mukai vector $v=(r+1,L,g-d+r)$. We have a morphism 
\begin{align*}
\psi_C: P^r_d &\to M_v \\
(A, \lambda) &\mapsto (Ker \lambda)^*
\end{align*} where $(Ker \lambda)^*$ denotes the dual bundle to $Ker \lambda$.
Let $M_C$ be the closure of the image of $\psi_C$, with the induced reduced scheme structure. By the description of $F$, if $[F^*] \in Im(\psi_C)$, $$c_2(F^*) \sim d c_X$$ where $c_X$ is 
the rational equivalence class of a point lying on a rational curve as defined in \cite{bv-chow}.

There is a natural symplectic form $\alpha$ on $M_v$ defined in \cite{mukai-symplectic}.
\begin{prop}
Let $\alpha$ be the natural symplectic form on $M_v$ and $i: M^0_C \to M_v$ the inclusion, where $M^0_C$ is the smooth locus of $M_C$. Then $i^*\alpha=0$.
\end{prop}
\begin{proof}
Since $g-d+r \geq 0$, \cite[Thm.\ 0.6(1)]{ogrady} applies and for any $[G] \in M_v$, there is an effective, degree $\rho(g,r,d)$ zero-cycle $Z$ with $c_2(G) \sim [Z]+ac_X$ for some $a \in \mathbb{Z}$.  Following \cite[Prop.\ 1.3]{ogrady} there is then a smooth quasi-projective variety $\widetilde{M}_v$ with morphisms $q: \widetilde{M}_v \to M_v$, $p:\widetilde{M}_v \to X^{[\rho(g,r,d)]}$  such that $q$ is surjective and generically finite, and with the property that if $x=[F^*] \in Im(\psi_C)$ and $y \in q^{-1}(x)$, then we have the rational equivalence 
\begin{align} \label{rateq}
p(y)+(d-\rho(g,r,d))c_X \sim c_2(F^*) \sim d c_X.
\end{align} Further if $\beta$ is the symplectic form on the Hilbert scheme of points $X^{[\rho(g,r,d)]}$, we have $q^* \alpha=k p^* \beta$ for some nonzero constant $k \in \C$. 
Let $\widetilde{M}_C \seq \widetilde{M}_v$ denote the closure of $q^{-1}(Im(\psi_C))$ with the induced reduced scheme structure, and let $p_C$ respectively $q_C$ be the restriction of $p$ respectively $q$ to the smooth locus of $\widetilde{M}_C$.  Then $p_C(s)$ is rationally equivalent to $p_C(t)$ for all $s, t$ in the smooth locus of $\widetilde{M}_C$, from (\ref{rateq}). Thus $p_C^*(\beta)=0$ by\cite{mumf-chow}. Hence $q_C^*(\alpha)=0$ and since $q$ is surjective, $i^*\alpha=0$.
\end{proof}

\begin{cor}
We have $\dim M_C \leq \rho(g,r,d)$.
\end{cor}
\begin{proof}
Indeed $\dim M_v=2\rho(g,r,d)$ from \cite[Thm.\ 0.1]{mukai-symplectic} so this follows from the proposition above.
\end{proof}

\begin{cor} \label{tf-bn}
If  $V_{d,r}$ is nonempty then $\dim V_{d,r} = \rho(g,r,d)$.
\end{cor}
\begin{proof}
If $V_{d,r}$ is nonempty then $\dim V_{d,r} \geq \rho(g,r,d)$ by Theorem \ref{bhosle-bn}, so it suffices to show $\dim V_{d,r} \leq \rho(g,r,d)$. It then suffices to show that $\psi_C: P^r_d \to M_v$ has fibres of dimension $\dim\text{PGL(r+1)}$. In other words, we need to show that for each fixed $F \in M_v$ there are only finitely many
$A \in V_{d,r}$ fitting into an exact sequence $0 \to F \to \mathbb{H} \otimes \mathcal{O}_X \to A \to 0$. But this follows immediately from the fact that in our circumstances the degeneracy locus map $Gr(r+1,H^0(F^*)) \to |L|$ is globally defined and finite, see \cite[\S2]{ogrady} (recall that all such $A$ are supported on a fixed $C$ by definition).
\end{proof}
\begin{rem}
Assume $V_{d,r}$ is nonempty. We have $\dim M_C = \dim V_{d,r}=\rho(g,r,d)$ from the above Corollary. Thus $M_C$ is a (possibly singular) Lagrangian subvariety of $M_v$. 
\end{rem}

\begin{cor} \label{bnempty}
Let $X$ be a K3 surface with $\text{Pic}(X) \simeq \mathbb{Z} L$ and $(L \cdot L)=2g-2$. Let $C \in |L|$ be rational and assume $\rho(g,r,d) <0$. Then
$$\overline{W}^r_d(C) := \{ \text{$A \in \bar{J}^d(C)$ with $h^0(A) \geq r+1$} \}$$ is empty.
\end{cor} 
\begin{proof}
Assume for a contradiction that $A \in \overline{W}^r_d(C)$. Let $A'$ be the image of the evaluation morphism $H^0(A) \otimes \mathcal{O}_C \to A$. Then $A'$ is a globally generated, torsion free, rank one sheaf of degree $d' \leq d$ with $r'+1 \geq r+1$ sections, and thus $A' \in V_{d',r'}$. But $\rho(g,r',d') \leq \rho(g,r,d) <0$ for $d' \leq d$, $r' \geq r$ and thus $V_{d',r'}$ is empty by Corollary \ref{tf-bn}. This is a contradiction. 
\end{proof}

To proceed we need two technical lemmas.
\begin{lem} \label{techlem}
Let $C$ be an arbitrary integral \textbf{nodal} curve. Suppose $A'$ is a rank one torsion free sheaf on $C$ and let $k(p)$ be the length one skyscraper sheaf on $C$ supported at a node $p \in C$. Then if $\mathcal{Z} \seq \overline{W}_{d}^r(C)$ is an irreducible family of rank one torsion free sheaves such that we have an exact sequence
$$0 \to A' \to A \to k(p) \to 0 $$
for all $A \in \mathcal{Z}$, then $\dim \mathcal{Z} \leq 1$.
\end{lem}
\begin{proof}
It suffices to show $\dim_{\C} Ext^1_{\mathcal{O}_C}(k(p), A') \leq 2$.  We have
\begin{align*}
Ext^1_{\mathcal{O}_C}(k(p), A') &\simeq Ext^1_{\mathcal{O}_C}(k(p), A'(n)) \text{\; for any $n\in \mathbb{Z}$} \\
&\simeq H^0(C,\mathcal{E}xt^1_{\mathcal{O}_C}(k(p), A'(n))) \text{\; for $n\gg 0$} \\
&\simeq H^0(C,\mathcal{E}xt^1_{\mathcal{O}_C}(k(p), A'))
\end{align*}
where the second line follows from \cite[Prop III.6.9]{har}. The sheaf $\mathcal{E}xt^1_{\mathcal{O}_C}(k(p), A')$ is a skyscraper sheaf supported at $p$. If $A'$ is nonsingular at $p$ then
$$\dim_{\C} Ext^1_{\mathcal{O}_C}(k(p), A') = \dim_{\C} Ext^1_{\mathcal{O}_C}(k(p), \omega_C)=1,$$ by Serre duality. 

Suppose now $A'$ is singular at $p$ and let $\pi: C' \to C$ be the normalisation of $C$. Then
$\dim_{\C} Ext^1_{\mathcal{O}_C}(k(p), A') = \dim_{\C} Ext^1_{\mathcal{O}_C}(k(p), \pi_*(\mathcal{O}_{C'}))$ since $A'_p \simeq m_p \simeq \pi_*(\mathcal{O}_{C'})_p$ where $m_p$ is the maximal ideal of $p$, by \cite[III.1]{cpt-jac-2}\footnote{Note that $m_p$ is a degree $-1$, rank one t.f.\ sheaf, and $\pi_*(\mathcal{O}_{C'})$ is a t.f.\ sheaf of strictly positive degree, so although these sheaves are locally isomorphic, they are \emph{not} globally isomorphic.}. But
$\dim_{\C} Ext^1_{\mathcal{O}_C}(k(p),\pi_*(\mathcal{O}_{C'}))=2$ as required, by \cite[Prop.\ 2.3]{av-lang}. 
\end{proof}

\begin{lem} \label{techlem2}
Let $C$ be an arbitrary integral nodal curve. Suppose $A'$ is a rank one torsion free sheaf and let $Q$ be a sheaf with zero-dimensional support such that $supp(Q) \seq C_{\text{sing}}$, where $C_{\text{sing}}$ is the singular locus of $C$. Then if $\mathcal{Z} \seq \overline{W}_{d}^r(C)$ is an irreducible family of rank one torsion free sheaves such that we have an exact sequence
$$0 \to A' \to A \to Q \to 0 $$
for all $A \in \mathcal{Z}$, then $\dim \mathcal{Z} \leq l(Q)$, where $l(Q)$ denotes the length of $Q$.
\end{lem}
\begin{proof}
We will prove the result by induction on $l(Q)$. When $l(Q)=1$ the result holds from
Lemma \ref{techlem}. Suppose $Q$ has length $r$ and choose a sheaf $Q'$ with zero-dimensional support and length $r-1$ such that we have a surjection $\phi : Q \twoheadrightarrow Q'$. For any $$0 \to A' \to A \to Q \to 0, $$ $\phi$ then induces
a short exact sequence
$$0 \to A'' \to A \to Q' \to 0$$
where $A''$ fits into the exact sequence
$$0 \to A' \to A'' \to Ker(\phi) \to 0.$$ Now let $\pi: \mathcal{T} \to \mathcal{Z}$ be the moduli space with fibre over $A \in \mathcal{Z}$ parametrising all extensions
$0 \to A' \to A \to Q \to 0$; this can be constructed from \cite[$\S  4.1$]{bridgeland}. After replacing $\mathcal{T}$ with an open set we have a morphism $\psi : \mathcal{T} \to \overline{W}_{d'}^{r'}(C)$ for some $d'$, $r'$, which sends a point representing the exact sequence $0 \to A' \to A \to Q \to 0$ to $A'':=Ker(A \to Q')$. By Lemma \ref{techlem} the image of $\psi$ is at most one dimensional as $l(Ker(\phi))=1$. Further, $\pi(\psi^{-1}(A''))$ is at most $l(Q')=l(Q)-1$ dimensional for any $A'' \in Im(\psi)$ by the induction hypothesis. It then follows than $\dim \mathcal{Z} \leq l(Q)$ as required.
\end{proof}

\begin{lem} \label{zero-case}
Let $C$ be an integral, nodal curve. Then $$\overline{W}^0_d(C) := \{ \text{$A \in \bar{J}^d(C)$ with $h^0(A) \geq 1$} \}$$ is irreducible of dimension $\rho(g,0,d)=d$.
\end{lem}
\begin{proof}
Let $U_d \seq \overline{W}^0_d(C)$ be the open subset consisting of line bundles. Let $V_d:=Div_d(C)$ denote the 
scheme parametrizing zero-dimensional schemes $Z \seq C$ such that the ideal sheaf $I_Z$ is invertible of degree $d$; i.e\ $V_d$ is the scheme of effective Cartier divisors. Let $\tilde{C} \to C$ denote the normalisation. From \cite[Thm.\ 2.4]{kleiman-special}, pullback induces a birational morphism $V_d \to Div_d(\tilde{C})$, and thus $\dim V_d=d$. We have a morphism $V_{d} \to \bar{J}^d(C)$ with image $U_d$, which sends a scheme $Z$ to the effective line bundle $I_Z^*$. Thus $U_d$ of dimension at most $d$. Since each component of $\overline{W}^0_d(C)$ has dimension at least $d$ by Theorem \ref{bhosle-bn}, we see $\dim(U_d) =d$.

Let $I$ be an irreducible component of $\overline{W}^0_d(C) \setminus U_d$; we need to prove $\dim(I) < d$. There is a nonempty open set $I^0$ of $I$, an integer $d' < d$ and a partial normalisation $\mu: C' \to C$ such that for each $A \in I^0$ there exists a unique effective line bundle $B \in \text{Pic}^{d'}(C')$ with $\mu_*(B) \simeq A$ by \cite[Prop.\ 3.4]{gagne}. Since the dimension of the moduli space of effective line bundles of degree $d'$ on $C'$ has dimension $d'$ by the above, we see that $\dim(I) \leq d' < d$.
\end{proof}

We now prove the main result of this section.
\begin{thm} \label{bn-rat-thm}
Let $X$ be a K3 surface with $\text{Pic}(X) \simeq \mathbb{Z} L$ and $(L \cdot L)=2g-2$. Suppose $C \in |L|$ is a rational, nodal curve. Then
$$\overline{W}^r_d(C) := \{ \text{$A \in \bar{J}^d(C)$ with $h^0(A) \geq r+1$} \}$$ is either empty or is equidimensional of the expected dimension $\rho(g,r,d)$.
\end{thm} 
\begin{proof}
By Corollary \ref{bnempty} the theorem holds whenever $\rho(g,r,d) <0$. Thus it suffices to prove the theorem for $\rho(g,r,d)  \geq -1$. We will proceed by induction on $\rho(g,r,d)$ starting from the case $\rho(g,r,d)  = -1$.

Choose nonnegative integers $r,d$ and suppose the claim holds for all $r',d'$ such that $\rho(g,r',d') < \rho(g,r,d)$. We know the claim holds for $r=0$ by Lemma \ref{zero-case}, so we may suppose $r >0$.   Let $I$ be an irreducible component of $\overline{W}^r_d(C) \setminus V_{d,r}$; from Theorem \ref{bhosle-bn} it suffices to show $\dim(I) < \rho(g,r,d)$. For all $A \in I$, we denote by $A'$ the globally generated part of $A$, i.e.\ the image of the evaluation morphism $H^0(A) \otimes \mathcal{O}_C \to A$. There is an open dense subset $I^0 \seq I$ such that $deg(A')=d'$, $h^0(A')=r'$  is constant for all $A \in I^0$. Replacing $I^0$ by a smaller open set if necessary, we have a morphism 
\begin{align*}
f \;: \; I^0_{\text{red}} &\to \overline{W}^{r'}_{d'}(C) \\
A & \mapsto A'.
\end{align*}
Indeed, let $S$ be an integral, locally Noetherian scheme over $\C$, let $\pi: C \times S  \to S$ be the projection, and let $\mathcal{A} $ be an $S$ flat family of rank one torsion free sheaves $\mathcal{A}_s$ on $C$, with $\deg(\mathcal{A}_s)=d$, $h^0(\mathcal{A}_s)=r'+1$ constant. Replacing $S$ with an open subset, we may assume that $\pi_* \mathcal{A}$ is a trivial vector bundle of rank $r'+1$, and that the image $\mathcal{A}'$ of the evaluation morphism $$H^0(\mathcal{A}) \otimes \mathcal{O}_{C \times S} \to \mathcal{A} $$ is flat over $S$. Replacing $S$ with another open subset, we may further assume $\pi_* \mathcal{A}'$ is a trivial vector bundle of rank $r'+1$. We claim that $\mathcal{A}'_s$ is the base-point free part of $\mathcal{A}_s$.  Let $B_s$ denote the base point free part of $\mathcal{A}_s$. The surjection
$H^0(\mathcal{A}_s) \otimes \mathcal{O}_{C} \to \mathcal{A}'_s$ shows $\mathcal{A}'_s \seq B_s \seq A$. Then the exact sequence
$$ 0 \to \mathcal{A}'_s \to B_s \to F \to 0,$$
where $F$ has zero-dimensional support, and the equality $h^0(\mathcal{A}'_s)=h^0(B_s)=r+1$ implies $F$ is the zero sheaf (since $B_s$ is base point free). Thus if $d':= \deg(\mathcal{A}'_s)$, $\mathcal{A}'$ is a flat family of rank one, torsion free sheaves on $C$ of degree $d'$ with $r'+1$ sections, so the universal property of  $\overline{W}^{r'}_{d'}(C)$ induces a morphism $S \to \overline{W}^{r'}_{d'}(C)$.

We  next claim that $f$ has fibres of dimension at most $d-d'$. This will then imply the result as $\rho(g,r',d') \leq \rho(g,r,d')= \rho(g,r,d)- (r+1)(d-d')$ so that $\dim(I^0) < \rho(g,r,d)$ for $r \neq 0$. For any $A \in I^0$ we have a sequence
$$0 \to f(A) \to A \to Q_A \to 0$$
where $Q_A$ has zero-dimensional support. We have a canonical decomposition
$Q_A = Q_{A,\text{sm}} \oplus Q_{A,\text{sing}}$ with $\text{Supp}(Q_{A,\text{sm}}) \seq C_{\text{sm}}$ and $\text{Supp}(Q_{A,\text{sing}}) \seq C_{\text{sing}}$,
where $C_{\text{sm}}$ is the smooth locus of $C$ and $C_{\text{sing}}=C-C_{\text{sm}}$. Replacing $I^0$ with a dense open set we may assume $Q':=Q_{A,\text{sing}}$ is independent of $A \in I^0$. Let $e:=l(Q')$. For any $A \in I^0$, there is a unique effective line bundle $M$ of degree $d-d'-e$ such that we have a short exact sequence
$$0 \to f(A)(M) \to A \to Q' \to 0 .$$
We have a morphism
\begin{align*}
g \;: \; I^0 &\to \overline{W}^{r'}_{d-e}(C) \\
A & \mapsto f(A)(M).
\end{align*}
By Lemma \ref{techlem2}, $g$ has fibres of dimension at most $e$. For any $A'$ in the image of $f$ consider
$$ g|_{f^{-1}(A')} \;: \; f^{-1}(A') \to \overline{W}^{r'}_{d-e}(C) .$$ The image of
$ g|_{f^{-1}(A')} $ is a subset of the space of tuples $A' \otimes M$ for $M \in \text{Pic}^{d-d'-e}(C)$ effective. The moduli space of effective line bundles in $\text{Pic}^{d-d'-e}(C)$ may be identified with the image of the natural map $C^{[d-d'-e]}_{\text{sm}} \to \text{Pic}^{d-d'-e}(C)$, where $C_{\text{sm}}$ is the smooth locus of $C$, and thus has dimension at most $d-d'-e$. Thus $\dim f^{-1}(A')\leq d-d'$ as required.

\end{proof}
\begin{rem}
It is clear from the proof that the theorem would hold for any constant cycle curve $C \in |L|$ such that $C$ is integral and nodal, see \cite{huy-const-cyc}.
\end{rem}

\bibliography{biblio}
\end{document}